\documentclass[12pt]{article} 
\usepackage{amssymb}
\usepackage{amsmath,bm}
\usepackage{array}
\usepackage{graphics}
\usepackage{color}
\usepackage{xcolor}
\usepackage{xypic}
\usepackage{tikz}
\usepackage[utf8]{inputenc}                    
\usepackage{microtype}                         
\usepackage{mathtools, amsthm, amssymb, eucal} 
\usepackage[normalem]{ulem}                    
\usepackage{mdframed}
\usepackage{verbatim}
\usepackage{booktabs} 
\usepackage{mathrsfs}
\usepackage{titling}
\usepackage{xypic}
\usepackage[paper=letterpaper, margin=1in, headsep=20pt]{geometry}
\usepackage{hyperref}
\usepackage{enumerate}
\usepackage{multicol}
    \theoremstyle{plain}
    \newtheorem{thm}{Theorem}[section]
    \newtheorem{lem}[thm]{Lemma}
    \newtheorem{prop}[thm]{Proposition}
    \newtheorem{cor}[thm]{Corollary}

    \theoremstyle{definition}
    \newtheorem{defn}{Definition}[section]
    
    \newtheorem{exmp}{Example}[section]

    \theoremstyle{remark}
    \newtheorem{rem}{Remark}[section]

\definecolor{kiyoshi}{RGB}{0,0,200}

\definecolor{ray}{RGB}{200,0,0}

\input prepictex   \input pictex    \input postpictex

\def\mput #1 #2/{\put{$\bullet_{#1}$} [tl] <-1mm,1mm> at #2}

\begin{document}

\centerline{\textbf{\Large On Clusters and Exceptional Sets in Types $\mathbb{A}$ and $\tilde{\mathbb{A}}$}}
\medskip
\centerline{Kiyoshi Igusa and Ray Maresca} 

\begin{abstract}
\noindent
In this paper we first study clusters in type $\tilde{\mathbb{A}}$ by collecting them into a finite number of infinite families given by Dehn twists of their corresponding triangulations, and show that these families are counted by the Catalan numbers. We also highlight the similarities and differences between the annuli diagrams used to study clusters and those used to study exceptional sets in type $\tilde{\mathbb{A}}$. We then focus on exceptional collections (sets) of modules over path algebras of quivers by first showing that the notion of relative projectivity in exceptional sets is well defined. We finish by counting the number of exceptional sets of representations of type $\mathbb{A}$ quivers with straight orientation and using this to count the number of families of exceptional sets of type $\tilde{\mathbb{A}}$ with straight orientation.
\end{abstract}

\section{Introduction}

\indent 

Clusters and cluster algebras have been of much interest since the early 2000's by work of Fomin and Zelevinsky. In [\ref{ref: clusters are in bijection with triangulations}], it was shown that clusters in type $\tilde{\mathbb{A}}_n$ are in bijection with triangulations of an annulus associated to a quiver of type $\tilde{\mathbb{A}}_n$, of which there are infinitely many. Using the methods introduced in [\ref{ref: Maresca}], these triangulations, and hence clusters in type $\tilde{\mathbb{A}}_n$, can be placed into families where we move from one member of the family to the next by $2\pi$ clockwise or counterclockwise Dehn twists about the inner circle of the annulus. Our first result, Theorem \ref{thm: counting clusters}, gives how many families of such clusters there are for the straightly orientated quiver of type $\tilde{\mathbb{A}}_n$.\\

In [\ref{ref: Master's Thesis Clusters and Triangulations}], the bijection between clusters and triangulations in type $\tilde{\mathbb{A}}_n$ is given explicitly, which we will repeat in this paper for convenience and will be called the cluster convention. In [\ref{ref: Maresca}], a different convention for assigning arcs on an annulus to exceptional modules is given, which we call the exceptional convention. In order for such a diagram to form an exceptional collection, no arcs in the diagram can cross or form cycles; however, most triangulations of the annulus contain cycles. On the surface, this seems like a contradiction since any cluster gives an exceptional collection by a result in [\ref{ref: Clusters form exceptional collections}]; however, our second result, Theorem \ref{thm: when the conventions coincide}, describes precisely when the two conventions coincide. \\

In Theorem \ref{thm: rel projective is independent of order}, we prove that relative projectivity and injectivity are well defined in exceptional sets of modules over any hereditary algebra. We then focus our attention to exceptional collections of modules over $\Bbbk Q$ where $Q$ is a quiver of type $\mathbb{A}_n$ with straight orientation. In particular, in Section \ref{sec: counting exceptional sets} we show that these exceptional sets are counted by generalized Catalan numbers and are in bijection with ternary trees and lattice paths in $\mathbb{R}^2$. Moreover, we classify precisely when modules are relatively projective and injective in the exceptional set. We then focus on quivers of type $\tilde{\mathbb{A}}_n$ with straight orientation. In this case we show that the number of families of exceptional collections is counted by a different collection of generalized Catalan numbers. We then provide an explicit bijection between these families of exceptional collections and lattice paths in $\mathbb{R}^2$. We conclude the paper with a section of explicit examples of both families of clusters and exceptional collections for $Q$ a quiver of type $\tilde{\mathbb{A}}_3$ with straight orientation.

\section{Preliminaries}
\indent 

Let $\Bbbk$ be a field. Let $Q=(Q_0,Q_1,s,t)$ denote a quiver with vertex set $Q_0$, arrow set $Q_1$ and functions $s,t:Q_1 \rightarrow Q_0$ that assign to each arrow a starting and terminal point respectively. Denote by $\Bbbk Q$ the path algebra of $Q$. This is a hereditary algebra when $Q$ is acyclic. Finitely generated $\Bbbk Q$-modules are finite dimensional representations of $Q$. These form the category $mod$-$\Bbbk Q$. An object $V$ in this category is \textbf{partial tilting} if Ext$_{\Bbbk Q}^1 (V,V) = 0$ and it is $\textbf{tilting}$ if in addition the number of indecomposable summands of $V$ equals the number of vertices in the quiver. Let $D^b(\Bbbk Q)$ denote the bounded derived category of $mod$-$\Bbbk Q$. It is known that $D^b(\Bbbk Q)$ is a triangulated category with shift functor $[i]:D^b(\Bbbk Q) \rightarrow D^b(\Bbbk Q)$ and almost split triangles induced by almost split sequences in mod $\Bbbk Q$. Let $\tau: D^b(\Bbbk Q) \rightarrow D^b(\Bbbk Q)$ denote the equivalence which induces the Auslander--Reiten translation so that $\tau C = A$ if we have a triangle of the form $A\rightarrow B \rightarrow C \rightarrow A[1]$. The \textbf{cluster category} $\mathcal{C}_Q$ is the orbit category $D^b(\Bbbk Q)/F$ where $F$ is the auto-equivalence given by $F = \tau^{-1}[1]:D^b(\Bbbk Q) \rightarrow D^b(\Bbbk Q)$. A module $T \in$  Ob$(\mathcal{C}_Q)$ is called a \textbf{cluster} if it is a tilting object in $\mathcal{C}_Q$; that is, Ext$_{\mathcal{C}_Q}^1 (T,T) = 0$ and the number of indecomposable summands of $T$ is $|Q_0|$. For more on tilting theory and general representations of quivers see [\ref{ref: blue book}] and [\ref{ref: Schiffler Quiver Reps}]. For more on cluster categories see [\ref{ref: Reiten Cluster Categories}]\\ 

We call $V$ \textbf{exceptional} if End$_{\Bbbk Q}(V)$ is a division ring and Ext$_{\Bbbk Q}^i (V,V) = 0$ for all $i\geq 1$. When $\Bbbk Q$ is hereditary, the first exceptionality condition implies that the representation $V$ is indecomposable and the second reduces to Ext$_{\Bbbk Q}^1 (V,V) =$ Ext$(V,V) = 0$ as higher extension groups vanish. An \textbf{exceptional sequence} $\xi = (V_1, \dots, V_k)$ is a sequence of exceptional representations of $Q$ such that Hom$(V_i,V_j) = 0 =$ Ext$(V_i,V_j)$ for all $j < i$. An \textbf{exceptional set}, or \textbf{exceptional collection} is a \textit{set} of representations $\bar{\xi}=\{V_1, \dots, V_k\}$ such that the $V_j$ can be ordered in such a way that they make an exceptional sequence. It is well known that if $(V_1, \dots, V_k)$ is an exceptional sequence, then $k \leq n$ [\ref{ref: Bill Exceptional Sequences}]. When $k = n$, we call $\xi$ $(\bar{\xi})$ a \textbf{complete} exceptional sequence (collection). From now on, we adopt the convention that `exceptional sequence (collection)' means complete. For more on exceptional sequences see [\ref{ref: Bill Exceptional Sequences}], [\ref{ref: Braid Group Action on Exceptional Sequences}], and [\ref{ref: Dynkin Exceptional Sequences}].

\subsection{Quivers of Type $\mathbb{A}_n$ and $\mathbb{\tilde{A}}_n$} 

\indent

 To make a Dynkin graph of type $\mathbb{A}_n$ a quiver, we define an \textbf{orientation vector} $\bm{\varepsilon} = (\varepsilon_0, \dots , \varepsilon_n) \in \{-,+\}^{n+1}$. Then \textbf{the quiver of type $\mathbb{A}_n$} with this orientation, denoted by $Q^{\bm{\varepsilon}} = (Q_0^{\bm{\varepsilon}},Q_1^{\bm{\varepsilon}},s,t)$, is the one such that $Q_0^{\bm{\varepsilon}} = \{1,2,\dots,n\}$ and for each $\alpha_i \in Q_1^{\bm{\varepsilon}}$ with $1\leq i \leq n-1$, 
\vspace{-.6cm}
\begin{center}
\begin{displaymath}
   \alpha_i = \left\{
     \begin{array}{lr}
       i \rightarrow i+1 & : \varepsilon_i = +\\
       i \leftarrow i+1 & :  \varepsilon_i = -
     \end{array}
   \right.
\end{displaymath}
\end{center}

Note that neither $\varepsilon_0$ nor $\varepsilon_n$ affect $Q^{\bm{\varepsilon}}$. Moreover, note that any quiver of type $\mathbb{A}_n$ is given in this way. A quiver of type $\mathbb{A}_n$ is said to have \textbf{straight orientation} if $\varepsilon_i = +$ for all $i \in \{1, 2, \dots , n-1\}$. Moreover, we assume the vertices are labeled from left to right in natural numerical order: 
\[
	\mathbb{A}_n:\quad 1\to 2\to \cdots\to n
\]
An indecomposable representation of a linear quiver of type $\mathbb{A}_n$, so $\varepsilon = +$ or $-$ for all $i$, is given by its support which is a closed interval $[a, b]$ where $1 \leq a \leq b \leq n$. For $\mathbb A_n$ with straight orientation, this representation is denoted $M_{a-1,b}$. Thus $M_{x,y}$, for $0\le x<y\le n$, has top the simple module at $x+1$ and socle or bottom the simple module at $y$.

\indent

 To make a graph of type $\tilde{\mathbb{A}}_n$ a quiver, we define an \textbf{orientation vector} $\bm{\varepsilon} = (\varepsilon_0, \dots , \varepsilon_n) \in \{-,+\}^{n+1}$. Then \textbf{the quiver of type $\tilde{\mathbb{A}}_n$} with this orientation, denoted by $Q^{\bm{\varepsilon}} = (Q_0^{\bm{\varepsilon}},Q_1^{\bm{\varepsilon}},s,t)$, is the one such that $Q_0^{\bm{\varepsilon}} = \{1,2,\dots,n,n+1\}$ and for each $\alpha_i \in Q_1^{\bm{\varepsilon}}$ with $0\leq i \leq n-1$, we define $\alpha_i \in Q^{\varepsilon}_1$ as

\vspace{-.7cm}

\begin{center}
\begin{multicols}{2}

 \begin{displaymath}
   \alpha_i = \left\{
     \begin{array}{lr}
       i \rightarrow i+1 & : \varepsilon_i = +\\
       i \leftarrow i+1 & :  \varepsilon_i = -
     \end{array}
   \right.
\end{displaymath}

\columnbreak

 \begin{displaymath}
   \alpha_0 = \left\{
     \begin{array}{lr}
       n+1 \rightarrow 1 & : \varepsilon_0 = +\\
       n+1 \leftarrow 1 & :  \varepsilon_0 = -
     \end{array}
   \right.
\end{displaymath}

\end{multicols}
\end{center}

Note that any quiver of type $\tilde{\mathbb{A}}_n$ is given in this way. Moreover, note that when considering our vertex set modulo $|Q_0| = n+1$, the convention for $\alpha_0$ is consistent with that of $\alpha_i$. So long as $\varepsilon_i \neq \varepsilon_j$ for some $i$ and $j$, these quivers are both hereditary and tame. We say a quiver $Q$ of type $\tilde{\mathbb{A}}_n$ has \textbf{straight orientation} if $\varepsilon_i = +$ for all $i \in \{1, 2, \dots , n\}$ and $\varepsilon_0 = -$. Moreover, we assume the vertices are labeled from left to right in natural numerical order:
\[
\xymatrix{
\tilde{\mathbb{A}}_n:&1 \ar[r]\ar@/^1.5pc/[rrrr]&2\ar[r]&\quad \cdots\quad  \ar[r]& n\ar[r] & n+1
	}
\]
By \textbf{tame}, we mean there are infinitely many indecomposable $\Bbbk Q$-modules and for all $n\in\mathbb{N}$, all but finitely many isomorphism classes of $n$-dimensional indecomposables occur in a finite number of one-parameter families. It is known that the module category, hence the Auslander--Reiten quiver $\Gamma_{\Bbbk Q}$, of a tame hereditary algebra can be partitioned into three sections: the preprojective, regular and preinjective components. For $\mathbb{\tilde{A}}_n$ quivers, the regular component consists of the left, right and homogeneous tubes. Moreover, the path algebras are also string algebras, simplifying the aforementioned tripartite classification of $\Gamma_{\Bbbk Q}$. For more on representation theory of tame algebras and definitions of these components see [\ref{ref: blue book}]. For more on string algebras see [\ref{ref: String Algebra Info}]. \\

It is well known that the indecomposable modules over string algebras are either string or band modules. For a quiver $Q$ whose path algebra is a string algebra, to define string modules, we first define for $a\in Q_1$ a \textbf{formal inverse} $a^{-1}$, such that $s(a^{-1}) = t(a)$ and $t(a^{-1}) = s(a)$. Let $Q_1^{-1}$ denote the set of formal inverses of arrows in $Q_1$. We define a \textbf{walk} as a sequence $w = w_0\cdots w_r$ such that for all $i\in\{1,\dots,r\}$, we have $t(w_{i-1}) = s(w_{i})$ where $w_i \in Q_1 \cup Q_1^{-1}$. We also consider walks of length 0 at a vertex $i$, given by the lazy path $e_i$. A \textbf{string} is a walk $w$ without consecutive arrows $aa^{-1}$ or $a^{-1}a$. A \textbf{band} $b = b_1 \cdots b_n$ is a cyclic string, that is, $t(b_n) = s(b_1)$. For quivers of type $\mathbb{\tilde{A}}_n$, we take the convention that all named strings move in the counterclockwise direction around $Q$. Moreover, the band modules lie in the homogeneous tubes and we can classify in which component of the Auslander--Reiten quiver the string modules reside by their shape, as we will soon see. We denote by $ij_k$ the string module of length $k$ associated to the walk $w_1\cdots w_{k-1}$ where $s(w_1)=i+1$ and $t(w_{k-1})=j$. For $k=1$, $(j-1)j_1$ denotes the simple module at vertex $j$. Note that beginning at the previous vertex is not the usual convention, however it agrees with our other notation, such as $M_{ij}$ for the module with support $[i+1,j]$. When drawing the strings, we take the convention that the head of each arrow is at the bottom as seen in Example \ref{exam: example of string module}. String modules are uniquely determined up to isomorphism by their walk $w$. So, we denote them by their string $w$. Band modules, which are in fact string modules, are not uniquely determined by their walk, but we nevertheless denote them by their walks.\\

\begin{exmp}   \label{exam: example of string module}

Let $Q$ be the following quiver.

\[
\xymatrix{
1 \ar_{\alpha_1}[r] \ar@/^2pc/^{\alpha_4}[rrr] & 2 \ar_{\alpha_2}[r] & 3 \ar_{\alpha_3}[r] & 4 
	}
\]

Consider the walk $\alpha_3\alpha_4^{-1}\alpha_1\alpha_2\alpha_3$. Then the string module associated to this walk is $24_6$, and we draw its graph as follows.

\[
\xymatrix{
& & & & 1 \ar[dddl] \ar[dr] & & & \\ & 2 \ar[dr] & &  & & 2 \ar[dr] & & \\ & & 3 \ar[dr] & & & & 3\ar[dr] & \\ & & & 4 & & & & 4
	}
\]

The representation associated to this string is the following.

\[
\xymatrix{
\Bbbk \ar_{1}[r] \ar@/^2pc/^{\begin{bmatrix} 0 \\ 1\end{bmatrix}}[rrr] & \Bbbk \ar_{\begin{bmatrix} 1 \\ 0\end{bmatrix}}[r] & \Bbbk^2 \ar_{\begin{bmatrix} 1 & 0 \\ 0 & 1\end{bmatrix}}[r] & \Bbbk^2 
	}
\]

\end{exmp}

Let $w = w_0\cdots w_r$ be an indecomposable $\Bbbk \mathbb{\tilde{A}}_n$-module. It is well known that $w$ is \textbf{preprojective} if and only if there are arrows $\alpha,\beta\in Q_1$ such that $t(\alpha) = s(w_0)$ and $t(\beta) = t(w_r)$. Similarly, $w$ is \textbf{preinjective} if and only if there are arrows $\alpha,\beta\in Q_1$ such that $s(\alpha) = s(w_0)$ and $s(\beta) = t(w_r)$, $w$ is \textbf{left regular} if and only if there are arrows $\alpha,\beta\in Q_1$ such that $t(\alpha) = s(w_0)$ and $s(\beta) = t(w_r)$, $w$ is \textbf{right regular} if and only if there are arrows $\alpha,\beta\in Q_1$ such that $s(\alpha) = s(w_0)$ and $t(\beta) = t(w_r)$ and finally $w$ is \textbf{homogeneous} if and only if $w$ is a band. For instance, the $34_6$ string is a preprojective module over the path algebra of the quiver in Example \ref{exam: example of string module}.

\subsection{Cluster Combinatorics in type $\tilde{\mathbb{A}}_n$}\label{sec: cluster combinatorics}

\indent 

It is known that clusters in type $\tilde{\mathbb{A}}_n$ are in bijection with triangulations of an annulus associated to the quiver $Q$ [\ref{ref: clusters are in bijection with triangulations}]. In this section, we will provide this bijection in the case when $Q$ is of type $\tilde{\mathbb{A}}_n$ with straight orientation. A detailed exposition of the general case is given in [\ref{ref: Master's Thesis Clusters and Triangulations}]. \\

To a quiver $Q^{\bm{\varepsilon}}$ of type $\tilde{\mathbb{A}}_n$, we associate an annulus $A_{Q^{\bm{\varepsilon}}}$ as follows. If $\varepsilon_i = +(-)$, then $i$ is a marked point  on the outer (inner) circle of the annulus. We moreover write the vertices in clockwise order respecting the natural numerical order of the vertices. We adopt the convention that $i \in \{0,1,\dots,n\}$ where we identify $n+1$ and $0$. Since $Q^{\bm\varepsilon}$ has straight orientation in this subsection, it has $n$ arrows pointed in the positive direction and only one pointed in the negative direction. The associated annulus has $n$ marked points on the outer circle labeled $1$ through $n$ and one marked point on the inner circle labeled $0$. \\

Recall that the universal cover of the annulus $\mathbb R/n\mathbb Z\times [0,1]$ is $\mathbb{R}\times [0,1]$. Consider the $n$ marked points $i=1,\dots,n$ on the outer boundary of the annulus. Then for each $i$, we have a set of marked points $(i-.5 +n\mathbb{Z})\times \{1\}$ on the upper boundary of the cover, $\mathbb{R} \times \{1\}$, where the point $(i-.5+kn,1)$ is labeled ``$i+nk$''. The single marked point 0 on the inner boundary of the annulus lifts to the set of marked points $n\mathbb{Z}\times \{0\}$ on the lower boundary of the cover, $\mathbb{R} \times \{0\}$ where the point $(kn,0)$ is marked ``$k$''. An example of an annulus and its universal cover can be seen in Figure \ref{fig:example of cover}. We now fix a triangulation of the cover consisting of the strands $X_i$ which are straight line segments connecting the marked points $(i-1,1)$ and $(n,0)=$``1'' for $i = 1, 2, \dots, n+1$. Thus $X_{n+1}$ is the vertical line segment connecting $(n,1)$ and $(n,0)$. Note that we take these strands modulo $n+1$ and the strand $X_{n+1}$ occurs twice in the fundamental lift of the associated annulus as in Figure \ref{fig:example of cover}. As done in [\ref{ref: Master's Thesis Clusters and Triangulations}], we call the $X_i$ \textbf{vertical strands}. \\

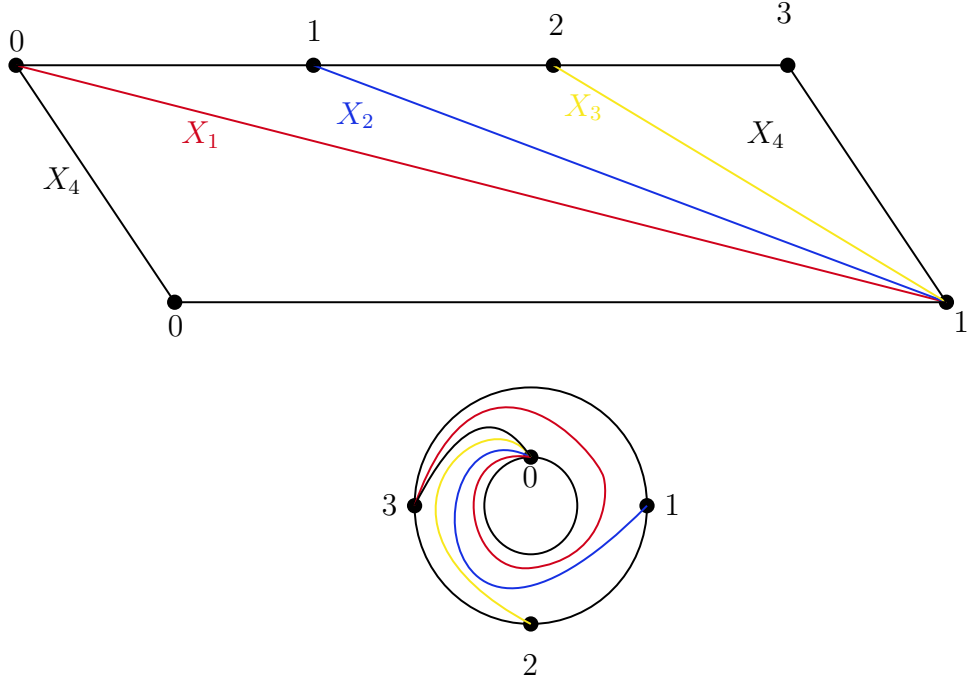
\begin{figure}[h!]
    \centering

\tikzset{every picture/.style={line width=0.75pt}} 

\begin{tikzpicture}[x=0.75pt,y=0.75pt,yscale=-1,xscale=1]

\draw    (150,170) ;
\draw [shift={(150,170)}, rotate = 0] [color={rgb, 255:red, 0; green, 0; blue, 0 }  ][fill={rgb, 255:red, 0; green, 0; blue, 0 }  ][line width=0.75]      (0, 0) circle [x radius= 3.35, y radius= 3.35]   ;
\draw    (539.22,170) ;
\draw [shift={(539.22,170)}, rotate = 0] [color={rgb, 255:red, 0; green, 0; blue, 0 }  ][fill={rgb, 255:red, 0; green, 0; blue, 0 }  ][line width=0.75]      (0, 0) circle [x radius= 3.35, y radius= 3.35]   ;
\draw    (70,50.39) ;
\draw [shift={(70,50.39)}, rotate = 0] [color={rgb, 255:red, 0; green, 0; blue, 0 }  ][fill={rgb, 255:red, 0; green, 0; blue, 0 }  ][line width=0.75]      (0, 0) circle [x radius= 3.35, y radius= 3.35]   ;
\draw    (220,50.39) ;
\draw [shift={(220,50.39)}, rotate = 0] [color={rgb, 255:red, 0; green, 0; blue, 0 }  ][fill={rgb, 255:red, 0; green, 0; blue, 0 }  ][line width=0.75]      (0, 0) circle [x radius= 3.35, y radius= 3.35]   ;
\draw    (341,50.39) ;
\draw [shift={(341,50.39)}, rotate = 0] [color={rgb, 255:red, 0; green, 0; blue, 0 }  ][fill={rgb, 255:red, 0; green, 0; blue, 0 }  ][line width=0.75]      (0, 0) circle [x radius= 3.35, y radius= 3.35]   ;
\draw    (459.22,50.39) ;
\draw [shift={(459.22,50.39)}, rotate = 0] [color={rgb, 255:red, 0; green, 0; blue, 0 }  ][fill={rgb, 255:red, 0; green, 0; blue, 0 }  ][line width=0.75]      (0, 0) circle [x radius= 3.35, y radius= 3.35]   ;
\draw [color={rgb, 255:red, 208; green, 2; blue, 27 }  ,draw opacity=1 ]   (70,50.39) -- (539.22,170) ;
\draw [color={rgb, 255:red, 22; green, 48; blue, 226 }  ,draw opacity=1 ]   (220,50.39) -- (539.22,170) ;
\draw [color={rgb, 255:red, 248; green, 231; blue, 28 }  ,draw opacity=1 ]   (341,50.39) -- (539.22,170) ;
\draw   (306.17,272.7) .. controls (306.17,259.15) and (316.66,248.17) .. (329.61,248.17) .. controls (342.56,248.17) and (353.06,259.15) .. (353.06,272.7) .. controls (353.06,286.24) and (342.56,297.23) .. (329.61,297.23) .. controls (316.66,297.23) and (306.17,286.24) .. (306.17,272.7)(271,272.7) .. controls (271,239.73) and (297.24,213) .. (329.61,213) .. controls (361.98,213) and (388.22,239.73) .. (388.22,272.7) .. controls (388.22,305.67) and (361.98,332.39) .. (329.61,332.39) .. controls (297.24,332.39) and (271,305.67) .. (271,272.7) ;
\draw    (329.61,248.17) ;
\draw [shift={(329.61,248.17)}, rotate = 0] [color={rgb, 255:red, 0; green, 0; blue, 0 }  ][fill={rgb, 255:red, 0; green, 0; blue, 0 }  ][line width=0.75]      (0, 0) circle [x radius= 3.35, y radius= 3.35]   ;
\draw    (388.22,272.7) ;
\draw [shift={(388.22,272.7)}, rotate = 0] [color={rgb, 255:red, 0; green, 0; blue, 0 }  ][fill={rgb, 255:red, 0; green, 0; blue, 0 }  ][line width=0.75]      (0, 0) circle [x radius= 3.35, y radius= 3.35]   ;
\draw    (329.61,332.39) ;
\draw [shift={(329.61,332.39)}, rotate = 0] [color={rgb, 255:red, 0; green, 0; blue, 0 }  ][fill={rgb, 255:red, 0; green, 0; blue, 0 }  ][line width=0.75]      (0, 0) circle [x radius= 3.35, y radius= 3.35]   ;
\draw    (271,272.7) ;
\draw [shift={(271,272.7)}, rotate = 0] [color={rgb, 255:red, 0; green, 0; blue, 0 }  ][fill={rgb, 255:red, 0; green, 0; blue, 0 }  ][line width=0.75]      (0, 0) circle [x radius= 3.35, y radius= 3.35]   ;
\draw [color={rgb, 255:red, 22; green, 48; blue, 226 }  ,draw opacity=1 ]   (329.61,248.17) .. controls (270.35,218.39) and (272.35,390.39) .. (388.22,272.7) ;
\draw [color={rgb, 255:red, 248; green, 231; blue, 28 }  ,draw opacity=1 ]   (329.61,248.17) .. controls (306.35,214.39) and (234.35,282.39) .. (329.61,332.39) ;
\draw [color={rgb, 255:red, 208; green, 2; blue, 27 }  ,draw opacity=1 ]   (271,272.7) .. controls (304.35,177.39) and (363.35,246.39) .. (366.35,258.39) .. controls (369.35,270.39) and (364.84,299.84) .. (330.03,304.11) .. controls (295.22,308.39) and (287.22,241.39) .. (329.61,248.17) ;
\draw    (70,50.39) -- (150,170) ;
\draw    (459.22,50.39) -- (539.22,170) ;
\draw    (70,50.39) -- (459.22,50.39) ;
\draw    (150,170) -- (539.22,170) ;
\draw    (271,272.7) .. controls (284.35,246.39) and (306.35,213.39) .. (329.61,248.17) ;

\draw (65,31.4) node [anchor=north west][inner sep=0.75pt]    {$0$};
\draw (145,175.4) node [anchor=north west][inner sep=0.75pt]    {$0$};
\draw (215,26.4) node [anchor=north west][inner sep=0.75pt]    {$1$};
\draw (337,23.4) node [anchor=north west][inner sep=0.75pt]    {$2$};
\draw (452,17.4) node [anchor=north west][inner sep=0.75pt]    {$3$};
\draw (541.22,173.4) node [anchor=north west][inner sep=0.75pt]    {$1$};
\draw (82,100.4) node [anchor=north west][inner sep=0.75pt]    {$X_{4}$};
\draw (437,77.4) node [anchor=north west][inner sep=0.75pt]    {$X_{4}$};
\draw (152,77.4) node [anchor=north west][inner sep=0.75pt]  [color={rgb, 255:red, 208; green, 2; blue, 27 }  ,opacity=1 ]  {$X_{1}$};
\draw (230,67.4) node [anchor=north west][inner sep=0.75pt]  [color={rgb, 255:red, 22; green, 48; blue, 226 }  ,opacity=1 ]  {$X_{2}$};
\draw (345,63.4) node [anchor=north west][inner sep=0.75pt]  [color={rgb, 255:red, 248; green, 231; blue, 28 }  ,opacity=1 ]  {$X_{3}$};
\draw (324,251.4) node [anchor=north west][inner sep=0.75pt]    {$0$};
\draw (395.61,265.57) node [anchor=north west][inner sep=0.75pt]    {$1$};
\draw (324,346.4) node [anchor=north west][inner sep=0.75pt]    {$2$};
\draw (253,265.4) node [anchor=north west][inner sep=0.75pt]    {$3$};

\end{tikzpicture}

    \caption{A fundamental domain for the universal cover of the annulus associated to straightly oriented $Q^{\bm\varepsilon}$ of type $\tilde{\mathbb{A}}_3$ from Example \ref{exam: example of string module}, with the triangulation by vertical strands.}
    \label{fig:example of cover}
\end{figure}

\begin{defn}
Let $i,j\in\{0,1,\dots , n-1\}$ be such that $i\neq j$. An \textbf{arc} $a(i,j)[\lambda]$ on $A_{Q^{\bm{\varepsilon}}}$ is an isotopy class of simple curves in $A_{Q^{\bm{\varepsilon}}}$ where any $\gamma \in a(i,j)[\lambda]$ satisfies: 
\begin{enumerate}
\item $\gamma$ begins at $i$ and ends at $j$, 
\item $\gamma$ travels clockwise through the interior of the annulus from $i$ to $j$. 
\item If $\gamma$ begins and ends on the same boundary component, $\gamma$ is called an \textbf{exterior or boundary} arc and the integer $\lambda$ is the winding number of $\gamma$ about the inner boundary component. If $\gamma$ connects the two boundary components, $\gamma$ is called a \textbf{bridging arc} and $\lambda$ is the clockwise winding number of $\gamma$ about the inner boundary circle of $A_{Q^{\bm{\varepsilon}}}$ when traversing $\gamma$ beginning from the outer boundary component. If an arc from $i$ to $j$ has counter clockwise winding number $k$, we write $a(i,j)[-k]$.
\end{enumerate}
A collection of such arcs will be called an $\textbf{arc diagram}$.
\end{defn}

\begin{exmp}\label{exmp: naming of arcs}
Let $Q^{\bm\varepsilon}$ of type $\tilde{\mathbb{A}}_3$ from Example \ref{exam: example of string module} and consider the arc diagram in Figure \ref{fig:example of cover}. Then the red arc is $a(3,0)[1]$, the blue arc is $a(1,0)[0]$, the yellow arc is $a(2,0)[0]$, and the black arc is $a(3,0)[0]$. Below are two more arc diagrams for the same quiver with the appropriate labels.
\begin{center}

\tikzset{every picture/.style={line width=0.75pt}} 

\begin{tikzpicture}[x=0.75pt,y=0.75pt,yscale=-1,xscale=1]

\draw   (106.57,124.6) .. controls (106.57,106.57) and (121.4,91.96) .. (139.7,91.96) .. controls (158,91.96) and (172.84,106.57) .. (172.84,124.6) .. controls (172.84,142.62) and (158,157.23) .. (139.7,157.23) .. controls (121.4,157.23) and (106.57,142.62) .. (106.57,124.6)(57.61,124.6) .. controls (57.61,79.53) and (94.37,43) .. (139.7,43) .. controls (185.04,43) and (221.79,79.53) .. (221.79,124.6) .. controls (221.79,169.66) and (185.04,206.19) .. (139.7,206.19) .. controls (94.37,206.19) and (57.61,169.66) .. (57.61,124.6) ;
\draw    (139.7,91.07) ;
\draw [shift={(139.7,91.07)}, rotate = 0] [color={rgb, 255:red, 0; green, 0; blue, 0 }  ][fill={rgb, 255:red, 0; green, 0; blue, 0 }  ][line width=0.75]      (0, 0) circle [x radius= 3.35, y radius= 3.35]   ;
\draw    (221.79,124.6) ;
\draw [shift={(221.79,124.6)}, rotate = 0] [color={rgb, 255:red, 0; green, 0; blue, 0 }  ][fill={rgb, 255:red, 0; green, 0; blue, 0 }  ][line width=0.75]      (0, 0) circle [x radius= 3.35, y radius= 3.35]   ;
\draw    (139.7,206.19) ;
\draw [shift={(139.7,206.19)}, rotate = 0] [color={rgb, 255:red, 0; green, 0; blue, 0 }  ][fill={rgb, 255:red, 0; green, 0; blue, 0 }  ][line width=0.75]      (0, 0) circle [x radius= 3.35, y radius= 3.35]   ;
\draw    (57.61,124.6) ;
\draw [shift={(57.61,124.6)}, rotate = 0] [color={rgb, 255:red, 0; green, 0; blue, 0 }  ][fill={rgb, 255:red, 0; green, 0; blue, 0 }  ][line width=0.75]      (0, 0) circle [x radius= 3.35, y radius= 3.35]   ;
\draw   (396.57,124.6) .. controls (396.57,106.57) and (411.4,91.96) .. (429.7,91.96) .. controls (448,91.96) and (462.84,106.57) .. (462.84,124.6) .. controls (462.84,142.62) and (448,157.23) .. (429.7,157.23) .. controls (411.4,157.23) and (396.57,142.62) .. (396.57,124.6)(347.61,124.6) .. controls (347.61,79.53) and (384.37,43) .. (429.7,43) .. controls (475.04,43) and (511.79,79.53) .. (511.79,124.6) .. controls (511.79,169.66) and (475.04,206.19) .. (429.7,206.19) .. controls (384.37,206.19) and (347.61,169.66) .. (347.61,124.6) ;
\draw    (429.7,91.07) ;
\draw [shift={(429.7,91.07)}, rotate = 0] [color={rgb, 255:red, 0; green, 0; blue, 0 }  ][fill={rgb, 255:red, 0; green, 0; blue, 0 }  ][line width=0.75]      (0, 0) circle [x radius= 3.35, y radius= 3.35]   ;
\draw    (511.79,124.6) ;
\draw [shift={(511.79,124.6)}, rotate = 0] [color={rgb, 255:red, 0; green, 0; blue, 0 }  ][fill={rgb, 255:red, 0; green, 0; blue, 0 }  ][line width=0.75]      (0, 0) circle [x radius= 3.35, y radius= 3.35]   ;
\draw    (429.7,206.19) ;
\draw [shift={(429.7,206.19)}, rotate = 0] [color={rgb, 255:red, 0; green, 0; blue, 0 }  ][fill={rgb, 255:red, 0; green, 0; blue, 0 }  ][line width=0.75]      (0, 0) circle [x radius= 3.35, y radius= 3.35]   ;
\draw    (347.61,124.6) ;
\draw [shift={(347.61,124.6)}, rotate = 0] [color={rgb, 255:red, 0; green, 0; blue, 0 }  ][fill={rgb, 255:red, 0; green, 0; blue, 0 }  ][line width=0.75]      (0, 0) circle [x radius= 3.35, y radius= 3.35]   ;
\draw [color={rgb, 255:red, 248; green, 231; blue, 28 }  ,draw opacity=1 ]   (429.7,91.07) .. controls (447.35,83.39) and (511.35,111.39) .. (429.7,205.3) ;
\draw    (139.7,91.96) .. controls (212.35,73.39) and (193.35,242.39) .. (57.61,124.6) ;
\draw [color={rgb, 255:red, 208; green, 2; blue, 27 }  ,draw opacity=1 ]   (429.7,91.07) .. controls (466.77,63.27) and (494.09,123.34) .. (480.39,147.41) .. controls (466.69,171.49) and (447.4,176.73) .. (427.42,177.43) .. controls (407.44,178.13) and (382.95,165.92) .. (374.65,146.16) .. controls (366.35,126.39) and (368.4,111.29) .. (374.36,97.84) .. controls (380.33,84.39) and (442.29,33.44) .. (488.32,96.41) .. controls (534.35,159.39) and (454.35,191.66) .. (428.85,193.53) .. controls (403.35,195.39) and (372.35,173.39) .. (347.61,124.6) ;
\draw [color={rgb, 255:red, 22; green, 48; blue, 226 }  ,draw opacity=1 ]   (139.7,91.96) .. controls (179.7,61.96) and (210.35,106.39) .. (221.79,124.6) ;

\draw (134.25,98.27) node [anchor=north west][inner sep=0.75pt]    {$0$};
\draw (234.55,117.64) node [anchor=north west][inner sep=0.75pt]    {$1$};
\draw (134.25,216.12) node [anchor=north west][inner sep=0.75pt]    {$2$};
\draw (34.8,117.41) node [anchor=north west][inner sep=0.75pt]    {$3$};
\draw (424.25,98.27) node [anchor=north west][inner sep=0.75pt]    {$0$};
\draw (524.55,117.64) node [anchor=north west][inner sep=0.75pt]    {$1$};
\draw (424.25,215.12) node [anchor=north west][inner sep=0.75pt]    {$2$};
\draw (324.8,117.41) node [anchor=north west][inner sep=0.75pt]    {$3$};
\draw (517,147.4) node [anchor=north west][inner sep=0.75pt]  [color={rgb, 255:red, 248; green, 231; blue, 28 }  ,opacity=1 ]  {$a( 0,2)[ 0]$};
\draw (82,166.4) node [anchor=north west][inner sep=0.75pt]  [color={rgb, 255:red, 0; green, 0; blue, 0 }  ,opacity=1 ]  {$a( 0,3)[ 0]$};
\draw (507,74.4) node [anchor=north west][inner sep=0.75pt]  [color={rgb, 255:red, 208; green, 2; blue, 27 }  ,opacity=1 ]  {$a( 0,3)[ -1]$};
\draw (128,59.4) node [anchor=north west][inner sep=0.75pt]  [color={rgb, 255:red, 22; green, 48; blue, 226 }  ,opacity=1 ]  {$a( 0,1)[ 0]$};

\end{tikzpicture}

\end{center}
\end{exmp}

We say that two arcs $a(i_1,j_1)[\lambda_1]$ and $a(i_2,j_2)[\lambda_2]$ \textbf{intersect nontrivially} if any two curves $\gamma_1 \in a(i_{1},j_{1})[\lambda_1]$ and $\gamma_2 \in a(i_2,j_2)[\lambda_2]$ intersect in their interiors, as in the annulus on the right of Example \ref{exmp: naming of arcs}. Otherwise we say that $a(i_1,j_1)[\lambda_1]$ and $a(i_2,j_2)[\lambda_2]$ \textbf{do not intersect nontrivially}. Now consider any arc $a(i,j)[\lambda]$ on the annulus which does not self intersect nontrivially. This arc lifts to a strand in the universal cover which either:
\begin{enumerate}
    \item intersects some finite subset of vertical strands,
    \item is one of the $n+1$ vertical strands: $X_1,X_2,\dots, X_{n+1}$, or
    \item is one of the $n+1$ \textbf{boundary strands}; that is, the arc on the annulus is of the form $a(1,2),\dots,a(n-1,n)$, $a(n,1)$ or $a(0,0)$.
\end{enumerate} We will now define a map $\varphi$ that associates to each strand in the universal cover of $A_{Q^{\bm\varepsilon}}$, a $\Bbbk Q$ string module as follows. If it is a vertical strand or boundary strand, $\varphi$ assigns to it the $0$ module. If it is not, suppose that when ordered from left to right the subset of vertical strands the lifted arc intersects is $\{X_i, X_{i+1}, \dots, X_j\}$ where subscripts are taken modulo $n+1$, where instead of $X_0$, we write $X_{n+1}$. Then to this lifted arc, $\varphi$ assigns the string module $(i-1) j_{j-i+1}$. This convention of associating modules and arcs will be called the \textbf{cluster convention}. As was shown in [\ref{ref: Master's Thesis Clusters and Triangulations}], $\varphi$ forms a bijection between non-vertical, non-boundary strands on the annulus and indecomposable $\Bbbk Q$-modules.\\

A \textbf{triangulation} of the annulus $A_{Q^{\bm{\varepsilon}}}$ is a maximal collection of arcs with no nontrivial intersection such that each arc has no nontrivial self-intersections, no intersection with the set of marked points or the boundary of the annulus except at its endpoints, and is not contractible to a subset of the boundary of the annulus, i.e., is not a boundary arc. Note that a collection of arcs is maximal if and only if it contains $|Q^{\bm{\varepsilon}}_0|$ arcs. Using the aforementioned bijection, we can now get a cluster from any triangulation of the annulus and any triangulation of the annulus from a cluster as follows. An arbitrary triangulation of the annulus will consist of $l\ge0$ vertical strands, say $\{X_{j_1}, X_{j_2}, \dots, X_{j_l}\}$, and $n+1 - l$ non-vertical strands, say $\{Y_{l+1}, Y_{l+2}, \dots, Y_{n+1}\}$. The cluster corresponding to this triangulation is $\varphi(Y_{l+1}) \oplus \dots \oplus \varphi(Y_{n+1}) \oplus P_{j_1}[1]\ \oplus \dots \oplus P_{j_l}[1]$ where $P_i[1]$ is the shifted projective at vertex $i$. As shown in [\ref{ref: Master's Thesis Clusters and Triangulations}], this association is a bijection. 

\begin{figure}[h!] \label{fig: triangle/cluser bijection}

\begin{center}

\tikzset{every picture/.style={line width=0.75pt}} 

\begin{tikzpicture}[x=0.75pt,y=0.75pt,yscale=-1,xscale=1]

\draw    (176,160) ;
\draw [shift={(176,160)}, rotate = 0] [color={rgb, 255:red, 0; green, 0; blue, 0 }  ][fill={rgb, 255:red, 0; green, 0; blue, 0 }  ][line width=0.75]      (0, 0) circle [x radius= 3.35, y radius= 3.35]   ;
\draw    (565.22,160) ;
\draw [shift={(565.22,160)}, rotate = 0] [color={rgb, 255:red, 0; green, 0; blue, 0 }  ][fill={rgb, 255:red, 0; green, 0; blue, 0 }  ][line width=0.75]      (0, 0) circle [x radius= 3.35, y radius= 3.35]   ;
\draw    (96,40.39) ;
\draw [shift={(96,40.39)}, rotate = 0] [color={rgb, 255:red, 0; green, 0; blue, 0 }  ][fill={rgb, 255:red, 0; green, 0; blue, 0 }  ][line width=0.75]      (0, 0) circle [x radius= 3.35, y radius= 3.35]   ;
\draw    (246,40.39) ;
\draw [shift={(246,40.39)}, rotate = 0] [color={rgb, 255:red, 0; green, 0; blue, 0 }  ][fill={rgb, 255:red, 0; green, 0; blue, 0 }  ][line width=0.75]      (0, 0) circle [x radius= 3.35, y radius= 3.35]   ;
\draw    (367,40.39) ;
\draw [shift={(367,40.39)}, rotate = 0] [color={rgb, 255:red, 0; green, 0; blue, 0 }  ][fill={rgb, 255:red, 0; green, 0; blue, 0 }  ][line width=0.75]      (0, 0) circle [x radius= 3.35, y radius= 3.35]   ;
\draw    (485.22,40.39) ;
\draw [shift={(485.22,40.39)}, rotate = 0] [color={rgb, 255:red, 0; green, 0; blue, 0 }  ][fill={rgb, 255:red, 0; green, 0; blue, 0 }  ][line width=0.75]      (0, 0) circle [x radius= 3.35, y radius= 3.35]   ;
\draw [color={rgb, 255:red, 0; green, 0; blue, 0 }  ,draw opacity=1 ]   (96,40.39) -- (384.95,114.05) -- (565.22,160) ;
\draw [color={rgb, 255:red, 144; green, 19; blue, 254 }  ,draw opacity=1 ]   (246,40.39) -- (565.22,160) ;
\draw [color={rgb, 255:red, 0; green, 0; blue, 0 }  ,draw opacity=1 ]   (367,40.39) -- (565.22,160) ;
\draw   (103.58,346.18) .. controls (103.58,330.87) and (115.06,318.46) .. (129.22,318.46) .. controls (143.38,318.46) and (154.86,330.87) .. (154.86,346.18) .. controls (154.86,361.49) and (143.38,373.9) .. (129.22,373.9) .. controls (115.06,373.9) and (103.58,361.49) .. (103.58,346.18)(65.13,346.18) .. controls (65.13,309.63) and (93.82,280) .. (129.22,280) .. controls (164.62,280) and (193.32,309.63) .. (193.32,346.18) .. controls (193.32,382.73) and (164.62,412.35) .. (129.22,412.35) .. controls (93.82,412.35) and (65.13,382.73) .. (65.13,346.18) ;
\draw    (129.22,318.98) ;
\draw [shift={(129.22,318.98)}, rotate = 0] [color={rgb, 255:red, 0; green, 0; blue, 0 }  ][fill={rgb, 255:red, 0; green, 0; blue, 0 }  ][line width=0.75]      (0, 0) circle [x radius= 3.35, y radius= 3.35]   ;
\draw    (193.32,346.18) ;
\draw [shift={(193.32,346.18)}, rotate = 0] [color={rgb, 255:red, 0; green, 0; blue, 0 }  ][fill={rgb, 255:red, 0; green, 0; blue, 0 }  ][line width=0.75]      (0, 0) circle [x radius= 3.35, y radius= 3.35]   ;
\draw    (129.22,412.35) ;
\draw [shift={(129.22,412.35)}, rotate = 0] [color={rgb, 255:red, 0; green, 0; blue, 0 }  ][fill={rgb, 255:red, 0; green, 0; blue, 0 }  ][line width=0.75]      (0, 0) circle [x radius= 3.35, y radius= 3.35]   ;
\draw    (65.13,346.18) ;
\draw [shift={(65.13,346.18)}, rotate = 0] [color={rgb, 255:red, 0; green, 0; blue, 0 }  ][fill={rgb, 255:red, 0; green, 0; blue, 0 }  ][line width=0.75]      (0, 0) circle [x radius= 3.35, y radius= 3.35]   ;
\draw [color={rgb, 255:red, 144; green, 19; blue, 254 }  ,draw opacity=1 ]   (129.22,318.98) .. controls (64.42,285.98) and (81.91,447.83) .. (193.32,346.18) ;
\draw [color={rgb, 255:red, 248; green, 231; blue, 28 }  ,draw opacity=1 ]   (65.13,346.18) .. controls (69.89,356.93) and (131.35,434.39) .. (193.32,346.18) ;
\draw [color={rgb, 255:red, 144; green, 19; blue, 254 }  ,draw opacity=1 ]   (96,40.39) -- (176,160) ;
\draw [color={rgb, 255:red, 144; green, 19; blue, 254 }  ,draw opacity=1 ]   (485.22,40.39) -- (565.22,160) ;
\draw    (96,40.39) -- (485.22,40.39) ;
\draw    (176,160) -- (565.22,160) ;
\draw [color={rgb, 255:red, 144; green, 19; blue, 254 }  ,draw opacity=1 ]   (65.13,346.18) .. controls (79.73,317.02) and (103.79,280.44) .. (129.22,318.98) ;
\draw [color={rgb, 255:red, 144; green, 19; blue, 254 }  ,draw opacity=1 ]   (129.22,318.98) .. controls (155.18,295.96) and (177.05,330.32) .. (193.32,346.18) ;
\draw [color={rgb, 255:red, 248; green, 231; blue, 28 }  ,draw opacity=1 ]   (246,40.39) .. controls (339.35,63.39) and (427.35,69.39) .. (485.22,40.39) ;
\draw [color={rgb, 255:red, 144; green, 19; blue, 254 }  ,draw opacity=1 ]   (246,40.39) -- (176,160) ;
\draw    (168.56,277.8) -- (246.14,175.98) ;
\draw [shift={(247.35,174.39)}, rotate = 127.3] [color={rgb, 255:red, 0; green, 0; blue, 0 }  ][line width=0.75]    (10.93,-3.29) .. controls (6.95,-1.4) and (3.31,-0.3) .. (0,0) .. controls (3.31,0.3) and (6.95,1.4) .. (10.93,3.29)   ;
\draw [shift={(167.35,279.39)}, rotate = 307.3] [color={rgb, 255:red, 0; green, 0; blue, 0 }  ][line width=0.75]    (10.93,-3.29) .. controls (6.95,-1.4) and (3.31,-0.3) .. (0,0) .. controls (3.31,0.3) and (6.95,1.4) .. (10.93,3.29)   ;
\draw    (502.13,283.81) -- (418.57,174.98) ;
\draw [shift={(417.35,173.39)}, rotate = 52.48] [color={rgb, 255:red, 0; green, 0; blue, 0 }  ][line width=0.75]    (10.93,-3.29) .. controls (6.95,-1.4) and (3.31,-0.3) .. (0,0) .. controls (3.31,0.3) and (6.95,1.4) .. (10.93,3.29)   ;
\draw [shift={(503.35,285.39)}, rotate = 232.48] [color={rgb, 255:red, 0; green, 0; blue, 0 }  ][line width=0.75]    (10.93,-3.29) .. controls (6.95,-1.4) and (3.31,-0.3) .. (0,0) .. controls (3.31,0.3) and (6.95,1.4) .. (10.93,3.29)   ;
\draw    (442.35,325.39) -- (219.35,325.39) ;
\draw [shift={(217.35,325.39)}, rotate = 360] [color={rgb, 255:red, 0; green, 0; blue, 0 }  ][line width=0.75]    (10.93,-3.29) .. controls (6.95,-1.4) and (3.31,-0.3) .. (0,0) .. controls (3.31,0.3) and (6.95,1.4) .. (10.93,3.29)   ;
\draw [shift={(444.35,325.39)}, rotate = 180] [color={rgb, 255:red, 0; green, 0; blue, 0 }  ][line width=0.75]    (10.93,-3.29) .. controls (6.95,-1.4) and (3.31,-0.3) .. (0,0) .. controls (3.31,0.3) and (6.95,1.4) .. (10.93,3.29)   ;

\draw (91,21.4) node [anchor=north west][inner sep=0.75pt]    {$0$};
\draw (171,165.4) node [anchor=north west][inner sep=0.75pt]    {$0$};
\draw (241,16.4) node [anchor=north west][inner sep=0.75pt]    {$1$};
\draw (363,13.4) node [anchor=north west][inner sep=0.75pt]    {$2$};
\draw (478,7.4) node [anchor=north west][inner sep=0.75pt]    {$3$};
\draw (567.22,163.4) node [anchor=north west][inner sep=0.75pt]    {$1$};
\draw (123.65,323.39) node [anchor=north west][inner sep=0.75pt]    {$0$};
\draw (201.96,339.1) node [anchor=north west][inner sep=0.75pt]    {$1$};
\draw (123.65,428.71) node [anchor=north west][inner sep=0.75pt]    {$2$};
\draw (46.01,338.91) node [anchor=north west][inner sep=0.75pt]    {$3$};
\draw (104,388.4) node [anchor=north west][inner sep=0.75pt]  [font=\small,color={rgb, 255:red, 248; green, 231; blue, 28 }  ,opacity=1 ]  {$S_{3}$};
\draw (164,302.4) node [anchor=north west][inner sep=0.75pt]  [font=\small,color={rgb, 255:red, 144; green, 19; blue, 254 }  ,opacity=1 ]  {$I_{1}$};
\draw (367,61.4) node [anchor=north west][inner sep=0.75pt]  [font=\small,color={rgb, 255:red, 248; green, 231; blue, 28 }  ,opacity=1 ]  {$S_{3}$};
\draw (213,103.6) node [anchor=north west][inner sep=0.75pt]  [font=\small,color={rgb, 255:red, 144; green, 19; blue, 254 }  ,opacity=1 ]  {$I_{1}$};
\draw (455,315.4) node [anchor=north west][inner sep=0.75pt]    {$I_{1} \oplus S_{3} \oplus P_{2}[ 1] \oplus P_{4}[ 1]$};

\end{tikzpicture}

\end{center}

\caption{On the bottom left, we see a triangulation of the annulus $A_{Q^{\bm\varepsilon}}$ where $Q^{\bm\varepsilon}$ is the quiver from Example \ref{exam: example of string module}. In the universal cover, the original triangulation giving rise to the quiver $Q^{\bm\varepsilon}$ from Figure \ref{fig:example of cover} is given by the black strands, while the colored strands correspond to the lifted arcs in the triangulation on the bottom left. This new triangulation gives the cluster on the bottom right.}

\label{fig: Example of cluster triangulation bijection}

\end{figure}
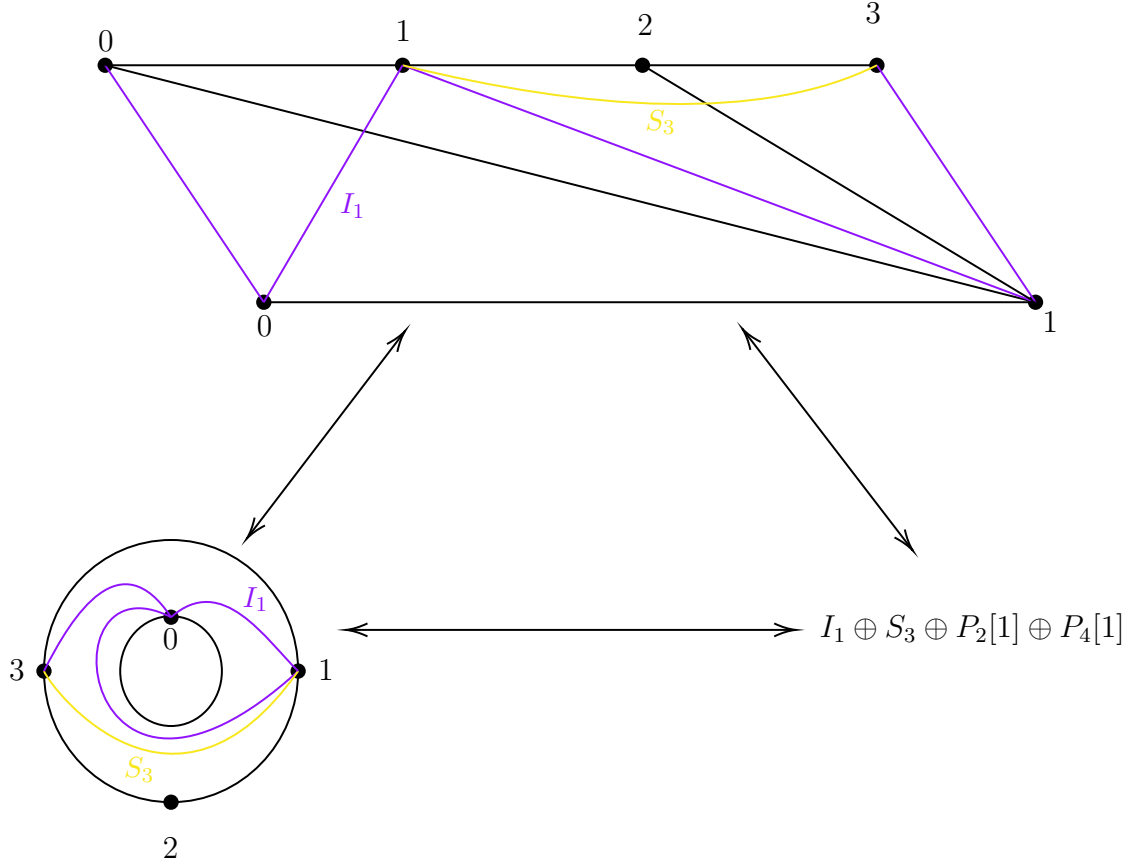

\subsection{Strand and Arc Diagrams For Exceptional Collections}

\indent

The following can be found in [\ref{ref: Maresca}] and [\ref{ref: combinatorics of exceptional sequences type A}], but we repeat it here for convenience. Throughout this subsection, let $Q^{\bm{\varepsilon}}$ be a quiver of type $\tilde{\mathbb{A}}_{n-1}$. Let $S_{n,\tilde{\varepsilon}} := \{\dots, (x_{-1},0), (x_0,0), (x_1,0),$ $ \dots\}\subset\mathbb{R}^2$ be a collection of points arranged on the $x$-axis from left to right together with the function $\tilde{\varepsilon} : S_{n,\tilde{\varepsilon}} \rightarrow \{+,-\}$ sending $(x_i,0) \mapsto \varepsilon_i$. The reason for decorating this set with a subscript $n$ is because we will be taking these vertices modulo $n$ once we associate these diagrams to an $\tilde{\mathbb{A}}_{n-1}$ quiver. We will label the point $(x_{i},0)$ by $i$ as in Figure \ref{fig: Example of Strand and Arc Diagrams}. The \textbf{fundamental domain} of $S_{n,\tilde{\varepsilon}}$, denoted by $FD$, is the collection $\{0,1,\dots,n\}$.

\begin{defn}
Let $i,j\in\mathbb{Z}$ be such that $i\neq j$. A \textbf{strand} $c(i,j) = c(j,i)$ on $S_{n,\tilde{\varepsilon}}$ is an isotopy class of simple curves in $\mathbb{R}^2$ where any $\gamma \in c(i,j)$ satisfies: 
\begin{enumerate}
\item the endpoints of $\gamma$ are $(x_i,0)$ and $(x_j,0)$, 
\item as a subset of $\mathbb{R}^2$, $\gamma \subset \{(x,y)\in \mathbb{R}^2 : x_{\text{min}(i,j)} \leq x \leq x_{\text{max}(i,j)}\} \backslash \{(x_k,0) : {\text{min}(i,j)} < k < {\text{max}(i,j)}\}$, 
\item if $\text{min}(i,j) \leq k \leq  \text{max}(i,j)$ and $\tilde{\varepsilon}_k = +$ (resp., $\tilde{\varepsilon}_k = -$), then $\gamma$ is locally below (resp., locally above) $(x_k,0)$.
\end{enumerate}
\end{defn}  

By locally below (locally above)$(x_k,0)$ we mean that for a parameterization of $\gamma = (\gamma^{(1)}, \gamma^{(2)}):[0,1]\rightarrow \mathbb{R}^2$, there exists a $\delta \in \mathbb{R}$ with $0 < \delta < {1\over 2} \text{min}\{|x_k-x_{k-1}|,|x_k-x_{k+1}|\}$ such that $\gamma$ satisfies $\gamma^{(2)}(t) < 0$ if $\tilde{\varepsilon}_k = +$ (resp., $\gamma^{(2)}(t) >0 $ if $\tilde{\varepsilon}_k = -$) for all $t \in (0,1)$ where $\gamma^{(1)}(t) \in (x_k - \delta, x_k + \delta)$.

\begin{defn} \label{defn: Fundamental lift of string mods}
Consider the string module $ij_k$. We define the \textbf{fundamental lift} of $ij_k$ to $S_{n,\tilde{\varepsilon}}$ as follows. 
\begin{itemize}
\item If $ij_k$ is preprojective or left regular, then the fundamental lift is the strand $c(\tilde{i},\tilde{j})$ with $\tilde{i} = i \in FD$ and ends at $\tilde{j} = i + k$.
\item If $ij_k$ is preinjective or right regular, then the fundamental lift is the strand $c(\tilde{i},\tilde{j})$ that begins at $\tilde{i} = j - k$ and ends at $\tilde{j} = j \in FD$.
\end{itemize}
\end{defn}

Note that there is an injection $\Phi_{\tilde{\varepsilon}}$ from the string modules in ind(rep$_{\Bbbk}(Q^{\bm\varepsilon}))$, the category of indecomposable $\Bbbk$ representations of $Q^{\bm\varepsilon}$, and the set of strands on $S_{n,\tilde{\varepsilon}}$ given by  $\Phi_{\tilde{\varepsilon}}(ij_k) := c(\tilde{i},\tilde{j})$ where $\tilde{i}$ and $\tilde{j}$ are as in Definition \ref{defn: Fundamental lift of string mods}. 

Note that in Definition \ref{defn: Fundamental lift of string mods}, $c(\tilde{i},\tilde{j})$ is one convention of specifying a fundamental lift of $\Bbbk Q^{\bm\varepsilon}$ modules to $S_{n,\tilde{\varepsilon}}$. We call a strand $c(i,j) \in S_{n,\tilde{\varepsilon}}$ a \textbf{fundamental strand} if $c(i,j) \in \text{im}(\Phi_{\tilde{\varepsilon}})$. Note that any strand $c(i,j)$ can be represented by a \textbf{monotone curve} $\gamma \in c(i,j)$. That is, there exists a curve $\gamma \in c(i,j)$ with a parameterization $\gamma = (\gamma^{(1)},\gamma^{(2)}):[0,1] \rightarrow \mathbb{R}^2$ such that if $t,s\in [0,1]$ and $t<s$, then $\gamma^{(1)}(t) < \gamma^{(1)}(s)$.  \\

For the following definitions we fix some $n \in \mathbb{N}$, again keeping in mind that we will associate these strands to a quiver $Q^{\bm{\varepsilon}}$ where $|Q_0| = n$. We say that two strands $c(i_1,j_1)$ and $c(i_2,j_2)$ \textbf{intersect nontrivially} if there exists $z\in\mathbb{Z}$ such that any two curves $\gamma_1 \in c(i_{1} + nz,j_{1} + nz)$ and $\gamma_2 \in c(i_2,j_2)$ intersect in their interiors. Otherwise we say that $c(i_1,j_1)$ and $c(i_2,j_2)$ \textbf{do not intersect nontrivially}. We say a strand $c(i,j)$ \textbf{self intersects} if there exists $z\in\mathbb{Z}$ such that any two curves $\gamma_1 \in c(i + nz,j + nz)$ and $\gamma_2 \in c(i,j)$ intersect in their interiors. If $c(i_1,j_1)$ and $c(i_2,j_2)$ do not intersect nontrivially, we say $c(i_2,j_2)$ is \textbf{clockwise} from $c(i_1,j_1)$ (or equivalently $c(i_1,j_1)$ is \textbf{counterclockwise} from $c(i_2,j_2)$) if and only if for any $z \in \mathbb{Z}$ such that there exists $\gamma_1 \in c(i_{1} + zn,j_{1} + zn)$ and $\gamma_2 \in c(i_2,j_2)$ that share an endpoint $(x_k,0)$ and do not intersect in their interiors, we have that $\gamma_1$ and $\gamma_2$ locally appear in one of the six configurations in Figure \ref{fig: Locally Clockwise} preserving the property that $\gamma_1 \in c(i_{1} + zn,j_{1} + zn)$ and $\gamma_2 \in c(i_2,j_2)$. We say that $c(i_2,j_2)$ is \textbf{locally clockwise} from $c(i_1,j_1)$ if there exists $z \in \mathbb{Z}$ such that some $\gamma_1 \in c(i_{1} + zn,j_{1} + zn)$ and $\gamma_2 \in c(i_2,j_2)$ share an endpoint $(x_k,0)$, do not intersect in their interiors, and $\gamma_1$ and $\gamma_2$ locally appear in one of the six configurations in Figure \ref{fig: Locally Clockwise} preserving the property that $\gamma_1 \in c(i_{1} + zn,j_{1} + zn)$ and $\gamma_2 \in c(i_2,j_2)$. It is possible for $c(i_2,j_2)$ to be locally clockwise from $c(i_1,j_1)$ at one of its endpoints, say $i_2$, but locally counterclockwise from $c(i_1,j_1)$ at its other endpoint. In that case $c(i_1,j_1),c(i_2,j_2)$ forms a \textbf{cycle} of length 2. More generally, we say that a collection of strands $\{c(i_1,j_1), c(i_2,j_2), \dots , c(i_k,j_k)\}$ form a \textbf{cycle} if and only if $c(i_l,j_l)$ is locally clockwise from $c(i_{l+1},j_{l+1})$ for all $l < k$ and $c(i_k,j_k)$ is locally clockwise from $c(i_1,j_1)$. We say a strand $c(i,j)$ is a \textbf{loop} if it does not self-intersect and $c(i,j)$ is both locally clockwise and locally counterclockwise from itself.
\\

\begin{figure}
    \centering
    \tikzset{every picture/.style={line width=0.75pt}} 

\begin{tikzpicture}[x=0.75pt,y=0.75pt,yscale=-1,xscale=1]

\draw    (72,78) -- (94.22,100.01) ;
\draw    (71,79) -- (52.22,100.01) ;
\draw    (168,79) -- (188.22,100.01) ;
\draw    (168,79) -- (179.22,107.01) ;
\draw    (288,76) -- (279.22,103.01) ;
\draw    (288,76) -- (264.22,96.01) ;
\draw    (379,96) -- (398.22,73.01) ;
\draw    (379,96) -- (357.22,72.01) ;
\draw    (499,100) -- (488.22,67.01) ;
\draw    (499,100) -- (469.22,70.01) ;
\draw    (597,98) -- (626.22,67.01) ;
\draw    (597,98) -- (603.22,72.01) ;

\draw (66,70) node [anchor=north west][inner sep=0.75pt]  [font=\Huge]  {$\cdot $};
\draw (37,72.4) node [anchor=north west][inner sep=0.75pt]    {$\gamma _{2}$};
\draw (91,73.4) node [anchor=north west][inner sep=0.75pt]    {$\gamma _{1}$};
\draw (47,108.4) node [anchor=north west][inner sep=0.75pt]    {$\tilde{\varepsilon} _{k} =+$};
\draw (62,133.4) node [anchor=north west][inner sep=0.75pt]    {$( a)$};
\draw (162,71) node [anchor=north west][inner sep=0.75pt]  [font=\Huge]  {$\cdot $};
\draw (153,85.4) node [anchor=north west][inner sep=0.75pt]    {$\gamma _{2}$};
\draw (185,73.4) node [anchor=north west][inner sep=0.75pt]    {$\gamma _{1}$};
\draw (151,109.4) node [anchor=north west][inner sep=0.75pt]    {$\tilde{\varepsilon} _{k} =+$};
\draw (165,133.4) node [anchor=north west][inner sep=0.75pt]    {$( b)$};
\draw (282,68) node [anchor=north west][inner sep=0.75pt]  [font=\Huge]  {$\cdot $};
\draw (257,65.4) node [anchor=north west][inner sep=0.75pt]    {$\gamma _{2}$};
\draw (290,85.4) node [anchor=north west][inner sep=0.75pt]    {$\gamma _{1}$};
\draw (250,109.4) node [anchor=north west][inner sep=0.75pt]    {$\tilde{\varepsilon} _{k} =+$};
\draw (266,133.4) node [anchor=north west][inner sep=0.75pt]    {$( c)$};
\draw (373,88) node [anchor=north west][inner sep=0.75pt]  [font=\Huge]  {$\cdot $};
\draw (344,79.4) node [anchor=north west][inner sep=0.75pt]    {$\gamma _{1}$};
\draw (394,82.4) node [anchor=north west][inner sep=0.75pt]    {$\gamma _{2}$};
\draw (354,108.4) node [anchor=north west][inner sep=0.75pt]    {$\tilde{\varepsilon} _{k} =-$};
\draw (369,133.4) node [anchor=north west][inner sep=0.75pt]    {$( d)$};
\draw (493,92) node [anchor=north west][inner sep=0.75pt]  [font=\Huge]  {$\cdot $};
\draw (464,83.4) node [anchor=north west][inner sep=0.75pt]    {$\gamma _{1}$};
\draw (499,63.4) node [anchor=north west][inner sep=0.75pt]    {$\gamma _{2}$};
\draw (474,112.4) node [anchor=north west][inner sep=0.75pt]    {$\tilde{\varepsilon} _{k} =-$};
\draw (489,137.4) node [anchor=north west][inner sep=0.75pt]    {$( e)$};
\draw (591,90) node [anchor=north west][inner sep=0.75pt]  [font=\Huge]  {$\cdot $};
\draw (579,67.4) node [anchor=north west][inner sep=0.75pt]    {$\gamma _{1}$};
\draw (612,84.4) node [anchor=north west][inner sep=0.75pt]    {$\gamma _{2}$};
\draw (572,110.4) node [anchor=north west][inner sep=0.75pt]    {$\tilde{\varepsilon} _{k} =-$};
\draw (587,135.4) node [anchor=north west][inner sep=0.75pt]    {$(f)$}; 
\end{tikzpicture}
    \caption{The six possible local configurations of strand $c(i_2,j_2)$ being clockwise from strand $c(i_1,j_1)$ near the shared endpoint $(x_k,0)$}
    \label{fig: Locally Clockwise}
\end{figure}
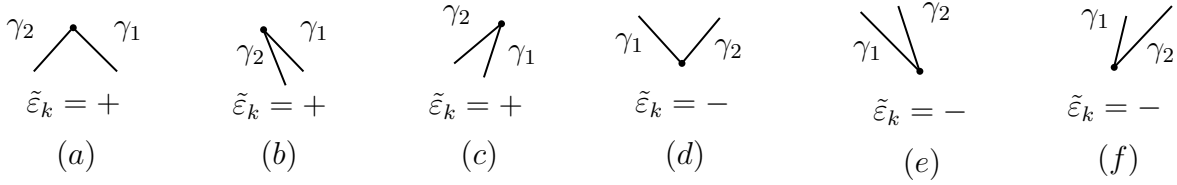

\begin{defn}
A \textbf{fundamental strand diagram} $\tilde{d} = \{(c(i_l,j_l)\}_{l \in [n]}$ is a collection of $n$ non-self intersecting strands on $\tilde{S}_{n,\tilde{\varepsilon}}$ that satisfies the following: 
\begin{enumerate}
\item $c(i_l,j_l)$ is a fundamental strand for all $l$,
\item distinct strands do not intersect nontrivially, and 
\item the graph determined by $\tilde{d}$ contains no loops or cycles.
\end{enumerate}
An example of a fundamental strand diagram can be seen in Figure \ref{fig: Example of Strand and Arc Diagrams}.
\end{defn}

\noindent
By restricting the above definitions to $S_{n,\varepsilon} := \{(x_0,0), (x_1,0), \dots (x_n,0)\}\subset\mathbb{R}^2$ and $\tilde{d}$ to $d = \{(c(i_l,j_l)\}_{l \in [n]}$, we get the strand diagrams from [\ref{ref: combinatorics of exceptional sequences type A}] which we associate to type $\mathbb{A}$ quivers. Note that in this case, it is not possible for a strand to self intersect or form a loop. \\

To quivers of type $\mathbb{\tilde{A}}_{n-1}$, we can associate another combinatorial object known as fundamental arc diagrams on the annulus $A_{Q^{\bm{\varepsilon}}}$ as follows. Recall that $n = |Q_0^{\bm\varepsilon}|$. To make a bijection between string modules and arcs on $A_{Q^{\bm{\varepsilon}}}$, we will introduce another notation for string modules.

\begin{defn}\label{defn: other notation for string}
For $l \in \mathbb{N}$ and $i,j\in[n]$, define the string module $(i,j;l)$ to be the module associated to the walk $e_{i+1}(\alpha_{i+1} \dots \alpha_i)^l\alpha_{i+1}\dots\alpha_{j-1}e_j$. 
\end{defn}

Notice that all strings in $Q^{\bm\varepsilon}$ can be uniquely written in this form. We take the convention that when we write strings in this notation, we take $l$ to be maximal. Using this notation, we have the following bijection $\psi$ between strings and arcs:

 \[ M = (i,j;l) \mapsto \begin{cases} 
          a(i \, (\text{mod } n),j\, (\text{mod } n))[-l] & \text{if $M$ preinjective} \\
          a(i\, (\text{mod } n),j\, (\text{mod } n))[l] & \text{otherwise}
       \end{cases}.
    \]

Notice that under this bijection, bridging arcs that begin on the inner (outer) boundary component correspond to preinjective (preprojective) modules. On the other hand, exterior arcs beginning and ending on the inner (outer) boundary component correspond to right (left) regular modules. This convention of associating string modules to arcs will be called the \textbf{exceptional convention}.

\begin{exmp}
Let $Q^{\bm\varepsilon}$ be the quiver of type $\tilde{\mathbb{A}}_3$ from Example \ref{exam: example of string module} and consider the walk $\alpha_4^{-1}\alpha_1\alpha_2\alpha_3$. Then we denote this string module by $34_5$. We can rewrite this walk as $e_4(\alpha_4^{-1}\alpha_1\alpha_2\alpha_3)^1e_4$. Thus in the notation in Definition \ref{defn: other notation for string}, we can write this string as $(3,4;1)$. Under the map $\psi$, the corresponding arc on the annulus is $a(3,0)[1]$ as shown by the red arc in Figure \ref{fig:example of cover}. 

Now consider the string $43_3 = \alpha_1\alpha_2 = e_1(\alpha_1\alpha_2\alpha_3\alpha_4^{-1})^0\alpha_1\alpha_2e_3$. Then this string is also denoted by $(4,3;0)$ and its corresponding arc is $a(0,3)[0]$ as shown in Example \ref{exmp: naming of arcs}. Finally, consider the string $43_7 = e_1(\alpha_1\alpha_2\alpha_3\alpha_4^{-1})^1\alpha_1\alpha_2e_3 = (4,3;1)$. Then this string corresponds to the arc $a(0,3)[-1]$ as in Example \ref{exmp: naming of arcs}.
\end{exmp}

Let $A_{Q^{\bm\varepsilon}}$ be an annulus associated to a quiver $Q^{\bm{\varepsilon}}$ and suppose that the difference between the two radii of the annulus is $r$. If $a(i_1,j_1)[\lambda_1]$ and $a(i_2,j_2)[\lambda_2]$ do not intersect nontrivially, we say $a(i_2,j_2)[\lambda_2]$ is \textbf{clockwise} from $a(i_1,j_1)[\lambda_1]$ (or equivalently $a(i_1,j_1)[\lambda_1]$ is \textbf{counterclockwise} from $a(i_2,j_2)[\lambda_2]$) if and only if there exists $\gamma_1 \in a(i_{1},j_{1})[\lambda_1]$ and $\gamma_2 \in a(i_2,j_2)[\lambda_2]$ that share at an endpoint $p$, do not intersect in their interiors, and we have the following: if we place a circle of radius ${r\over 2}$ about the shared point $p$, then the circle must be traversed clockwise through the interior of the annulus to get from the point of intersection of $a(i_2,j_2)[\lambda_2]$ with the circle to that of $a(i_1,j_1)[\lambda_1]$ with the circle. Again, we say that a collection of arcs $\{a(i_1,j_1)[\lambda_1], a(i_2,j_2)[\lambda_2], \dots , a(i_k,j_k)[\lambda_k]\}$ forms a \textbf{cycle} if and only if $a(i_l,j_l)[\lambda_l]$ is clockwise from $a(i_{l+1},j_{l+1})[\lambda_{l+1}]$ for all $l < k$ and $a(i_k,j_k)[\lambda_k]$ is clockwise from $a(i_1,j_1)[\lambda_1]$. We say a an arc $a(i,j)[\lambda]$ is a \textbf{loop} if it does not self-intersect and $i=j$.

\begin{rem}\label{rem: difference between two conventions}
To illustrate that the cluster and exceptional conventions are indeed different, notice that the yellow arc in Figure \ref{fig: Example of cluster triangulation bijection} gives $S_3$ in the cluster convention, but gives the string $13_2$ in the exceptional convention. \\
\end{rem}

\begin{defn}
A \textbf{fundamental arc diagram} is a collection of $n$ non-self intersecting arcs on $A_{Q^{\bm{\varepsilon}}}$ that satisfies the following: 
\begin{enumerate}
\item distinct arcs do not intersect nontrivially, and 
\item the arcs do not form any loops or cycles. 
\end{enumerate}
An example of a fundamental arc diagram can be seen in Figure \ref{fig: Example of Strand and Arc Diagrams}.

\end{defn}

\begin{figure}
    \centering

\tikzset{every picture/.style={line width=0.75pt}} 

\begin{tikzpicture}[x=0.75pt,y=0.75pt,yscale=-1,xscale=1]

\draw   [color={rgb, 255:red, 0; green, 0; blue, 0 }  ,draw opacity=1 ] (61,130) .. controls (101,100) and (106.22,213.01) .. (169.22,146.01) ;
\draw    [color={rgb, 255:red, 0; green, 0; blue, 0 }  ,draw opacity=1 ](171,146) .. controls (231.22,216.01) and (271.22,87.01) .. (290.22,129.01) ;
\draw   [color={rgb, 255:red, 0; green, 0; blue, 0 } ,draw opacity=1 ] (114,144) .. controls (115.22,148.01) and (129.22,175.01) .. (169.22,145.01) ;
\draw    [color={rgb, 255:red, 0; green, 0; blue, 0 }  ,draw opacity=1] (172.22,146.01) .. controls (183.44,153.01) and (191.22,167.01) .. (227.22,144.01) ;
\draw   (460.22,137) .. controls (460.22,124.85) and (470.47,115) .. (483.11,115) .. controls (495.75,115) and (506,124.85) .. (506,137) .. controls (506,149.15) and (495.75,159) .. (483.11,159) .. controls (470.47,159) and (460.22,149.15) .. (460.22,137)(427.22,137) .. controls (427.22,106.63) and (452.24,82.01) .. (483.11,82.01) .. controls (513.98,82.01) and (539,106.63) .. (539,137) .. controls (539,167.38) and (513.98,192) .. (483.11,192) .. controls (452.24,192) and (427.22,167.38) .. (427.22,137) ;
\draw    (483.22,115.01) ;
\draw [shift={(483.22,115.01)}, rotate = 0] [color={rgb, 255:red, 0; green, 0; blue, 0 }  ][fill={rgb, 255:red, 0; green, 0; blue, 0 }  ][line width=0.75]      (0, 0) circle [x radius= 3.35, y radius= 3.35]   ;
\draw    (537.22,136.01) -- (538.22,135.01) ;
\draw [shift={(538.22,135.01)}, rotate = 315] [color={rgb, 255:red, 0; green, 0; blue, 0 }  ][fill={rgb, 255:red, 0; green, 0; blue, 0 }  ][line width=0.75]      (0, 0) circle [x radius= 3.35, y radius= 3.35]   ;
\draw    (484.22,191.01) ;
\draw [shift={(484.22,191.01)}, rotate = 0] [color={rgb, 255:red, 0; green, 0; blue, 0 }  ][fill={rgb, 255:red, 0; green, 0; blue, 0 }  ][line width=0.75]      (0, 0) circle [x radius= 3.35, y radius= 3.35]   ;
\draw    (427.22,136.01) ;
\draw [shift={(427.22,136.01)}, rotate = 0] [color={rgb, 255:red, 0; green, 0; blue, 0 }  ][fill={rgb, 255:red, 0; green, 0; blue, 0 }  ][line width=0.75]      (0, 0) circle [x radius= 3.35, y radius= 3.35]   ;
\draw    (484,114) .. controls (516.22,85.01) and (549.22,137.01) .. (484.22,190.01) ;
\draw    (484,115) .. controls (466.22,80.01) and (404.22,136.01) .. (484.22,191.01) ;
\draw    (426.22,136.01) .. controls (431.22,141.01) and (434.22,179.01) .. (485,191) ;
\draw    (484.22,191.01) .. controls (513.22,188.01) and (535.22,142.01) .. (538.22,135.01) ;

\draw (54,128.4) node [anchor=north west][inner sep=0.75pt]    {$0$};
\draw (108,128.4) node [anchor=north west][inner sep=0.75pt]    {$1$};
\draw (165,128.4) node [anchor=north west][inner sep=0.75pt]    {$2$};
\draw (223,128.4) node [anchor=north west][inner sep=0.75pt]    {$3$};
\draw (286,128.4) node [anchor=north west][inner sep=0.75pt]    {$4$};
\draw (479,120.4) node [anchor=north west][inner sep=0.75pt]    {$0$};
\draw (547,127.4) node [anchor=north west][inner sep=0.75pt]    {$1$};
\draw (478,206.4) node [anchor=north west][inner sep=0.75pt]    {$2$};
\draw (404,128.4) node [anchor=north west][inner sep=0.75pt]    {$3$};

\end{tikzpicture}

    \caption{A fundamental strand diagram and its corresponding fundamental arc diagram for the quiver of type $\tilde{\mathbb{A}}_3$ with orientation vector $\bm{\varepsilon} = (-,+,+,+)$ from Example \ref{exam: example of string module}. Note that this is also a strand diagram for a quiver of type $\mathbb{A}_4$ with the same orientation vector.}
    \label{fig: Example of Strand and Arc Diagrams}
\end{figure}
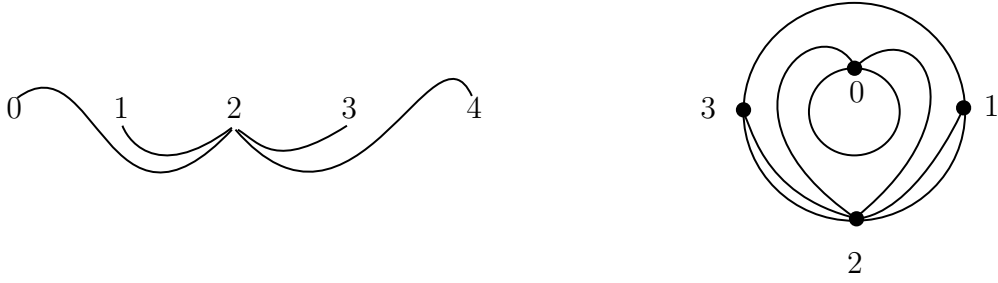

\noindent
The following theorems are from [\ref{ref: combinatorics of exceptional sequences type A}] and [\ref{ref: Maresca}] respectively:

\begin{thm} {\color{white} .}

\begin{itemize}
	\item Exceptional sets of type $\mathbb{A}_n$ are in bijection with fundamental strand diagrams $d$ on $S_{n,\varepsilon}$. [\ref{ref: combinatorics of exceptional sequences type A}]
	\item Exceptional sets of type $\mathbb{\tilde{A}}_{n-1}$ are in bijection with fundamental strand diagrams $\tilde{d}$ on $\tilde{S}_{n,\tilde{\varepsilon}}$ and fundamental arc diagrams on $A_{Q^{\bm{\varepsilon}}}$. [\ref{ref: Maresca}]
\end{itemize}

\end{thm}

It is known that there are infinitely many exceptional sets of type $\mathbb{\tilde{A}}_n$. However, we can place them into families as follows.

\begin{defn}
A \textbf{parametrized family} of complete exceptional collections of $\Bbbk Q^{\bm{\varepsilon}}$ modules, denoted by $\Xi$, is a set of exceptional collections such that $\xi_1,\xi_2 \in \Xi$ if and only if the fundamental arc diagram corresponding to $\xi_1$ differs from the fundamental arc diagram corresponding to $\xi_2$ by a sequence of $2\pi$ Dehn twists of the inner boundary component of $A_{Q^{\bm\varepsilon}}$. 
\end{defn}

\begin{rem}
An algebraic description of parametrized families has been given in [\ref{ref: Maresca}] in terms of the Auslander--Reiten quiver of the transjective component of $D^b(\Bbbk Q^{\bm\varepsilon})$.
\end{rem}

It was shown in [\ref{ref: Maresca}] that under the exceptional convention, regular modules (exterior arcs) $a(i,j)[\lambda]$ require $\lambda = 0$ to be exceptional. A diagram is called \textbf{small} if $\lambda = 0$ for all arcs in the diagram. Two small diagrams are \textbf{inner equivalent} if they are in the same family. The following is from [\ref{ref: Maresca}]:

\begin{thm}\label{thm: bijection with small diagrams}
Let $Q^{\bm{\varepsilon}}$ be a quiver of type $\tilde{\mathbb{A}}_{n}$. Then there are finitely many parametrized families of exceptional collections of $\Bbbk Q^{\bm{\varepsilon}}$-modules and they are in bijection with inner equivalence classes of small fundamental arc diagrams on $A_{Q^{\bm{\varepsilon}}}$. \qed
\end{thm}

\section{Counting Families of Clusters}

\indent 

As was done in [\ref{ref: Maresca}], we can also place triangulations of the annulus into families. Let $Q^{\bm\varepsilon}$ be a quiver of type $\tilde{\mathbb{A}}_n$ and $T$ be a triangulation of the annulus $A_{Q^{\bm\varepsilon}}$ with $p$ marked points on the outer circle corresponding to the $i\in Q_0^{\bm{\varepsilon}}$ such that $\varepsilon_i = +$. Each point on the outer circle is labeled with its corresponding $i$ in clockwise order respecting their natural numerical order. Moreover, we label the $q$ marked points on the inner circle analogously. Then $T$ consists of $p + q = n+1 = |Q_0^{\bm\varepsilon}|$ arcs of the form $a(i,j)[\lambda]$. 

\begin{defn} {\color{white} .}
\begin{itemize}
\item A \textbf{parametrized family} of triangulations of $A_{Q^{\bm\varepsilon}}$, denoted by $\mathscr{T}$, is a collection of triangulations of $A_{Q^{\bm\varepsilon}}$ such that $T_1,T_2 \in \mathscr{T}$ if and only if $T_1$ is attained from $T_2$ by a sequence of $2\pi$ Dehn twists of the inner boundary circle of $A_{Q^{\bm\varepsilon}}$.
\item A triangulation is called \textbf{small} if $\lambda = 0$ for all bridging arcs in the triangulation. Two small triangulations are \textbf{inner equivalent} if they belong to the same parametrized family.
\end{itemize}
\end{defn}

Note that by definition, every triangulation belongs in a unique family. As was shown in [\ref{ref: Master's Thesis Clusters and Triangulations}] and [\ref{ref: Maresca}], the parameters of all bridging arcs can differ by at most $1$ to avoid crossings. The work done in [\ref{ref: Maresca}] to show the existence of a unique inner equivalence class of small diagrams in each family of exceptional collections also shows that there is a unique equivalence class of small triangulations in each family of triangulations since the proofs are independent of the existence of cycles. Moreover, the fact that clusters in type $\tilde{\mathbb{A}}_n$ are in bijection with triangulations of the annulus $A_{Q^{\bm\varepsilon}}$ allows us to conclude that clusters also occur in parametrized families. Our next goal is to count how many of these families there are in the straight orientation case by counting all possible small triangulations. \\

Consider the annulus $A_{Q^{\bm\varepsilon}}$ associated to the straightly oriented $\tilde{\mathbb{A}}_n$ quiver $Q^{\bm\varepsilon}$. Then there is one marked point on the inner circle and $n$ marked points on the outer circle. By placing these outer marked points equidistant from one another, we have an action of ${2\pi \over n}$ clockwise Dehn twists of the unlabeled annulus about the \textit{outer} circle. After we perform this twist, we label the vertices with the same initial label as depicted in Figure \ref{fig: Example of orbit}. This action is well defined on both triangulations and fundamental arc diagrams since it is a homeomorphism. Moreover, it is well defined on parametrized families of both triangulations and fundamental arc diagrams as shown in the next proposition.

\begin{prop}
The action of ${2\pi\over n}$ clockwise outer Dehn twists is well defined on parametrized families of triangulations and fundamental arc diagrams.
\end{prop}

\begin{proof}
To make the proof more concise, we adopt the convention that $a(0,i)[\lambda]$ where $\lambda \geq 1$ is the arc $a(i,0)[\lambda-1]$. Throughout, by \textbf{arc diagram}, we mean a collection of $n+1$ non-crossing arcs on the annulus $A_Q^{\bm\varepsilon}$. Consider the arc diagram $D$ given by $\{a(0,i_1)[\lambda_1]$, $a(0,i_2)[\lambda_2]$, $\dots,$ $ a(0,i_k)[\lambda_k]$, $a(i_{k+1},j_{k+1})$, $\dots$ , $a(i_{n+1},j_{n+1})$\} where $i_1\leq i_2\leq \dots \leq i_k$ and the arcs without parameters have both endpoints on the outer boundary. Then if $i_k\neq n$, a ${2\pi\over n}$ clockwise outer Dehn twist gives the diagram containing the arcs$\{a(0,i_1+1)[\lambda_1]$, $a(0,i_2+1)[\lambda_2]$, $\dots,$ $ a(0,i_k+1)[\lambda_k]$, $a(i_{k+1}+1,j_{k+1}+1)$, $\dots$ , $a(i_{n+1}+1,j_{n+1}+1)$\} where for $p >k$, $i_p+1 = 1$ and  $j_p + 1 = 1$ if $i_p = j_p = n$ and $i_p+1$ and $j_p+1$ equal themselves otherwise. If $i_k = n$, since the arcs do not cross, we have either $i_{k-1} = i_{k}$ or $i_{k-1} < i_k$. In the first case, we have that the ${2\pi\over n}$ clockwise outer Dehn twist gives the diagram containing the arcs $\{a(0,i_1+1)[\lambda_1], a(0,i_2+1)[\lambda_2], \dots, a(0,1)[\lambda_{k-1}-1], a(0,1)[\lambda_k-1], a(i_{k+1}+1,j_{k+1}+1), \dots , a(i_{n+1}+1,j_{n+1}+1)\}$ and in the second case, we get the diagram $\{a(0,i_1+1)[\lambda_1], a(0,i_2+1)[\lambda_2], \dots, a(0,i_{k-1}+1)[\lambda_{k-1}], a(0,1)[\lambda_k-1], a(i_{k+1}+1,j_{k+1}+1), \dots , a(i_{n+1}+1,j_{n+1}+1)\}$. In both these cases, we still have that for $p >k$, $i_p+1 = 1$ and  $j_p + 1 = 1$ if $i_p = j_p = n$ and $i_p+1$ and $j_p+1$ equal themselves otherwise. 

Now suppose we have $D'$ in the same family as $D$, so $D'$ contains the arcs $\{a(0,i_1)[\lambda'_1]$, $a(0,i_2)[\lambda'_2]$, $\dots,$ $ a(0,i_k)[\lambda'_k]$, $a(i_{k+1},j_{k+1})$, $\dots$ , $a(i_{n+1},j_{n+1})$\} where for all $l$, $\lambda'_l = \lambda_l + z$ for some $z\in\mathbb{Z}$. Then if $i_k\neq n$, a ${2\pi\over n}$ clockwise outer Dehn twist gives the diagram containing the arcs $\{a(0,i_1+1)[\lambda'_1]$, $a(0,i_2+1)[\lambda'_2]$, $\dots,$ $ a(0,i_k+1)[\lambda'_k]$, $a(i_{k+1}+1,j_{k+1}+1)$, $\dots$ , $a(i_{n+1}+1,j_{n+1}+1)$\}. After $z$ $2\pi$ inner Dehn twists in the appropriate direction, we will get the diagram $\{a(0,i_1+1)[\lambda_1], a(0,i_2+1)[\lambda_2], \dots, a(0,i_k+1)[\lambda_k], a(i_{k+1}+1,j_{k+1}+1), \dots , a(i_{n+1}+1,j_{n+1}+1)\}$. A similar argument holds when $k = n$. Therefore we conclude that if two arc diagrams are in the same family, a ${2\pi\over n}$ clockwise outer Dehn twist of each diagram will give arc diagrams that are in the same family.
\end{proof}

This leads us to the following definition.

\begin{defn}
We define two arc diagrams to be \textbf{outer equivalent} if their families differ by a sequence of ${2\pi\over n}$ clockwise outer Dehn twists.
\end{defn}

Note that outer equivalence forms an equivalence relation on families and therefore, each family of arc diagrams belongs in one and only one outer equivalence class. Moreover, note that given an arc diagram, there are precisely $n$ families of arc diagrams in its outer equivalence class.  

As shown in Lemma \ref{lem: unique small triangulation}, each of these outer equivalence classes of a small triangulation contains precisely $n$ distinct small triangulations, and hence families of clusters. We will show there are precisely $C_n$ outer equivalence classes of triangulations where $C_n = {1\over n+1} \binom{2n}{n}$ is the $n^{th}$ Catalan number [\ref{ref: Catalan Numbers}]. To do this, we begin with a lemma.

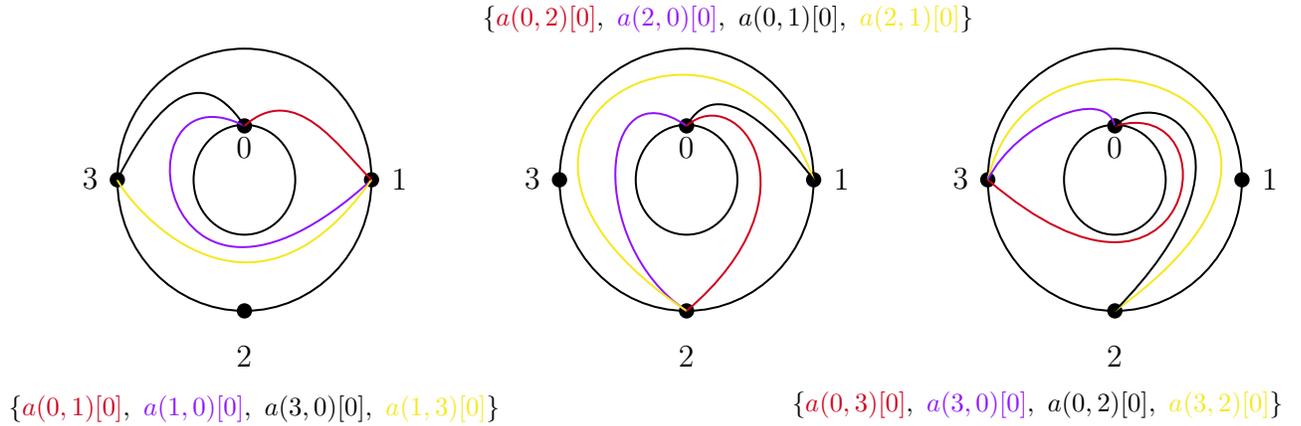
\begin{figure}
    \centering

\tikzset{every picture/.style={line width=0.75pt}} 

\begin{tikzpicture}[x=0.75pt,y=0.75pt,yscale=-1,xscale=1]

\draw   (104.58,113.87) .. controls (104.58,98.56) and (116.06,86.15) .. (130.22,86.15) .. controls (144.38,86.15) and (155.86,98.56) .. (155.86,113.87) .. controls (155.86,129.18) and (144.38,141.59) .. (130.22,141.59) .. controls (116.06,141.59) and (104.58,129.18) .. (104.58,113.87)(66.13,113.87) .. controls (66.13,77.32) and (94.82,47.69) .. (130.22,47.69) .. controls (165.62,47.69) and (194.32,77.32) .. (194.32,113.87) .. controls (194.32,150.42) and (165.62,180.05) .. (130.22,180.05) .. controls (94.82,180.05) and (66.13,150.42) .. (66.13,113.87) ;
\draw    (130.22,86.68) ;
\draw [shift={(130.22,86.68)}, rotate = 0] [color={rgb, 255:red, 0; green, 0; blue, 0 }  ][fill={rgb, 255:red, 0; green, 0; blue, 0 }  ][line width=0.75]      (0, 0) circle [x radius= 3.35, y radius= 3.35]   ;
\draw    (194.32,113.87) ;
\draw [shift={(194.32,113.87)}, rotate = 0] [color={rgb, 255:red, 0; green, 0; blue, 0 }  ][fill={rgb, 255:red, 0; green, 0; blue, 0 }  ][line width=0.75]      (0, 0) circle [x radius= 3.35, y radius= 3.35]   ;
\draw    (130.22,180.05) ;
\draw [shift={(130.22,180.05)}, rotate = 0] [color={rgb, 255:red, 0; green, 0; blue, 0 }  ][fill={rgb, 255:red, 0; green, 0; blue, 0 }  ][line width=0.75]      (0, 0) circle [x radius= 3.35, y radius= 3.35]   ;
\draw    (66.13,113.87) ;
\draw [shift={(66.13,113.87)}, rotate = 0] [color={rgb, 255:red, 0; green, 0; blue, 0 }  ][fill={rgb, 255:red, 0; green, 0; blue, 0 }  ][line width=0.75]      (0, 0) circle [x radius= 3.35, y radius= 3.35]   ;
\draw [color={rgb, 255:red, 144; green, 19; blue, 254 }  ,draw opacity=1 ]   (130.22,86.68) .. controls (65.42,53.67) and (82.91,215.52) .. (194.32,113.87) ;
\draw [color={rgb, 255:red, 248; green, 231; blue, 28 }  ,draw opacity=1 ]   (66.13,113.87) .. controls (70.89,124.62) and (132.35,202.09) .. (194.32,113.87) ;
\draw [color={rgb, 255:red, 0; green, 0; blue, 0 }  ,draw opacity=1 ]   (66.13,113.87) .. controls (80.73,84.71) and (104.79,48.13) .. (130.22,86.68) ;
\draw [color={rgb, 255:red, 208; green, 2; blue, 27 }  ,draw opacity=1 ]   (130.22,86.68) .. controls (156.18,63.65) and (178.05,98.01) .. (194.32,113.87) ;
\draw   (327.58,113.87) .. controls (327.58,98.56) and (339.06,86.15) .. (353.22,86.15) .. controls (367.38,86.15) and (378.86,98.56) .. (378.86,113.87) .. controls (378.86,129.18) and (367.38,141.59) .. (353.22,141.59) .. controls (339.06,141.59) and (327.58,129.18) .. (327.58,113.87)(289.13,113.87) .. controls (289.13,77.32) and (317.82,47.69) .. (353.22,47.69) .. controls (388.62,47.69) and (417.32,77.32) .. (417.32,113.87) .. controls (417.32,150.42) and (388.62,180.05) .. (353.22,180.05) .. controls (317.82,180.05) and (289.13,150.42) .. (289.13,113.87) ;
\draw    (353.22,86.68) ;
\draw [shift={(353.22,86.68)}, rotate = 0] [color={rgb, 255:red, 0; green, 0; blue, 0 }  ][fill={rgb, 255:red, 0; green, 0; blue, 0 }  ][line width=0.75]      (0, 0) circle [x radius= 3.35, y radius= 3.35]   ;
\draw    (417.32,113.87) ;
\draw [shift={(417.32,113.87)}, rotate = 0] [color={rgb, 255:red, 0; green, 0; blue, 0 }  ][fill={rgb, 255:red, 0; green, 0; blue, 0 }  ][line width=0.75]      (0, 0) circle [x radius= 3.35, y radius= 3.35]   ;
\draw    (353.22,180.05) ;
\draw [shift={(353.22,180.05)}, rotate = 0] [color={rgb, 255:red, 0; green, 0; blue, 0 }  ][fill={rgb, 255:red, 0; green, 0; blue, 0 }  ][line width=0.75]      (0, 0) circle [x radius= 3.35, y radius= 3.35]   ;
\draw    (289.13,113.87) ;
\draw [shift={(289.13,113.87)}, rotate = 0] [color={rgb, 255:red, 0; green, 0; blue, 0 }  ][fill={rgb, 255:red, 0; green, 0; blue, 0 }  ][line width=0.75]      (0, 0) circle [x radius= 3.35, y radius= 3.35]   ;
\draw [color={rgb, 255:red, 144; green, 19; blue, 254 }  ,draw opacity=1 ]   (353.22,86.68) .. controls (308.35,55.39) and (302.35,146.39) .. (353.22,180.05) ;
\draw [color={rgb, 255:red, 248; green, 231; blue, 28 }  ,draw opacity=1 ]   (417.32,113.87) .. controls (378.35,2.39) and (214.35,87.39) .. (353.22,180.05) ;
\draw [color={rgb, 255:red, 0; green, 0; blue, 0 }  ,draw opacity=1 ]   (417.32,113.87) .. controls (400.35,91.39) and (369.35,58.39) .. (353.22,86.68) ;
\draw [color={rgb, 255:red, 208; green, 2; blue, 27 }  ,draw opacity=1 ]   (353.22,86.68) .. controls (379.18,63.65) and (422.35,121.39) .. (353.22,180.05) ;
\draw   (543.58,113.87) .. controls (543.58,98.56) and (555.06,86.15) .. (569.22,86.15) .. controls (583.38,86.15) and (594.86,98.56) .. (594.86,113.87) .. controls (594.86,129.18) and (583.38,141.59) .. (569.22,141.59) .. controls (555.06,141.59) and (543.58,129.18) .. (543.58,113.87)(505.13,113.87) .. controls (505.13,77.32) and (533.82,47.69) .. (569.22,47.69) .. controls (604.62,47.69) and (633.32,77.32) .. (633.32,113.87) .. controls (633.32,150.42) and (604.62,180.05) .. (569.22,180.05) .. controls (533.82,180.05) and (505.13,150.42) .. (505.13,113.87) ;
\draw    (569.22,86.68) ;
\draw [shift={(569.22,86.68)}, rotate = 0] [color={rgb, 255:red, 0; green, 0; blue, 0 }  ][fill={rgb, 255:red, 0; green, 0; blue, 0 }  ][line width=0.75]      (0, 0) circle [x radius= 3.35, y radius= 3.35]   ;
\draw    (633.32,113.87) ;
\draw [shift={(633.32,113.87)}, rotate = 0] [color={rgb, 255:red, 0; green, 0; blue, 0 }  ][fill={rgb, 255:red, 0; green, 0; blue, 0 }  ][line width=0.75]      (0, 0) circle [x radius= 3.35, y radius= 3.35]   ;
\draw    (569.22,180.05) ;
\draw [shift={(569.22,180.05)}, rotate = 0] [color={rgb, 255:red, 0; green, 0; blue, 0 }  ][fill={rgb, 255:red, 0; green, 0; blue, 0 }  ][line width=0.75]      (0, 0) circle [x radius= 3.35, y radius= 3.35]   ;
\draw    (505.13,113.87) ;
\draw [shift={(505.13,113.87)}, rotate = 0] [color={rgb, 255:red, 0; green, 0; blue, 0 }  ][fill={rgb, 255:red, 0; green, 0; blue, 0 }  ][line width=0.75]      (0, 0) circle [x radius= 3.35, y radius= 3.35]   ;
\draw [color={rgb, 255:red, 208; green, 2; blue, 27 }  ,draw opacity=1 ]   (569.22,86.68) .. controls (630.35,69.39) and (610.35,204.39) .. (505.13,113.87) ;
\draw [color={rgb, 255:red, 248; green, 231; blue, 28 }  ,draw opacity=1 ]   (505.13,113.87) .. controls (529.35,12.39) and (711.35,78.39) .. (569.22,180.05) ;
\draw [color={rgb, 255:red, 144; green, 19; blue, 254 }  ,draw opacity=1 ]   (505.13,113.87) .. controls (519.73,84.71) and (564.35,66.39) .. (569.22,86.68) ;
\draw [color={rgb, 255:red, 0; green, 0; blue, 0 }  ,draw opacity=1 ]   (569.22,86.68) .. controls (595.18,63.65) and (646.35,101.39) .. (569.22,180.05) ;

\draw (124.65,91.09) node [anchor=north west][inner sep=0.75pt]    {$0$};
\draw (202.96,106.79) node [anchor=north west][inner sep=0.75pt]    {$1$};
\draw (124.65,196.4) node [anchor=north west][inner sep=0.75pt]    {$2$};
\draw (47.01,106.61) node [anchor=north west][inner sep=0.75pt]    {$3$};
\draw (347.65,91.09) node [anchor=north west][inner sep=0.75pt]    {$0$};
\draw (425.96,106.79) node [anchor=north west][inner sep=0.75pt]    {$1$};
\draw (347.65,196.4) node [anchor=north west][inner sep=0.75pt]    {$2$};
\draw (270.01,106.61) node [anchor=north west][inner sep=0.75pt]    {$3$};
\draw (563.65,91.09) node [anchor=north west][inner sep=0.75pt]    {$0$};
\draw (641.96,106.79) node [anchor=north west][inner sep=0.75pt]    {$1$};
\draw (563.65,196.4) node [anchor=north west][inner sep=0.75pt]    {$2$};
\draw (486.01,106.61) node [anchor=north west][inner sep=0.75pt]    {$3$};
\draw (10,221.4) node [anchor=north west][inner sep=0.75pt]  [font=\footnotesize]  {$\{\textcolor[rgb]{0.82,0.01,0.11}{a( 0,1)[ 0]} ,\ \textcolor[rgb]{0.56,0.07,1}{a( 1,0)[ 0]} ,\ a( 3,0)[ 0] ,\textcolor[rgb]{0.97,0.91,0.11}{\ a( 1,3)[0]}\}$};
\draw (249,24.4) node [anchor=north west][inner sep=0.75pt]  [font=\footnotesize]  {$\{\textcolor[rgb]{0.82,0.01,0.11}{a}\textcolor[rgb]{0.82,0.01,0.11}{(}\textcolor[rgb]{0.82,0.01,0.11}{0,2}\textcolor[rgb]{0.82,0.01,0.11}{)}\textcolor[rgb]{0.82,0.01,0.11}{[}\textcolor[rgb]{0.82,0.01,0.11}{0}\textcolor[rgb]{0.82,0.01,0.11}{]} ,\ \textcolor[rgb]{0.56,0.07,1}{a}\textcolor[rgb]{0.56,0.07,1}{(}\textcolor[rgb]{0.56,0.07,1}{2,0}\textcolor[rgb]{0.56,0.07,1}{)}\textcolor[rgb]{0.56,0.07,1}{[}\textcolor[rgb]{0.56,0.07,1}{0}\textcolor[rgb]{0.56,0.07,1}{]} ,\ a( 0,1)[ 0] ,\textcolor[rgb]{0.97,0.91,0.11}{\ a}\textcolor[rgb]{0.97,0.91,0.11}{(}\textcolor[rgb]{0.97,0.91,0.11}{2,1}\textcolor[rgb]{0.97,0.91,0.11}{)[0]}\}$};
\draw (405,219.4) node [anchor=north west][inner sep=0.75pt]  [font=\footnotesize]  {$\{\textcolor[rgb]{0.82,0.01,0.11}{a}\textcolor[rgb]{0.82,0.01,0.11}{(}\textcolor[rgb]{0.82,0.01,0.11}{0,3}\textcolor[rgb]{0.82,0.01,0.11}{)}\textcolor[rgb]{0.82,0.01,0.11}{[}\textcolor[rgb]{0.82,0.01,0.11}{0}\textcolor[rgb]{0.82,0.01,0.11}{]} ,\ \textcolor[rgb]{0.56,0.07,1}{a}\textcolor[rgb]{0.56,0.07,1}{(}\textcolor[rgb]{0.56,0.07,1}{3,0}\textcolor[rgb]{0.56,0.07,1}{)}\textcolor[rgb]{0.56,0.07,1}{[}\textcolor[rgb]{0.56,0.07,1}{0}\textcolor[rgb]{0.56,0.07,1}{]} ,\ a( 0,2)[ 0] ,\textcolor[rgb]{0.97,0.91,0.11}{\ a}\textcolor[rgb]{0.97,0.91,0.11}{(}\textcolor[rgb]{0.97,0.91,0.11}{3,2}\textcolor[rgb]{0.97,0.91,0.11}{)[0]}\}$};

\end{tikzpicture}

    \caption{The outer equivalence class of the triangulation in Figure \ref{fig: Example of cluster triangulation bijection}}
    \label{fig: Example of orbit}
\end{figure}

\begin{lem} \label{lem: All triangulations have a heart}
Any small triangulation of $A_{Q^{\bm\varepsilon}}$ contains arcs of the form $a(0,i)[0]$ and $a(i,0)[0]$ for some $i$.
\end{lem}

\begin{proof}
Suppose $T$ is a triangulation of $A_{Q^{\bm\varepsilon}}$ and that there do not exist bridging arcs such that one ends at $i$ and the other begins at $i$. Since $T$ is a triangulation, there can be at most $n-2$ arcs connecting marked points on the outer circle. Whence for $j\geq 3$, suppose the only arcs in $T$ that begin or end at zero are $a_1, a_2, \dots , a_j$. Let $i_1<i_2<\dots<i_j$ be the nonzero endpoints of these arcs. Since $T$ is a triangulation, there must be precisely $i_{l+1} - i_l - 1$ arcs that begin and end between $i_{l+1}$ and $i_l$ when $l\neq j$ and $n-i_j$ arcs when $l=j$. We conclude that the total number of arcs in $T$ is 
$$j + n-i_j + \sum_{l=1}^{j-1} i_{l+1}-i_l-1 = j + n-i_j + i_j - i_1-j+1 = n-i_1+1.$$

\noindent
Since $i_1 \geq 1$ we have $n-i_1+1 \leq n$ which contradicts the assumption that $T$ is a triangulation. Thus any triangulation contains bridging arcs such that one ends at $i$ and the other begins at $i$. Since $T$ is small, they must be of the form $a(0,i)[0]$ and $a(i,0)[0]$. 
\end{proof}

This lemma allows us to prove that families of triangulations are in bijection with small triangulations; contrary to the case of exceptional collections in which we needed equivalence classes of small diagrams to form the analogous bijection.

\begin{lem}\label{lem: unique small triangulation}
Each family of triangulations contains a unique small triangulation.
\end{lem}

\begin{proof}
For any triangulation, there exists a small triangulation in the same family. Let $T$ be this small triangulation. Then by the previous lemma, $T$ contains arcs of the form $a(0,i)[0]$ and $a(i,0)[0]$. If there were another small triangulation in this family, it must differ from $T$ by a $2\pi$ Dehn twist about the inner circle. But notice that a $2\pi$ clockwise Dehn twist creates the arc $a(i,0)[1]$ and a $2\pi$ counter clockwise Dehn twist creates the arc $a(0,i)[-1]$. In either case, the diagram is no longer small, so $T$ is unique.  
\end{proof}

By Lemmas \ref{lem: All triangulations have a heart} and \ref{lem: unique small triangulation}, it suffices to count the number of small triangulations containing the arcs $a(0,1)$ and $a(1,0)$ since triangulations are preserved under the outer rotation action and every outer equivalence class contains a unique small triangulation with these arcs.

\begin{lem}\label{lem: Counting number of triangulations}
The number of small triangulations of the annulus containing the arcs $a(0,1)$ and $a(1,0)$ is $C_n$ where $C_n$ is the $n^{th}$ Catalan number.
\end{lem}

\begin{proof}
Since the triangulation contains the arcs $a(0,1)$ and $a(1,0)$, the universal cover appears as it does on the left of Figure \ref{fig: Lift of Triangulation}. To triangulate the annulus, it suffices to triangulate the polygon on the right of Figure \ref{fig: Lift of Triangulation} and there are precisely $C_n$ triangulations of a convex $(n+2)$-gon [\ref{ref: Catalan Numbers}]. \\

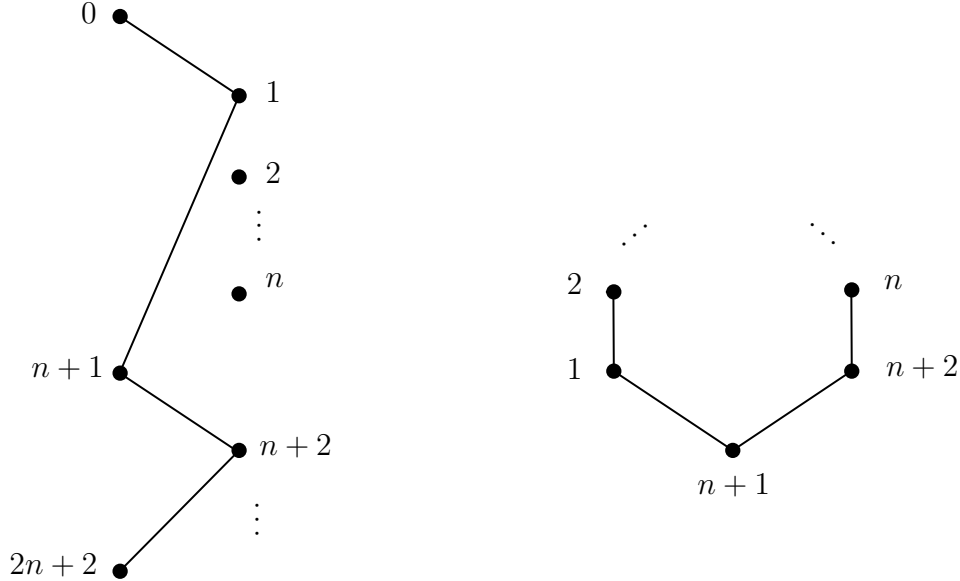
\begin{figure}[h!]
    \centering

\tikzset{every picture/.style={line width=0.75pt}} 

\begin{tikzpicture}[x=0.75pt,y=0.75pt,yscale=-1,xscale=1]

\draw    (69.22,21.01) ;
\draw [shift={(69.22,21.01)}, rotate = 0] [color={rgb, 255:red, 0; green, 0; blue, 0 }  ][fill={rgb, 255:red, 0; green, 0; blue, 0 }  ][line width=0.75]      (0, 0) circle [x radius= 3.35, y radius= 3.35]   ;
\draw    (129.22,61.01) ;
\draw [shift={(129.22,61.01)}, rotate = 0] [color={rgb, 255:red, 0; green, 0; blue, 0 }  ][fill={rgb, 255:red, 0; green, 0; blue, 0 }  ][line width=0.75]      (0, 0) circle [x radius= 3.35, y radius= 3.35]   ;
\draw    (129.22,102.01) ;
\draw [shift={(129.22,102.01)}, rotate = 0] [color={rgb, 255:red, 0; green, 0; blue, 0 }  ][fill={rgb, 255:red, 0; green, 0; blue, 0 }  ][line width=0.75]      (0, 0) circle [x radius= 3.35, y radius= 3.35]   ;
\draw    (129.22,161.01) ;
\draw [shift={(129.22,161.01)}, rotate = 0] [color={rgb, 255:red, 0; green, 0; blue, 0 }  ][fill={rgb, 255:red, 0; green, 0; blue, 0 }  ][line width=0.75]      (0, 0) circle [x radius= 3.35, y radius= 3.35]   ;
\draw    (69.22,201.01) ;
\draw [shift={(69.22,201.01)}, rotate = 0] [color={rgb, 255:red, 0; green, 0; blue, 0 }  ][fill={rgb, 255:red, 0; green, 0; blue, 0 }  ][line width=0.75]      (0, 0) circle [x radius= 3.35, y radius= 3.35]   ;
\draw    (129.22,240.01) ;
\draw [shift={(129.22,240.01)}, rotate = 0] [color={rgb, 255:red, 0; green, 0; blue, 0 }  ][fill={rgb, 255:red, 0; green, 0; blue, 0 }  ][line width=0.75]      (0, 0) circle [x radius= 3.35, y radius= 3.35]   ;
\draw    (68.22,301.01) -- (69.22,301.01) ;
\draw [shift={(69.22,301.01)}, rotate = 0] [color={rgb, 255:red, 0; green, 0; blue, 0 }  ][fill={rgb, 255:red, 0; green, 0; blue, 0 }  ][line width=0.75]      (0, 0) circle [x radius= 3.35, y radius= 3.35]   ;
\draw    (69,21) -- (129.22,61.01) ;
\draw    (69,201) -- (129.22,241.01) ;
\draw    (129.22,240.01) -- (69,301) ;
\draw    (129.22,61.01) -- (69,200.91) ;
\draw    (318.22,201.01) -- (318.22,200.01) ;
\draw [shift={(318.22,200.01)}, rotate = 270] [color={rgb, 255:red, 0; green, 0; blue, 0 }  ][fill={rgb, 255:red, 0; green, 0; blue, 0 }  ][line width=0.75]      (0, 0) circle [x radius= 3.35, y radius= 3.35]   ;
\draw    (314.22,160.01) -- (318.22,160.01) ;
\draw [shift={(318.22,160.01)}, rotate = 0] [color={rgb, 255:red, 0; green, 0; blue, 0 }  ][fill={rgb, 255:red, 0; green, 0; blue, 0 }  ][line width=0.75]      (0, 0) circle [x radius= 3.35, y radius= 3.35]   ;
\draw    (438.22,159.01) ;
\draw [shift={(438.22,159.01)}, rotate = 0] [color={rgb, 255:red, 0; green, 0; blue, 0 }  ][fill={rgb, 255:red, 0; green, 0; blue, 0 }  ][line width=0.75]      (0, 0) circle [x radius= 3.35, y radius= 3.35]   ;
\draw    (438.22,200.01) ;
\draw [shift={(438.22,200.01)}, rotate = 0] [color={rgb, 255:red, 0; green, 0; blue, 0 }  ][fill={rgb, 255:red, 0; green, 0; blue, 0 }  ][line width=0.75]      (0, 0) circle [x radius= 3.35, y radius= 3.35]   ;
\draw    (378.22,239.01) -- (378.22,240.01) ;
\draw [shift={(378.22,240.01)}, rotate = 90] [color={rgb, 255:red, 0; green, 0; blue, 0 }  ][fill={rgb, 255:red, 0; green, 0; blue, 0 }  ][line width=0.75]      (0, 0) circle [x radius= 3.35, y radius= 3.35]   ;
\draw    (318,160) -- (318.22,200.01) ;
\draw    (438,159) -- (438.22,199.01) ;
\draw    (438,200) -- (378.22,240.01) ;
\draw    (378.22,240.01) -- (318.22,200.01) ;
\draw (141.47,115.92) node [anchor=north west][inner sep=0.75pt]  [rotate=-89.41]  {$\dotsc $};
\draw (140.47,263.92) node [anchor=north west][inner sep=0.75pt]  [rotate=-89.41]  {$\dotsc $};
\draw (48,12.4) node [anchor=north west][inner sep=0.75pt]    {$0$};
\draw (141,52.4) node [anchor=north west][inner sep=0.75pt]    {$1$};
\draw (141,93.4) node [anchor=north west][inner sep=0.75pt]    {$2$};
\draw (141,149.4) node [anchor=north west][inner sep=0.75pt]    {$n$};
\draw (23,191.4) node [anchor=north west][inner sep=0.75pt]    {$n+1$};
\draw (138,230.4) node [anchor=north west][inner sep=0.75pt]    {$n+2$};
\draw (12,290.4) node [anchor=north west][inner sep=0.75pt]    {$2n+2$};
\draw (319.17,136.63) node [anchor=north west][inner sep=0.75pt]  [rotate=-319.62]  {$\dotsc $};
\draw (417.15,121.74) node [anchor=north west][inner sep=0.75pt]  [rotate=-40.91]  {$\dotsc $};
\draw (293,192.4) node [anchor=north west][inner sep=0.75pt]    {$1$};
\draw (293,149.4) node [anchor=north west][inner sep=0.75pt]    {$2$};
\draw (453,150.4) node [anchor=north west][inner sep=0.75pt]    {$n$};
\draw (454,190.4) node [anchor=north west][inner sep=0.75pt]    {$n+2$};
\draw (359,250.4) node [anchor=north west][inner sep=0.75pt]    {$n+1$};

\end{tikzpicture}
    \caption{The lift of a triangulation with the corresponding convex polygon}
    \label{fig: Lift of Triangulation}
\end{figure}
\end{proof}

From Lemma \ref{lem: Counting number of triangulations}, we conclude that there are $nC_n$ small triangulations of $A_Q$, and hence families of triangulations by Lemma \ref{lem: unique small triangulation}. Since clusters are in bijection with triangulations of $A_Q^{\bm{\varepsilon}}$, we attain the following theorem.

\begin{thm} \label{thm: counting clusters}
The number of parametrized families of clusters in $\mathcal{C}_{Q^{\bm{\varepsilon}}}$ is $nC_n$. \qed
\end{thm}

\section{The Relationship Between the Exceptional and Cluster Conventions}

\indent 

As noted in Remark \ref{rem: difference between two conventions}, the cluster and exceptional conventions are indeed different. Moreover in triangulations, cycles and loops are allowed whereas they are not in fundamental arc diagrams. In this section we will illuminate precisely when the two conventions coincide when $Q$ is a quiver of type $\tilde{\mathbb{A}}_n$ with straight orientation. \\

We first wish to show that any triangulation containing an arc with both endpoints on the outer boundary circle can not be a fundamental diagram, hence must contain a cycle since no arcs cross. To do this, we need the following lemma.

\begin{lem}\label{lem: only inner outer arcs}
If $\xi$ is an exceptional collection containing a regular module whose fundamental arc diagram has no dateline crossings as in Definition \ref{defn: Universal Dateline}, then $\xi$ contains a simple regular. 
\end{lem}

\begin{proof}
Consider the strand diagram associated to $\xi$ and let $R = ij$ be a regular string module contained in $\xi$. Since there are no dateline crossings, we have $0<i<j<n+1$. If $R$ is simple there is nothing to show, so suppose $j-i\geq 2$. Since $\xi$ is exceptional, there exists some $p\in [i,j]$ such that no strand ends at $p$ for if not, the strand diagram would contain a cycle. Let $qp$ be the regular module in $\xi$ whose support is maximal with $i\leq q < p$, so $0<p-q\leq j-i+1$. Repeating this process will produce a simple regular, a regular strand connecting consecutive vertices, in finite time. 
\end{proof}

\begin{rem}
    Note that the assumption that $\xi$ contains a regular module is needed. For instance in Example \ref{exmp: families of exceptional collections}, the fundamental arc diagram in the second row of the second column of the table does not have a dateline crossing, yet also does not  contain a simple regular.
\end{rem}

\begin{prop}\label{prop: only inner outer arcs}
Any fundamental arc diagram containing an arc with both endpoints on the outer circle of the annulus is not a triangulation.
\end{prop}

\begin{proof}
As is the case with triangulations, the action of ${2\pi \over n}$ outer Dehn twists of the annulus is also well defined on fundamental arc diagrams. Since by Lemma \ref{lem: existence of nice diagram} every outer equivalence class of fundamental arc diagrams can be represented by an arc diagram with no dateline crossings, it follows from Lemma \ref{lem: only inner outer arcs} that each of these diagrams will contain a simple module, which is contractible to the boundary. After a sequence of ${2\pi\over n}$ outer clockwise Dehn twists, this arc will still be contractible to the boundary, hence every diagram in any given outer equivalence class that contains regular modules will contain one that is contractible to the boundary. That is, any small fundamental arc diagram that contains a strand with endpoints on the outer circle is not a triangulation.  
\end{proof}

Now to compare the two conventions, given a cluster $\varphi(X_{l+1}) \oplus \dots \oplus \varphi(X_{n+1}) \oplus P_1[1]\ \oplus \dots \oplus P_l[1]$, we analyze the fundamental arc diagram corresponding to the collection $(\varphi(X_{l+1}), \dots,$ $\varphi(X_{n+1}), P_1, \dots, P_l)$. This arc diagram is fundamental by a result in [\ref{ref: Clusters form exceptional collections}] that states that any cluster can be realized as an exceptional collection by simply taking the collection of all unshifted modules. The following theorem tells when the triangulation and the fundamental arc diagrams of the cluster coincide.

\begin{thm} \label{thm: when the conventions coincide}
The cluster and exceptional conventions coincide; that is, the triangulation and fundamental arc diagram of a cluster and its exceptional collection are the same if and only if the triangulation is a fundamental arc diagram.
\end{thm}

\begin{proof}
Suppose first that the triangulation is not fundamental. Then certainly the arc diagram can't coincide, for if it did, it would contradict the result that any cluster can be arranged into an exceptional sequence, and hence a collection. \\

Suppose now that the triangulation is a fundamental arc diagram. Then by Proposition \ref{prop: only inner outer arcs}, the triangulation contains only arcs that connect the inner and outer circles of the annulus. Suppose we have the arc $a(i,0)[0]$. Then this arc lifts to a vertical strand $X_i$ in the universal cover. We have that $X_i$ becomes $P_i[1]$ in the cluster, $P_i$ in the exceptional collection and $a(i,0)[0]$ in the fundamental arc diagram. Suppose now that we have an arc $a(i,0)[j]$. Then the fundamental lift will be a strand connecting $i$ on $\mathbb{R} \times \{1\}$ to $j+1$ on $\mathbb{R} \times \{0\}$. This strand will cross each of $\{X_{i+1},\dots,X_{n+1}\}$ precisely $j+1$ times and will cross each of $\{X_{1},\dots,X_{i}\}$ precisely $j$ times. The corresponding module in the cluster is the string module $i n+1_{n-i+1+j(n+1)}$. In the exceptional convention, this module is the arc $a(i,0)[j]$. Thus arcs of the form $a(i,0)[j]$ in the triangulation stay the same in the fundamental arc diagram. The proof that arcs of the form $a(0,i)[j]$ also coincide in both conventions is analogous.
\end{proof}

\noindent 
For the remainder of the paper, our focus will be on exceptional collections and the exceptional convention.

\section{Relative Projectivity in Exceptional Sets}

\indent 

In this section, we will prove that the notion of a module being relatively projective in an exceptional set is well defined; that is, the property of being relatively projective in an exceptional sequence is independent of the order. Throughout this section, let $Q$ be a quiver with $n$ vertices whose path algebra $\Bbbk Q = \Lambda$ is hereditary. In order to prove this well-definedness, we require some definitions and two lemmas.

A representation of a quiver $Q$ is \textbf{projective/injective} if its corresponding $\Bbbk Q$ module under the equivalence of categories between representations of $Q$ and modules over $\Bbbk Q$ is projective/injective respectively. Note that a representation $V$ is projective if and only if $\tau V = 0$; that is, its image under the Auslander--Reiten translate vanishes. Let $V$ be a representation of $Q$. The \textbf{right perpendicular category} of $V$, denoted by $V^{\perp}$, is the full subcategory of rep$_{\Bbbk}(Q)$ of representations $W$ such that Hom$(V,W) = 0 =$ Ext$(V,W)$. The \textbf{left perpendicular category} of $V$, denoted $^{\perp\!}V$, is defined analogously. A module $M$ is \textbf{relatively projective with respect to $V$} if $M$ is projective in $V^{\perp}$. Dually, $M$ is \textbf{relatively injective with respect to $V$} if $M$ is injective in $^{\perp\!}V$. Given an exceptional sequence of representations of $Q$, say $\xi = (E_1,E_2,\dots,E_n)$, we say that $E_i$ is \textbf{relatively projective in $\xi$} if $E_i$ is relatively projective with respect to $E_{i+1} \oplus \cdots \oplus E_n$. Dually, we say that $E_i$ is \textbf{relatively injective in $\xi$} if $E_i$ is relatively injective with respect to $E_{1} \oplus \cdots \oplus E_{i-1}$.

\begin{lem} \label{lem: Can't switch X and Y}
Let $(E_1, E_2, \dots, E_{n-2},Y,X)$ be an exceptional sequence of $\Bbbk Q$-modules where $Y$ is relatively projective with respect to $X$ but not projective. Then $(X,Y)$ is not an exceptional pair.
\end{lem}

\begin{proof}
Since $Y$ is not projective, its image under the Auslander--Reiten translate does not vanish. Moreover, since $Y$ is projective in $X^{\perp}$, $\tau Y \not\in X^{\perp}$. Thus either Hom$(X,\tau Y) \neq 0$ or Ext$(X,\tau Y)\neq0$, which, by Auslander--Reiten duality, implies  Ext$(Y,X) \neq 0$ or Hom$(Y,X)\neq 0$ respectively.
\end{proof}

\begin{lem} \label{lem: if both exceptional, then can switch}
Suppose two consecutive terms in an exceptional sequence commute; that is, suppose both $(E_1, E_2, \dots , E_j, X, Y, E_{j+3}, \dots E_n)$ and $(E_1, E_2, \dots , E_j, Y, X, E_{j+3}, \dots E_n)$ are exceptional sequences. Then $Y$ is relatively projective in the first if and only if it is relatively projective in the second.
\end{lem}

\begin{proof}
Let $E = E_{j+3} \oplus \dots \oplus E_n$. Since $Y$ is relatively projective in the first sequence, $Y$ is projective in $E^{\perp}$ by definition. Therefore $Y$ is projective in the smaller category $X^{\perp} \cap E^{\perp}$, implying that $Y$ is relatively projective in the second sequence. Conversely, suppose $Y$ is not projective in $E^{\perp}$ but is in $X^{\perp} \cap E^{\perp}$. Then by Lemma \ref{lem: Can't switch X and Y}, either Hom$(X,\tau Y) \neq 0$ or Ext$(X,\tau Y)\neq0$, contradicting the exceptionality of the sequence.
\end{proof}

We are now ready to prove the main result of this section.

\begin{thm}\label{thm: rel projective is independent of order}
The property of being relatively projective in an exceptional sequence is independent of the order, and thus the relatively projective objects of an exceptional set form a well-defined subset.
\end{thm}

\begin{proof}
Suppose $\{E_i\} = \{E_1, E_2, \dots, E_n\}$ is an exceptional collection. We define a partial order $\prec$ on $\{E_i\}$ given by the transitive closure of the relation $E_i \prec E_j$ if either Hom$(E_i,E_j) \neq 0$ and/or Ext$(E_i,E_j) \neq 0$. If $Y$ is another module in the collection, let $E_1, E_2, \dots, E_k$ be the elements in the collection that are strictly greater than $Y$. By definition, these elements must follow $Y$ in any exceptional sequence given by rearranging the $E_i$. Set $E = E_1 \oplus E_2 \oplus \dots \oplus E_k$. Note that Hom$(E,E_j)=0=$ Ext$(E,E_j)$ for any other $E_j$ in the collection since, otherwise, $Y\prec E_j$. For any ordering of the $E_i$ that yields an exceptional sequence, we will show that $Y$ is relatively projective in that sequence if and only if it is projective in $E^{\perp}$. \\

Indeed, the summands of $E$ must appear after $Y$ in any exceptional sequence. Let $X_1, X_2, \dots, X_m$ be the other objects in the exceptional sequence that follow $Y$. Then the first one, $X_1$, must commute with $Y$ and all the $E_i$ between $Y$ and $X_1$ as Hom$(Y\oplus E,X_1) = 0$ and Ext$(Y\oplus E,X_1) = 0$. Therefore we can move $X_1$ to the left of $Y$ in the sequence. By Lemma \ref{lem: if both exceptional, then can switch}, $Y$ is relatively projective in the new sequence if and only if it was relatively projective in the original sequence. It follows by induction on $m$ that we can move all $X_i$ to the left of $Y$ without changing its relative projectivity. After doing so, only the summands of $E$ follow $Y$. Therefore $Y$ was relatively projective in the original sequence if and only if $Y$ is projective in $E^{\perp}$. Since this condition is independent of the order of the exceptional sequence, the theorem follows.
\end{proof}

Recall that there is a duality $D=\text{Hom}_{\Bbbk}(-,\Bbbk)$ between the categories of representations of a quiver $Q$ and representations of the opposite quiver $Q^{op}$. By the dual object, we mean the image of the object under this duality. There is a symmetry between relatively projective and relatively injective objects in exceptional sequences.

\begin{lem}\label{lem: duality of exceptional sequences}
$(E_1,\dots,E_n)$ is an exceptional sequence for a hereditary algebra $\Lambda$ if and only if the dual objects $(DE_n,\dots,DE_1)$ form an exceptional sequence for $\Lambda^{op}$. Furthermore, $E_i$ is relatively projective or injective if and only if $DE_i$ is relatively injective or projective, respectively.
\end{lem}

\begin{proof}
This follows from the fact that ${\rm Hom}_\Lambda(X,Y)\cong {\rm Hom}_{\Lambda^{op}}(DY,DX)$ and, similarly, ${\rm Ext}_\Lambda(X,Y)\cong {\rm Ext}_{\Lambda^{op}}(DY,DX)$.
\end{proof}

This symmetry together with Theorem \ref{thm: rel projective is independent of order} imply the following.

\begin{cor}\label{cor: rel injectives are independent of order}
The property of being relatively injective in an exceptional sequence is independent of the order, and thus the relatively injective objects of an exceptional set form a well-defined subset.\qed
\end{cor}

\section{Counting Exceptional Sets} \label{sec: counting exceptional sets}

\indent

In this section we will count the number of exceptional sets in the Dynkin case and families of exceptional sets in the affine case for quivers of type $\mathbb{A}$ and $\tilde{\mathbb{A}}$ with straight orientation respectively. We begin with the Dynkin case.
\subsection{$\mathbb{A}_n$ Straight Orientation}

\indent 

Throughout this subsection, let $Q$ represent a quiver of type $\mathbb{A}_n$ with straight orientation. Let $\Bbbk Q$ denote its path algebra over $\Bbbk$. We use the notation $M_{i,j}$ for the $\Bbbk Q$-module with support $(i,j]$ corresponding to the strand $c(i,j)$. Thus, $M_{0,k}$ are the injective objects, $M_{k,n}$ are the projective objects and the length of $M_{i,j}$ is $j-i$. We will show that the number of exceptional sets over $\Bbbk Q$ is counted by the generalized Catalan numbers, also known as the Rothe numbers, the Rothe--Hagan coefficients of the first type, and the $k$-Catalan numbers. For more on these numbers see [\ref{ref: Gould Generalized Catalan Numbers}] and [\ref{ref: Master's Thesis Generalized Catalan Numbers}], or Section \ref{sec: affine A exceptional collections} below. Throughout, let $E(n)$ denote the number of exceptional sets over $\Bbbk Q$. We begin with a pair of lemmas, followed by the Rothe recursion result. Note that the following lemma has also been proven in [\ref{ref: Igusa and Sen Rooted Labeled Forests}].

\begin{lem} \label{lem: existence of gap and maximal injective}
In any exceptional collection over $\Bbbk Q$ there exists at least one injective object and moreover, in the support of the injective object with maximal support, there exists a unique vertex at which no other object in the exceptional set has support.
\end{lem}

\begin{proof}
The injective objects are those having $1$ in their support. Thus, if an exceptional set has no injectives, then vertex $1$ would not be in the support of the set and we could add $I_1 = M_{0,1}$ to the set. So let $k$ be maximal so that $M_{0,k}$ is in the set. The corresponding strand connects $(x_0,0)$ to $(x_k,0)$. There must be at least one vertex at which no module in the set has support, for if not, there would be a strand connecting each point between $(x_0,0)$ and $(x_k,0)$, hence creating a cycle. On the other hand, if there were two vertices say $i,j \in [1,k]$ that are not in the support of any module in the set, we could add the simple module $M_{i-1,i}$ or the module $M_{j-1,j}$ to the set without changing exceptionality.
\end{proof}

\begin{lem}\label{lem: recursion for g(z)} {\color{white} .} 
$E(0) = E(1) = 1$ and $E(n) = \displaystyle\sum_{\substack{a+b+c = n-1\\a,b,c\,\geq 0}} E(a)E(b)E(c)$. Equivalently, the generating function $g(z) = \displaystyle\sum_{n\geq0} E(n)z^n$ satisfies $g=zg^3+1$.
\end{lem} 

\begin{proof}
In this proof we will show that the modules in an exceptional set can only have support in three possible subsets of $[1,n]$. By Lemma \ref{lem: existence of gap and maximal injective}, let $k$ be maximal so that $M_{0,k}$ is in the set. Now the other objects in the exceptional set must have support in $[1,k]$ or have support disjoint from $[1,k]$. We can see this from the fact that in the corresponding strand diagram, since we have straight orientation, any strand beginning before $k$ and ending after $k$ must cross the strand corresponding to $M_{0,k}$, hence these objects can not occur in the same exceptional set at $M_{0,k}$. Again by Lemma \ref{lem: existence of gap and maximal injective}, there is a unique $p\in[1,k]$ which is not in the support of any object in the set other than $M_{0,k}$. \\

We have that $1 \leq p \leq k \leq n$. Let $a = p-1$, $b = k-p$, and $c = n-k$. Then $a,b,c \geq 0$, $a + b + c = n-1$, and $k = a+b+1$. Thus the exceptional set consists of $M_{0,k}$ and a subset of $A\coprod B\coprod C$ where
\begin{enumerate}
    \item $A$ is the set of objects with support in $[1,a]$,
    \item $B$ is the set of objects with support in $[a+2,k]$ and
    \item $C$ is the set of objects with support in $[k+1,n]$.
\end{enumerate}
Since an exceptional sequence can have at most $a, b$, and $c$ objects in $A, B$, and $C$, respectively, and $a+b+c = n-1$, we have the maximum. We conclude that our exceptional set consists of $M_{0,k}$ and exceptional sets in $A$, $B$, and $C$. The number of such sets are $E(a)$, $E(b)$, and $E(c)$ respectively. We conclude that the number of possible exceptional sets is 
$$\sum_{k=1}^n \sum_{p=1}^k E(p-1)E(k-p)E(n-k) = \displaystyle\sum_{\substack{a+b+c = n-1\\a,b,c\,\geq 0}} E(a)E(b)E(c).$$
\end{proof}

It is known that the generalized Catalan numbers satisfy the Rothe recursion given in the following proposition [\ref{ref: Rothe Convolution}].

\begin{prop}
For $k\geq 1$, the recursion $A_0(1,k) = A_1(1,k) = 1$ and $$A_n(1,k) = \displaystyle\sum_{\substack{a_1+a_2+\dots + a_k = n-1\\(a_1,a_2,\dots, a_k)\geq 0}} \displaystyle\prod_{i} A_{a_i}(1,k)$$
has unique solution $A_n(1,k) = {1\over kn+1}\binom{kn+1}{n}$.
\end{prop}

Combining the Rothe convolution with Lemma \ref{lem: recursion for g(z)} yields the following theorem.

\begin{thm} \label{thm: counting exceptional sets type A}
The number of exceptional sets of type $\mathbb{A}_n$ with straight orientation is $E(n) ={1\over{3n+1}}\binom{3n+1}n= {1\over 2n+1}\binom{3n}{n}. \hfill \square$
\end{thm}

\subsection{Ternary trees and relative projectives and injectives}

Using Lemma \ref{lem: recursion for g(z)} and its recursive proof we easily see that exceptional sets for $\mathbb{A}_n$ with straight orientation are in bijection with rooted ternary trees with $n$ nodes and the nodes are in bijection with the objects of the exceptional collection. Figures \ref{Fig: strand diagram for Fig 1} and \ref{Fig: 3-nary tree} give an example. Strands labeled $Xa,Xb,Xc$ by Definition \ref{def: Xa, Xb, Xc} give the children of strand $X$ in the ternary tree. If there is no strand labeled $Xx$, we put a leaf in the ternary tree in place of node $Xx$ for $x=a,b,c$. E.g., in Figure \ref{Fig: strand diagram for Fig 1}, $Aba,Abc$ do not exist. So, these are leaves in Figure \ref{Fig: 3-nary tree}.

%
\begin{figure}[htbp]
\begin{center}
\begin{tikzpicture}
\begin{scope} 
\draw[thick,blue] (0,0)..controls (0,-.5) and (1,-.5)..(1,0)
(.5,0) node{$Aa$}; 
\draw[thick,blue] (0,0)..controls (0,-1.25) and (4.3,-1.75)..(4,0)
(1.85,-.85) node{$A$}; 
\draw[thick,blue] (0,0)..controls (-.5,-2.5) and (11.5,-2.5)..(11,0)
(5.5,-1.5) node{$R$}; 
\end{scope}
\begin{scope}[xshift=2cm] 
\draw[thick,blue] (0,0)..controls (0,-.5) and (1,-.5)..(1,0)
(.5,0) node{$Abb$}; 
\draw[thick,red] (0,0)..controls (0,-1) and (2,-1)..(2,0)
(1,-.5) node{$Ab$}; 
\end{scope}
\begin{scope}[xshift=4cm] 
\draw[thick,blue] (0,0)..controls (0,-.5) and (1,-.5)..(1,0)
(.5,0) node{$Ac$}; 
\end{scope}
\begin{scope}[xshift=6cm] 
\draw[thick,red] (0,0)..controls (0,-.5) and (1,-.5)..(1,0)
(.5,0) node{$Bc$}; 
\draw[thick,red] (1,0)..controls (.7,-1.75) and (5,-1.25)..(5,0)
(3.15,-.85) node{$B$}; 
\end{scope}
\begin{scope}[xshift=8cm] 
\draw[thick,red] (0,0)..controls (0,-.5) and (1,-.5)..(1,0)
(.5,0) node{$Bbb$}; 
\draw[thick,blue] (-1,0)..controls (-1,-1) and (1,-1)..(1,0)
(0,-.5) node{$Bb$}; 
\end{scope}
\begin{scope}[xshift=10cm] 
\draw[thick,red] (0,0)..controls (0,-.5) and (1,-.5)..(1,0)
(.5,0) node{$Ba$}; 
\end{scope}
\begin{scope}[xshift=12cm] 
\draw[thick,red] (0,0)..controls (0,-.5) and (1,-.5)..(1,0)
(.5,0) node{$Cbc$}; 
\end{scope}
\begin{scope}[xshift=13cm] 
\draw[thick,red] (0,0)..controls (0,-.5) and (1,-.5)..(1,0)
(.5,0) node{$Cb$}; 
\draw[thick,blue] (-2,0)..controls (-1.7,-1.3) and (1,-1.3)..(1,0)
(-.5,-.75) node{$C$}; 
\end{scope}
\begin{scope}[xshift=14cm] 
\draw[thick,blue] (0,0)..controls (0,-.5) and (1,-.5)..(1,0)
(.5,0) node{$Cc$}; 
\end{scope}
\foreach \x in {0,1,...,15}
\draw[fill] (\x,0) circle[radius=1mm] (\x,.4) node{\x};
\end{tikzpicture}
\caption{A strand diagram for $A_{15}$ with labels according to Definition \ref{def: Xa, Xb, Xc}. This is a rooted tree with root 0. The longest injective $M_{0,11}$ gives the strand $R=c(0,11)$. The other injectives are $A,Aa$ (and $Aaa,Aaaa,$ etc if these were strands). Relative injectives, oriented right, are in blue. In Lemma \ref{lem: recursion for g(z)}, $k=11,p=6, a=5,b=5,c=4$. Note that the $A$ and $B$ branches are mirror images with opposite colors.
}
\label{Fig: strand diagram for Fig 1}
\end{center}
\end{figure}
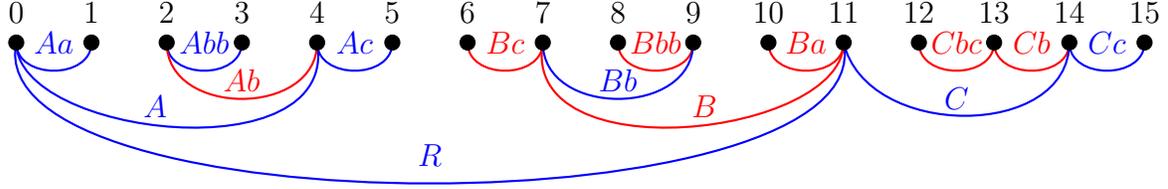
%

Since the strand diagram has the maximum number of strands with vertex set $\{0,1,\cdots,n\}$ and has no cycles, it is a tree with $n$ edges. We take 0 to be root. Then each strand has an orientation pointing away from the endpoint closer to the root 0 in the tree. For example, the strand $c(i,j)$ is oriented to the right if and only if the unique path from the root 0 to $j$ passes through $i$. We now explain the labeling of our strands. This labeling uses the Hasse diagram of the exceptional set partially ordered by inclusion of supports (Figure \ref{Fig: support Hasse diagram}). We call this the {\bf support Hasse diagram}. Thus $M_{i,j}\le M_{p,q}$ if $(i,j]\subseteq (p,q]$. As in any partially ordered set, we say that $A$ {\bf covers} $B$ if $A>B$ and there is no $C$ so that $A>C>B$.

\begin{defn}\label{def: Xa, Xb, Xc}
Let $X=c(i,j)$ be a strand in our strand diagram corresponding to an object $M_{i,j}$ of an exceptional collection. Suppose that $X$ is oriented to the left so that $j$ is closer to the root than $i$. Then, the strands $Xa,Xb,Xc$, if they exist, are given as follows.
\begin{enumerate}
\item $Xa=c(x,j)$ where $i<x<j$ is minimal so that $M_{x,j}$ is an element of the exceptional collection. ($Xa$ is a leaf, the left child of $X$ in the ternary tree if no such $x$ exists.)
\item $Xb=c(i,y)$ is the largest strand so that $i<y<j$.
\item $Xc=c(z,i)$ is the largest strand so that $z<i$.
\end{enumerate}
Figure \ref{Fig: strand diagram for Fig 1} illustrates this with $X=B$. When $X=c(i,j)$ is oriented to the right, the definitions are analogous: $Xa=c(i,x)$ with $x<j$ maximal, $Xb=c(y,j)$ is maximal so that $i<y<j$ and $Xc=c(j,z)$ with $z>j$ maximal. (Figure \ref{Fig: strand diagram for Fig 1} illustrates this with $X=A,C$.) 
\end{defn}

In both cases of Definition \ref{def: Xa, Xb, Xc}, $Xa,Xc$ have the same orientation as $X$ and $Xb$ has the opposite orientation. $Xa,Xb$ are children of $X$ in the support Hasse diagram but $Xc$ is a sibling of $X$. ($Xc,X$ have the same parent or they both have maximal support.) Note that, for $X=R$ we have, $A=Ra$, $B=Rb$ and $C=Rc$.

The notation of Definition \ref{def: Xa, Xb, Xc} gives a ternary tree with root $R$. For any node $X$, its three children are, from left to right, $Xa,Xb,Xc$. See Figure \ref{Fig: 3-nary tree} where leaves are not labeled.

%
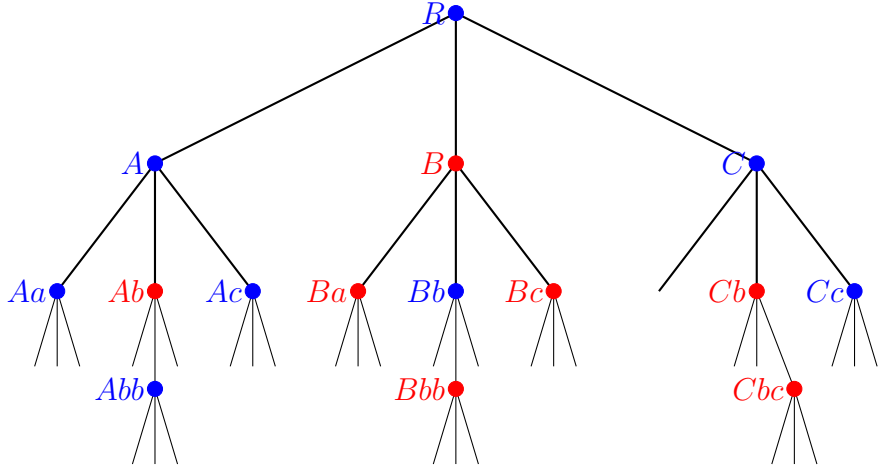
\begin{figure}[htbp]
\begin{center}
\begin{tikzpicture}
\coordinate (R) at (0,2);
\coordinate (A) at (-4,0);
\coordinate (B) at (0,0);
\coordinate (C) at (4,0);
\coordinate (Aa) at (-5.3,-1.7);
\coordinate (Ab) at (-4,-1.7);
\coordinate (Abb) at (-4,-3);
\coordinate (Ac) at (-2.7,-1.7);
\coordinate (Ba) at (-1.3,-1.7);
\coordinate (Bb) at (0,-1.7);
\coordinate (Bbb) at (0,-3);
\coordinate (Bc) at (1.3,-1.7);
\coordinate (Ca) at (2.7,-1.7);
\coordinate (Cb) at (4,-1.7);
\coordinate (Cbc) at (4.5,-3);
\coordinate (Cc) at (5.3,-1.7);
\begin{scope} 
\begin{scope}[xshift=-1.3cm,yshift=-1.7cm] 
\draw (0,0)--(0,-1) (-.3,-1)--(0,0)--(.3,-1);
\end{scope}
\begin{scope}[yshift=-1.7cm] 
\draw (0,0)--(0,-1.3) (-.3,-1)--(0,0)--(.3,-1);
\begin{scope}[yshift=-13mm] 
\draw (0,0)--(0,-1) (-.3,-1)--(0,0)--(.3,-1);
\end{scope}
\end{scope}
\begin{scope}[xshift=1.3cm,yshift=-1.7cm] 
\draw (0,0)--(0,-1) (-.3,-1)--(0,0)--(.3,-1);
\end{scope}
\end{scope}
\begin{scope}[xshift=-4cm] 
\begin{scope}[xshift=-1.3cm,yshift=-1.7cm] 
\draw (0,0)--(0,-1) (-.3,-1)--(0,0)--(.3,-1);
\end{scope}
\begin{scope}[yshift=-1.7cm] 
\draw (0,0)--(0,-1.3) (-.3,-1)--(0,0)--(.3,-1);
\begin{scope}[yshift=-13mm] 
\draw (0,0)--(0,-1) (-.3,-1)--(0,0)--(.3,-1);
\end{scope}
\end{scope}
\begin{scope}[xshift=1.3cm,yshift=-1.7cm] 
\draw (0,0)--(0,-1) (-.3,-1)--(0,0)--(.3,-1);
\end{scope}
\end{scope}
\begin{scope}[xshift=4cm] 
\begin{scope}[xshift=-1.3cm,yshift=-1.7cm] 
\end{scope}
\begin{scope}[yshift=-1.7cm] 
\draw (0,0)--(0,-1) (-.3,-1)--(0,0)--(.5,-1.3);
\begin{scope}[yshift=-13mm,xshift=5mm] 
\draw (0,0)--(0,-1) (-.3,-1)--(0,0)--(.3,-1);
\end{scope}
\end{scope}
\begin{scope}[xshift=1.3cm,yshift=-1.7cm] 
\draw (0,0)--(0,-1) (-.3,-1)--(0,0)--(.3,-1);
\end{scope}
\end{scope}
\draw[thick] (Aa)--(A)--(R) (Ab)--(A)--(Ac);
\draw[thick] (Ba)--(B)--(R) (Bb)--(B)--(Bc);
\draw[thick] (Ca)--(C)--(R) (Cb)--(C)--(Cc);
\foreach \x in {R,A,Aa,Ac,Bb,C,Cc,Abb}
\draw[blue,fill] (\x) circle[radius=1mm] node[left]{$\x$} ;
\foreach \x in {B,Ab,Ba,Bc,Cb,Cbc,Bbb}
\draw[red,fill] (\x) circle[radius=1mm] node[left]{$\x$} ;
\end{tikzpicture}
\caption{This is the ternary tree corrsponding to Figure \ref{Fig: strand diagram for Fig 1}. Each node $X$ has three children $Xa,Xb,Xc$ given by Definition \ref{def: Xa, Xb, Xc}. Strands which are blue in Figure \ref{Fig: strand diagram for Fig 1} give blue nodes here. Blue nodes are those whose labels have an even number of $B$'s.}
\label{Fig: 3-nary tree}
\end{center}
\end{figure}
%

\begin{lem}\label{lem: bijection with 3-nary trees}
Exceptional sets on $\mathbb{A}_n$ are in bijection with rooted ternary trees with $n$ nodes (and $2n+1$ leaves) where each node is labeled with an object of the exceptional set.
\end{lem}

\begin{proof} Definition \ref{def: Xa, Xb, Xc} gives the formula for the ternary tree with node labels associated to each strand diagram. Referring to the proof of Lemma \ref{lem: recursion for g(z)}, for any exceptional collection for $\mathbb{A}_n$ with $n\ge1$, the root $R$ of the ternary tree corresponds to the maximal injective module $M_{0,k}$ in the exceptional set. The children $A=Ra,B=Rb,C=Rc$ are as given in Definition \ref{def: Xa, Xb, Xc}. The rest is given by induction on $n$ where we note that the descendants of $B$ are in reverse lateral order since, unlike $R,A$ and $C$, $B$ is oriented to the left.
\end{proof}

We use the support Hasse diagram to determine the positions in the ternary tree occupied by the relatively injective and relatively projective objects of an exceptional collection on linear $\mathbb{A}_n$. An exceptional sequence is an ordered exceptional set. Such an ordering makes the support Hasse diagram into a rooted labeled forest. We then use the following theorem proved in [\ref{ref: Igusa and Sen Rooted Labeled Forests}]. 

\begin{lem}\label{lem: rooted labeled forests}
Objects in an exceptional sequence are both relatively projective and relatively injective if any only if they are roots of the rooted forest (objects with maximal support). A vertex which is not a root corresponds to a relatively projective, resp injective, object if and only if it comes before, resp. after, its parent in the exceptional sequence.
\end{lem}

%
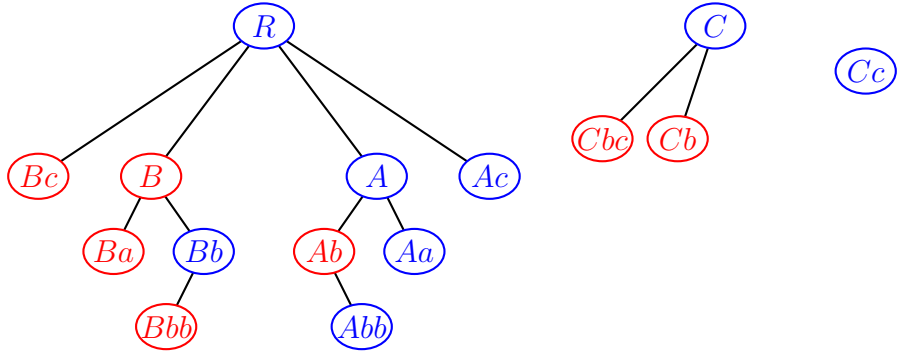
\begin{figure}[htbp]
\begin{center}
\begin{tikzpicture}
\clip (0,-2.4) rectangle (12,2.4);
\coordinate (R) at (3.5,2);
\coordinate (C) at (9.5,2);
\coordinate (Cc) at (11.5,1.4);
\coordinate (Bc) at (.5,0);
\coordinate (B) at (2,0);
\coordinate (A) at (5,0);
\coordinate (Ac) at (6.5,0);
\coordinate (Cbc) at (8,0.5);
\coordinate (Cb) at (9,0.5);
\coordinate (Ba) at (1.5,-1);
\coordinate (Bb) at (2.7,-1);
\coordinate (Ab) at (4.3,-1);
\coordinate (Aa) at (5.5,-1);
\coordinate (Bbb) at (2.2,-2);
\coordinate (Abb) at (4.8,-2);
\draw[thick] (Bc)--(R)--(B)--(Ba) (B)--(Bb)--(Bbb)
;
\draw[thick] (Ac)--(R)--(A)--(Aa) (A)--(Ab)--(Abb)
;
\draw[thick] (Cbc)--(C)--(Cb)
;
\foreach \x in {A,Ac,R,C,Cc,Bb,Aa,Abb,Bc,B,Cbc,Cb,Ba,Ab,Bbb}
\draw[fill,white] (\x) ellipse[x radius=4mm,y radius=3mm];
\foreach \x in {A,Ac,R,C,Cc,Bb,Aa,Abb}
\draw[blue,thick] (\x) ellipse[x radius=4mm,y radius=3mm] node{$\x$};
\foreach \x in {Bc,B,Cbc,Cb,Ba,Ab,Bbb}
\draw[red,thick] (\x) ellipse[x radius=4mm,y radius=3mm] node{$\x$};
\end{tikzpicture}
\caption{Support Hasse diagram for Figure \ref{Fig: strand diagram for Fig 1}. Left children (those which come before their parents in the exceptional sequence) are red. These are relatively projective but not relatively injective. The rest: the roots, $R,C,Cc$, and the right children (which come after their parents), are relatively injective and colored blue.}
\label{Fig: support Hasse diagram}
\end{center}
\end{figure}
%

\begin{prop}\label{prop: rel injectives are oriented left}
A strand $c(i,j)$ corresponds to a relatively injective object $M_{i,j}$ in the exceptional set if and only if it is oriented to the right.
\end{prop}

By duality (Lemma \ref{lem: duality of exceptional sequences}) this is equivalent to the following.

\begin{cor}\label{cor: rel projectives are oriented right}
A strand $c(i,j)$ corresponds to a relatively projective object in the exceptional set if and only if the unique path from $i$ to $n$ in the strand diagram passes through $j$.
\end{cor}

\begin{proof}[Proof of Proposition \ref{prop: rel injectives are oriented left}]
Let $k_0=0$ and let $k_1>k_0$ be maximal so that $c(k_0,k_1)$ is a strand in the strand diagram. This is $R=c(0,k)$. Let $k_2>k_1$ be maximal so that $C=c(k_1,k_2)$ is a strand, and so on. Then $c(k_i,k_{i+1})$ form a sequence or strands going from 0 to $n$. These are all oriented right, giving a path from 0 to $n$, and they give the roots of the support Hasse diagram and are thus both relatively projective and relatively injective. 

Next, consider a strand whose support is not maximal. This will be either $Xa,Xb$ or $Xc$ for another strand $X=c(i,j)$. Suppose first  $X=c(i,j)$ is oriented to the right. Then $Xa=c(i,x)$ is also oriented right and comes after its parent $X$. So, $Xa$ is relatively injective by Lemma \ref{lem: rooted labeled forests}. $Xb=c(y,j)$ is oriented left and comes before its parent $X$ since $M_{yj}\subset M_{ij}$. So, $Xb$ is relatively projective and not relatively injective. Lastly, $Xc=c(j,z)$ is oriented to the right and Ext$(X,Xc)\neq0$. So, either $Xc$ is a root (with maximal support), in which case it is relative injective and relatively projective, or $Xc,X$ have the same parent $Z$ in which case $X$, being oriented right, is relatively injective by downward induction on size and comes after its parent $Z$. Since $Xc$ comes after $X$, it also comes after its parent $Z$. So, $Xc$ is also relatively injective. 

The case when $X$ is oriented left is similar, but easier since, in that case, $X$ cannot be a root, eliminating one subcase. This proves the proposition in all cases.
\end{proof}

\begin{thm}\label{thm: rel injectives are blue}
A node in the ternary tree corresponds to a relatively injective object in the exceptional set if and only if it has an even number of letter $b$'s in its label. For example, $Cc$ and $Bb$ are relatively injective, but $Cb,Bbb$ are not.
\end{thm}

\begin{proof}
This follows from Proposition \ref{prop: rel injectives are oriented left} by induction on the length of the label since $R,A,C$ point to the right, but $B$ is oriented left and $Xa,Xc$ point in the same direction as $X$ but $Xb$ points in the opposite direction from $X$.
\end{proof}

\begin{cor}\label{cor: recursion for rel inj}
For $n,m\ge0$, let $N_{n,m}$ denote the number of exceptional sets on $A_{n+m}$ having $n$ relatively injective elements. Then the $N_{n,m}$ are given by the following recursion. $N_{0,0}=1$, $N_{0,k}=0$ for $k>0$ and, for $n\ge1$ we have:
\begin{equation}\label{eq: recursion for Nnm}
	N_{n,m}=\sum_{\substack{a+b+c=n-1\\ i+j+k=m\\ a,b,c,i,j,k\ge0}} N_{a,i}N_{j,b}N_{c,k}.
\end{equation}
The same recursion holds if ``injective'' is replaced with ``projective''.
\end{cor}

\begin{proof}
Since the root $R$ is relatively injective, $n$ must be $\ge1$ if $n+m\ge1$. This makes $N_{0,k}=0$ for $k>0$. After the root $R$, the remaining $n-1$ relatively injective and $m$ other objects must be distributed among the three subtrees with roots $A,B,C$, say $a,b,c$ relative injectives and $i,j,k$ others nodes for these subtrees. The $A$ and $C$ subtrees are like the $R$-tree and there are $N_{a,i}$, $N_{c,k}$ such trees. But the $B$-subtree has the colors reversed in Figure \ref{Fig: 3-nary tree}. So, there are $N_{j,b}$ such subtrees. This gives the recursion.
\end{proof}

In terms of the two variable generating function
\[
	f(x,y)=\sum_{n,m\ge0} N_{n,m}x^ny^m
\]
the recursion is given by:
\[
	f(x,y)=1+xf(x,y)^2f(y,x).
\]
For one variable $g(z)=f(z,z)$ this recovers the previous recursion: $g(z)=1+zg(z)^3$. For $m=0$, the solution is the Catalan number:
\[
	N_{n,0}=\frac1{n+1}\binom{2n}n=C_n.
\]


Exceptional collections for $\mathbb{A}_n$ are also in bijection with \emph{lattice paths of height $n$} which are defined to be paths from $(0,0)$ to $(2n+1,n)$ consisting of edges between lattice points in the plane (points with integer coordinates) where each step is either up or right of unit length and so that $x\le 2y$ except at the last step where $x=2y+1$. We also need the notion of a \emph{lattice subpath of height $h$} which is a portion of the lattice path going from some point $(a,b)$ to a point $(a+2h+1,b+h)$ which stays above the line $x-a=2(y-b)+1$ until the last step, i.e., this is the translation by $(a,b)$ of a lattice path of height $h$. We say that the lattice subpath is \emph{even or odd} depending on the parity of $a$.

\begin{thm}\label{thm: lattice paths and relative injectives}
Exceptional collections for $\mathbb{A}_n$ are in bijection with lattice paths of height $n$. The edges of the lattice paths are given by the vertices of the corresponding ternary tree with vertical edge $R$ followed by other vertices in lexicographic order with leaves of the tree giving the $2n+1$ horizontal edges and nodes giving the $n$ vertical edges. Thus, the $n$ vertical edges correspond to the objects of the exceptional collection. Furthermore, the relatively injective objects are those vertical edges whose first coordinates are even integers.
\end{thm}

\begin{proof} The proof is by induction on $n\ge1$. The unique ternary tree with one node $R$ has three leaves $A,B,C$ and these four vertices correspond to the four steps in the unique lattice path of height 1. The unique vertical step in this lattice path, $R$ going from $(0,0)$ to $(0,1)$, corresponds to the unique object in the exceptional collection of $A_1$ and this is relatively injective and thus blue in the tree and in the lattice path. So, the theorem holds for $n=1$.
%
\begin{center}
\begin{tikzpicture}

\coordinate (R0) at (0,0);
\coordinate (R1) at (1,0);
\coordinate (R2) at (0,1);
\coordinate (A1) at (1,1);
\coordinate (B1) at (2,1);
\coordinate (C1) at (3,1);
\coordinate (A) at (.5,1);
\coordinate (B) at (1.5,1);
\coordinate (C) at (2.5,1);
\coordinate (R) at (0,.5);
\draw (R1)--(C1);

\draw[very thick,blue] (R0)--(R2);
\draw[blue,thick] (R2)--(A1) (B1)--(C1);
\draw[red,thick] (A1)--(B1);
\draw[blue] (R)node[left]{$R$};
\draw[blue] (A)node[below]{\tiny$A$};
\draw[blue] (C)node[below]{\tiny$C$};
\draw[red] (B)node[below]{\tiny$B$};

\foreach \x in {R0,R1,R2,C1}
\draw[fill,blue] (\x) circle[radius=2pt];
\foreach \x in {A1,B1}
\draw[fill,red] (\x) circle[radius=2pt];
\end{tikzpicture}
\end{center}
%

For $n\ge2$, we recall that ternary trees with $n$ nodes consist of the root $R$ together with three subtrees with roots $A,B,C$ which are the children of $R$. The children of any node $X\neq R$ are labeled $Xa,Xb,Xc$ and the relatively injective objects are those with an even number of $b$'s in their label. These are colored blue and the other are colored red. Thus, subtrees $A,C$ have the same color scheme as the whole tree but subtree $B$ has the opposite color scheme. See Figure \ref{Fig: 3-nary tree} and the bottom right of Figure \ref{Fig: lattice path to ternary tree bijection}.

Any lattice path of height $n\ge2$ consists of a vertical edge $R$ followed by three lattice subpaths of heights $a,b,c$ where $a+b+c=n-1$. The reason for this is simple. In order to get from the point $(0,1)$ at the top of edge $R$ to the line $L_1:x=2y+1$, the path must pass the lines $L_{-1}:x=2y-1$ and $L_0:x=2y+0$. If $a,b,c\ge0$ are the heights of these lattice subpaths, we must have $a+b+c=n-1$. Since $a,b,c<n$, we have by induction that these lattice subpaths correspond bijectively to ternary trees with $a,b,c$ nodes respectively.

The first and third lattice subpaths are even since they start at $(0,1)$, $(2a+2b+2,a+b+1)$ with even first coordinates. Thus the red/blue color schemes are the same as for the whole tree. However, the second lattice subpath is odd since it starts at $(2a+1,a+1)$. So, the parity of the first coordinates in the second lattice subpath is reversed. This matched the reversal of the colors in the corresponding subtree. So, the blue nodes in the ternary tree correspond exactly to those vertical edges of the lattice path with even first coordinates. This proves the theorem.
\end{proof}

%
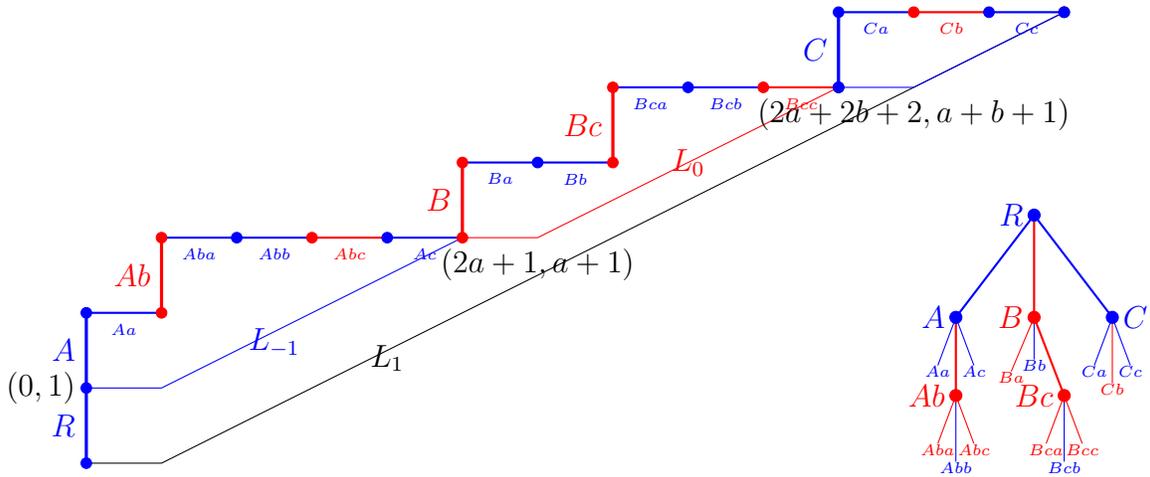
\begin{figure}[htbp]
\begin{center}

\begin{tikzpicture}

\begin{scope}[yshift=3.3cm,xshift=11cm,scale=.8] 
\begin{scope}[xshift=2cm] 
\begin{scope}[xshift=-1.3cm,yshift=-1.7cm] 
	\draw[thick,red] (0,0)--(0,-1.3);
	\draw[blue] (-.3,-.9) node{\tiny$Aa$};
	\draw[blue] (.3,-.9) node{\tiny$Ac$};
	\draw[blue] (-.3,-.8)--(0,0)--(.3,-.8);
	\begin{scope}[yshift=-13mm,xshift=0mm] 
		\draw[blue] (0,0)--(0,-1.1) ;
		\draw[red] (-.3,-.8)--(0,0)--(.3,-.8);
		\draw[red] (-.3,-.9) node{\tiny$Aba$};
		\draw[red] (.3,-.9) node{\tiny$Abc$};
		\draw[blue] (0,-1.2) node{\tiny$Abb$};
	\end{scope}
\end{scope}
\begin{scope}[yshift=-1.7cm] 
	\draw[blue] (0,0)--(0,-.7) ;
	\draw[red]  (-.38,-.9)--(0,0);
	\draw[red,thick]  (0,0)--(.5,-1.3);
	\draw[red] (-.38,-1) node{\tiny$Ba$};
	\draw[blue] (0,-.8) node{\tiny$Bb$};
	\begin{scope}[yshift=-13mm,xshift=5mm] 
		\draw[blue] (0,0)--(0,-1.1);
		\draw[red] (-.3,-.8)--(0,0)--(.3,-.8);
		\draw[red] (-.3,-.9) node{\tiny$Bca$};
		\draw[red] (.3,-.9) node{\tiny$Bcc$};
		\draw[blue] (0,-1.2) node{\tiny$Bcb$};
	\end{scope}
\end{scope}
\begin{scope}[xshift=1.3cm,yshift=-1.7cm] 
	\draw[red] (0,0)--(0,-1.1);
	\draw[blue](-.3,-.8)--(0,0)--(.3,-.8);
	\draw[blue] (-.3,-.9) node{\tiny$Ca$};
	\draw[blue] (.3,-.9) node{\tiny$Cc$};
	\draw[red] (0,-1.2) node{\tiny$Cb$};
\end{scope}
\end{scope} 

\begin{scope}[xshift=-2cm] 
\draw[thick,blue] (2.7,-1.7)--(4,0)--(5.3,-1.7);
\draw[thick,red] (4,-1.7)--(4,0);
\draw[blue,fill] (4,0) circle[radius=1mm] node[left]{$R$} ;
\draw[blue,fill] (5.3,-1.7) circle[radius=1mm] node[right]{$C$} ;
\draw[blue,fill] (2.7,-1.7) circle[radius=1mm] node[left]{$A$} ;
\draw[red,fill] (2.7,-3) circle[radius=1mm] node[left]{$Ab$} ;
\draw[red,fill] (4,-1.7) circle[radius=1mm] node[left]{$B$} ;
\draw[red,fill] (4.5,-3) circle[radius=1mm] node[left]{$Bc$} ;
\end{scope} 

\end{scope} 

\coordinate (R0) at (0,0);
\coordinate (R1) at (1,0);
\coordinate (R2) at (0,1);
\coordinate (R3) at (1,1);

\coordinate (A1) at (0,2);
\coordinate (Aa1) at (1,2);
\coordinate (Ab1) at (1,3);
\coordinate (Aba1) at (2,3);
\coordinate (Abb1) at (3,3);
\coordinate (Abc1) at (4,3);
\coordinate (Ac1) at (5,3); 
\coordinate (Ac2) at (6,3); 
\draw[thick,blue] (A1)--(Aa1)--(Ab1)--(Aba1)--(Abb1)(Abc1)--(Ac1);
\draw[thick,red] (Abb1)--(Abc1);
\coordinate (B1) at (5,4);
\coordinate (Ba1) at (6,4);
\coordinate (Bb1) at (7,4);
\coordinate (Bc1) at (7,5);
\coordinate (Bca1) at (8,5);
\coordinate (Bcb1) at (9,5);
\coordinate (Bcc1) at (10,5); 
\coordinate (Bcc2) at (11,5);
\draw[thick,blue] (Ac1)--(B1)--(Ba1)--(Bb1)--(Bc1)--(Bca1)--(Bcb1);
\draw[fill] (Bcc1) circle[radius=2pt];

\coordinate (C1) at (10,6);
\coordinate (Ca1) at (11,6);
\coordinate (Cb1) at (12,6);
\coordinate (Cc1) at (13,6);
\draw[thick,blue] (Bcc1)--(C1)--(Ca1) (Cb1)--(Cc1);
\draw[fill] (Cc1) circle[radius=2pt];

\coordinate (A) at (0,1.5);
\coordinate (Ab) at (1,2.5);
\coordinate (Aa) at (.5,2);
\coordinate (Aba) at (1.5,3);
\coordinate (Abb) at (2.5,3);
\coordinate (Abc) at (3.5,3);
\coordinate (Ac) at (4.5,3);

\coordinate (B) at (5,3.5);
\coordinate (Ba) at (5.5,4);
\coordinate (Bb) at (6.5,4);
\coordinate (Bc) at (7,4.5);
\coordinate (Bca) at (7.5,5);
\coordinate (Bcb) at (8.5,5);
\coordinate (Bcc) at (9.5,5);
\coordinate (C) at (10,5.5);
\coordinate (Ca) at (10.5,6);
\coordinate (Cb) at (11.5,6);
\coordinate (Cc) at (12.5,6);
\coordinate (L1) at (4,1.4);
\coordinate (Lm1) at (2.5,1.6);
\coordinate (L0) at (8,4);

\draw (L1) node{$L_1$};
\draw[red] (L0) node{$L_0$};
\draw[blue] (Lm1) node{$L_{-1}$};

\foreach \x in {Aa,Aba,Abb,Ac,Ba,Bb,Bca,Bcb,Ca,Cc}
\draw[blue] (\x) node[below]{\tiny$\x$};
\foreach \x in {Abc,Bcc,Cb}
\draw[red] (\x) node[below]{\tiny$\x$};
\coordinate (R) at (0,.5);
\draw (R0)--(R1)--(Cc1);
\draw[blue] (Ac1)--(R3)--(R2);
\draw[red](Ac1)--(Ac2)--(Bcc1);
\draw[blue] (Bcc1)--(Bcc2)--(Cc1);

\draw[red,thick] (Bcb1)--(Bcc1);
\draw[red,thick] (Ca1)--(Cb1);

\draw[very thick,blue] (R0)--(R2);
\draw[blue,very thick] (R2)--(A1);
\draw[blue,very thick] (Bcc1)--(C1);
\draw[red,very thick] (Aa1)--(Ab1);
\draw[red,very thick] (Ac1)--(B1);
\draw[red,very thick] (Bb1)--(Bc1);
\draw[blue] (R)node[left]{$R$};
\draw[blue] (A)node[left]{$A$};
\draw[blue] (C)node[left]{$C$};
\draw[red] (B)node[left]{$B$};
\draw[red] (Ab)node[left]{$Ab$};
\draw[red] (Bc)node[left]{$Bc$};

\draw (R2) node[left]{$(0,1)$};
\draw (Ac2) node[below]{$(2a+1,a+1)$};
\draw (Bcc2) node[below]{$(2a+2b+2,a+b+1)$};
\foreach \x in {R0,R2,C1,A1,Aba1,Abc1,Ba1,Bca1,Cb1,Bcc1,Cc1}
\draw[fill,blue] (\x) circle[radius=2pt] ;
\foreach \x in {Aa1,Ab1,Abb1,Ac1,B1,Bb1,Bc1,Bcb1,Ca1}
\draw[fill,red] (\x) circle[radius=2pt] ;
\end{tikzpicture}
\caption{Illustrating Theorem \ref{thm: lattice paths and relative injectives}: A lattice path of height $n$ consists of the vertical edge $R$ together with three sublattice paths of heights $a,b,c$ where $a+b+c=n-1$. These lattice subpaths have first edges $A,B,C$. Lattice subpath $B$ is odd since it starts at $(2a+1,a+1)$, but lattice subpaths $A,C$ starting at $(0,1)$, $(2a+2b+2,a+b+1)$ are even. The edge labels are the vertex labels of a uniquely determined ternary tree in lexicographic order after $R$. The nodes of the tree correspond to the vertical edges of the lattice path.}
\label{Fig: lattice path to ternary tree bijection}
\end{center}
\end{figure}
%


\subsection{$\tilde{\mathbb{A}}_n$ Straight Orientation}\label{sec: affine A exceptional collections}

\indent 

Throughout this subsection, let $Q$ represent a quiver of type $\tilde{\mathbb{A}}_{n}$ with straight orientation. We will show that the number of families of exceptional sets over the path algebra of this quiver is also counted by the generalized Catalan numbers, or the Rothe numbers, or the Rothe--Hagan coefficients of the first type [\ref{ref: Gould Generalized Catalan Numbers}] and [\ref{ref: Master's Thesis Generalized Catalan Numbers}]. In 1956, Gould defined the following numbers
$$A_n(a,b) = {a\over a + bn}\binom{a+bn}{n}.$$

\noindent
We call these numbers the \textbf{Rothe numbers, Rothe--Hagan coefficients,} or the \textbf{generalized Catalan numbers}. By taking $a = 1$ and $b = 2$, we get the Catalan numbers. Moreover, we by taking $a = 1$ and $b = k$, we get 

$$C_n^k = {1\over kn+1}\binom{kn+1}{n}.$$

These are called the \textbf{$k$-ary numbers} or the \textbf{$k$-Catalan numbers}, which also generalize the Catalan numbers. It is known that the $k$-ary numbers count certain lattice paths, staircase tilings, and $k$-ary trees [\ref{ref: RK}]. As we will see in the upcoming sections, the Rothe numbers also count lattice paths in $\mathbb{R}^2$ that satisfy certain conditions, and in the case $a = 4$ and $b = 3$, families of exceptional collections. To this end, consider the annulus $A_{Q^{\bm{\varepsilon}}}$ associated to $Q^{\bm{\varepsilon}}$, so there is one marked point on the inner circle and $n$ marked points on the outer circle. Recall the action  of ${2\pi\over n}$ clockwise outer Dehn twists of the annulus on the set of small fundamental arc diagrams over $\Bbbk Q$, analogous to the one in Figure \ref{fig: Example of orbit}. Each outer equivalence class of a small fundamental arc diagram contains precisely $n$ distinct families and hence inner equivalence classes of small fundamental arc diagrams by Theorem \ref{thm: bijection with small diagrams}. We will show there are precisely $A_{n-1}(4,3) = {4\over 3n+1} \binom{3n+1}{n-1}$ outer equivalence classes. To do this, we need the following definition. \\

\begin{defn}\label{defn: Universal Dateline}
The \textbf{universal date line} of the annulus $A_{Q^{\bm{\varepsilon}}}$ associated to $Q^{\bm{\varepsilon}}$ is the line segment connecting the marked point on the inner circle of the annulus to the midpoint of the segment of the outer circle of the annulus connecting the marked points corresponding to $n$ and $1$.
\end{defn}

Notice that the only way an arc connecting the inner and outer circles crosses the universal date line is if its parameter is nonzero. Moreover, note that an arc corresponding to a regular module crosses the universal dateline if and only if its corresponding strand in the universal cover is not contained in $[0,n+1]$. In order to count the outer equivalence classes, we need the following lemma and another Rothe recursion formula. \\

\begin{lem} \label{lem: existence of nice diagram}
There exists a unique small fundamental arc diagram in each outer equivalence class with no universal date line crossings that contains only one preprojective arc.
\end{lem}

\begin{proof}
Let $D$ be a fundamental arc diagram. Then there exists some interval at which no regular arc has support, for if not $D$ would contain a cycle and not be fundamental. Suppose this gap in support is $[i,j]$. Then after some number of ${2\pi\over n}$ clockwise outer Dehn twists, the universal date line would be between the marked points corresponding to $i$ and $j$. By taking a small diagram in this family we have in each outer equivalence class a small fundamental arc diagram that has no universal date line crossings. \\

Now let $D'$ be an element of the outer equivalence class of $D$ with no universal date line crossings. This means that the corresponding strand diagram $\tilde{d'}$ under the bijection given in [\ref{ref: Maresca}] is entirely contained in $[0,n+1]$. Moreover, this strand diagram has an orientation vector given by $\bm{\varepsilon} = (-,+,+,\dots,+,-)$, hence yields an exceptional collection of an $\mathbb{A}_{n+1}$ quiver with straight orientation. Suppose $\tilde{d'}$ has at least two preprojective strands. Let $k$ and $l$ be the lengths of the longest and second longest strands. Then the strands are denoted by $c(n-k,n+1)$ and $c(n-l,n+1)$ respectively. By the dual of Lemma \ref{lem: existence of gap and maximal injective}, there is a unique $i \in [n-k,n-l]$ that is not in the support of any module other than the one corresponding to $c(n-k,n+1)$. Therefore no arc in $D'$ has support in $[i-1,i]$. Since there exists a sequence of ${2\pi\over n}$ clockwise outer Dehn twists placing the universal date line between the marked points $i-1$ and $i$, there exists a small fundamental arc diagram in each outer equivalence class that has no universal date line crossings and only one preprojective arc. \\

It remains to show uniqueness, which we will do in two cases. First suppose there are two non-regular arcs in the arc diagram $D$. Then the gap in the support of the non-regular arcs is necessarily unique, for if not, the exceptional collection would not be complete. Now suppose $D$ contains $r\geq3$ non-regular arcs. Let $c(0,l)$ and $c(n-k,n+1)$ be the longest preinjective and preprojective strands respectively. Since the collection is complete, each vertex from $l$ to $n-k$ must be in the support of some strand, so there is no gap in support from $l$ to $n-k$. The argument in the second paragraph asserts that there is a unique gap in the support of $c(n-k,n+1)$ in which we can place the universal date line and obtain an exceptional collection with only one preprojective and $r-1$ preinjectives. By duality, there is a unique gap in the support of $c(0,l)$ in which we can place the universal date line and obtain an exceptional collection with only one preinjective and $r-1$ preprojectives. Since these two gaps are unique and there is no gap between these two maximal strands, it follows that there is only one way to obtain a small fundamental arc diagram in each outer equivalence class with no universal date line crossings that contains only one preprojective arc.
\end{proof}

The following proposition from [\ref{ref: Rothe Convolution}] gives the recursion formula we need to count the families. Note that the Rothe numbers in general satisfy a recursion like the one in the next proposition.

\begin{prop} \label{prop: Rothe recursion A(4,3)} 
For $k\geq 1$, the recursion $A_0(4,3)=1, A_1(4,3) = 4$ and \\
$$A_n(4,3) = \displaystyle\sum_{k = 0}^{n-1}\displaystyle\sum_{i=0}^k A_{i}(1,3)A_{k-i}(1,3)\displaystyle\sum_{j=0}^{n-k-1}A_j(1,3)A_{n-k-j-1}(1,3)$$
has unique solution $A_{n-1}(4,3) = {4\over 3n+1} \binom{3n+1}{n-1}$.
\end{prop}

\noindent
We are now ready to count the families.

\begin{thm} \label{thm: counting exceptional families}
The number of outer equivalence classes of families of exceptional collections of $\Bbbk Q$-modules is $A_{n-1}(4,3) = {4\over 3n+1} \binom{3n+1}{n-1}$ and thus the number of families of exceptional collections of $\Bbbk Q$-modules is $nA_{n-1}(4,3) = {4n\over 3n+1} \binom{3n+1}{n-1}$.
\end{thm}

\begin{proof}
By Lemma \ref{lem: existence of nice diagram}, each outer equivalence class is uniquely represented by a strand diagram on $n+2$ vertices with only one preprojective. In this strand diagram, there must exist a maximal length preinjective strand $c(0,k+1)$ and a maximal length preprojective strand $c(n-j,n+1)$. Note that $n-j \geq k+1$, for if not the strands would cross. By Lemma \ref{lem: existence of gap and maximal injective}, there exists a unique $i\in [0,k+1]$ such that no other strand has support in $[i-1,i]$. Therefore, in $[0,k+1]$, strands can have support on $[0,i-1]$ or $[i,k+1]$. Note that no strand can have support beginning before $k+1$ and ending after, since it would cross the maximal preinjective strand. Similarly, no strand can have support beginning before $n-j$ and ending after. Thus the next possible interval in which strands can have support is $[k+1,n-j]$. Finally, since there is only one preprojective, no other strand has support in $[n,n+1]$ other than $c(n-j,n+1)$. It follows that the final possible support for strands is $[n-j,n]$. Note that the number of strands in each interval is given by $i-1, k-i+1, n - j - k - 1$, and $j$ respectively. Since these strands correspond to exceptional collections in type $\mathbb{A}$ with straight orientation, Theorem \ref{thm: counting exceptional sets type A}  asserts that the number of such strands is given by 
$$\displaystyle \sum_{\substack{a + b + c + d = n-1\\a+b = k\\c+d=n-k-1\\a,b,c,d\geq0}} E(a)E(b)E(c)E(d) = \displaystyle\sum_{k = 0}^{n-1}\displaystyle\sum_{i=0}^k A_{i}(1,3)A_{k-i}(1,3)\displaystyle\sum_{j=0}^{n-k-1}A_j(1,3)A_{n-k-j-1}(1,3).$$

Since the number of families of exceptional collections of type $\tilde{\mathbb{A}}_i$ for $i=1,2$ are $1$ and $8$ respectively as computed in $[\ref{ref: Maresca}]$, Proposition \ref{prop: Rothe recursion A(4,3)} yields the theorem.
\end{proof}

\subsection{Bijection with Lattice Paths}

\indent 

Let $P_n(c, d)$ be the set of lattice paths from $(0, 0)$ to $(c + nd, n)$ not touching the line $y = {x-c\over d}$ except at the end. It is well known that the generalized Catalan number $A_n(c,d+1) = |P_n(c,d)|$ [\ref{ref: generalized Catalan numbers count paths 1}], [\ref{ref: generalized Catalan numbers count paths 2}], and [\ref{ref: generalized Catalan numbers count paths 3}]. If follows from Theorem \ref{thm: counting exceptional families} that we have a bijection between outer equivalence classes of families of exceptional collections and lattice paths lying above the line $l$ defined by $y = {x-4\over2}$. In this subsection we provide an explicit formula for this bijection between strand diagrams and the aforementioned lattice paths by assigning a label to each diagram.

\begin{defn} \label{defn: label}
Let $\tilde{d}$ be a small fundamental strand diagram as in Lemma \ref{lem: existence of nice diagram} for a quiver of type $\tilde{\mathbb{A}}_n$ with straight orientation. Let $A = c(0,i)$ and $B=c(j,n+1)$ denote the maximal preinjective and unique maximal preprojective respectively. We begin our labels with $A$ and $a, b, c$ underneath and to the right followed by $B$ with $a$ underneath and to the right. There are two possibilities for a strand $X \neq B$ immediately counter clockwise from $A$ and one possibility for $X$ to be immediately clockwise from $A$ given as follows. 

\begin{enumerate}
\item $X=c(i,z)$ where $z<j$.
\item $X=c(0,y)$ where $0<y<i$.
\item $X=c(p,i)$ where $0<p<i$.
\end{enumerate}

In our label under $A$, the three words $a, b,$ and $c$ correspond to cases $1., 2.,$ and $3.$ respectively. If there is a strand in $\tilde{d}$ that falls into one of these three cases, we circle the corresponding word $x$ and introduce three new words $xa, xb,$ and $xc$. The process continues analogously for strands immediately counter clockwise from $X$ that fall in cases $1.$ and $2$. Suppose $X = c(p,i)$ falls in case $3$ and hence is clockwise from $A$. Then there are three possibilities for $Z$, some strand sharing support with $X$, given as follows.

\begin{enumerate}[(i)]
\item $Z=c(k,p)$ where $0<k<p$.
\item $X=c(p,f)$ where $p<f<i$.
\item $X=c(g,i)$ where $p<g<i$.
\end{enumerate}

In our label under the circled $x$, the three words $xa, xb,$ and $xc$ correspond to cases $(i), (ii),$ and $(iii)$ respectively. The labeling continues analogously for all other strands immediately clockwise from this new one. \\

There is only one way a strand can be either clockwise or counter clockwise from $B$. If a strand is clockwise from $B$, it is of the form $Y = c(b,j)$ where $j>b>i$, and has already been labeled under $A$. If a strand is counter clockwise, it is of the form $Y = c(j,q)$ were $q<n+1$. This case corresponds to the label $a$ under $B$. If a strand of this form is in $\tilde{d}$, we circle $a$ under $B$ and proceed as above.\\

The label is complete once there are $n-1$ circled words. 
\end{defn}

\begin{figure}[h]
    \centering

\tikzset{every picture/.style={line width=0.75pt}} 

\begin{tikzpicture}[x=0.75pt,y=0.75pt,yscale=-1,xscale=1]

\draw [color={rgb, 255:red, 0; green, 0; blue, 0 }  ,draw opacity=1 ]   (9.22,177.48) .. controls (17.22,120.48) and (95.22,376.48) .. (218.22,195.01) ;
\draw    (219.22,195.01) .. controls (342.22,350.48) and (448.22,130.48) .. (451.22,178.01) ;
\draw    (104,195) .. controls (105.22,199.01) and (119.22,226.01) .. (159.22,196.01) ;
\draw    (104.22,194.01) .. controls (102.22,231.01) and (165.22,251.01) .. (217.22,195.01) ;
\draw    (221.22,195.01) .. controls (240.22,211.01) and (300.22,238.48) .. (337.22,196.01) ;
\draw    (283.22,195.01) .. controls (283.44,199.01) and (298.22,224.01) .. (338.22,194.01) ;
\draw    (338.22,193.01) .. controls (338.44,197.01) and (353.22,222.01) .. (393.22,192.01) ;
\draw   (472.22,56.49) .. controls (472.28,50.16) and (477.47,45.08) .. (483.8,45.15) .. controls (490.13,45.21) and (495.21,50.4) .. (495.14,56.73) .. controls (495.08,63.06) and (489.89,68.14) .. (483.56,68.07) .. controls (477.23,68.01) and (472.15,61.82) .. (472.22,56.49) -- cycle ;
\draw   (486.22,77.49) .. controls (486.28,71.16) and (491.47,66.08) .. (497.8,66.15) .. controls (504.13,66.21) and (509.21,71.4) .. (509.14,77.73) .. controls (509.08,84.06) and (503.89,89.14) .. (497.56,89.07) .. controls (491.23,89.01) and (486.15,83.82) .. (486.22,77.49) -- cycle ;
\draw   (475.22,253.49) .. controls (475.28,247.16) and (480.47,242.08) .. (486.8,242.15) .. controls (493.13,242.21) and (498.21,247.4) .. (498.14,253.73) .. controls (498.08,260.06) and (492.89,265.14) .. (486.56,265.07) .. controls (480.23,265.01) and (475.15,259.82) .. (475.22,253.49) -- cycle ;
\draw   (489.22,278.49) .. controls (489.28,273.16) and (494.47,267.08) .. (500.8,267.15) .. controls (507.13,267.21) and (512.21,272.4) .. (512.14,278.73) .. controls (512.08,285.06) and (506.89,290.14) .. (500.56,290.07) .. controls (494.23,290.01) and (489.15,284.82) .. (489.22,278.49) -- cycle ;
\draw   (489.22,366.49) .. controls (489.28,360.16) and (494.47,355.08) .. (500.8,355.15) .. controls (507.13,355.21) and (512.21,360.4) .. (512.14,366.73) .. controls (512.08,373.06) and (506.89,378.14) .. (500.56,378.07) .. controls (494.23,378.01) and (489.15,372.82) .. (489.22,366.49) -- cycle ;
\draw    (50.22,195.48) .. controls (51.44,199.49) and (64.22,224.01) .. (104.22,194.01) ;
\draw   (487.22,149.49) .. controls (487.28,143.16) and (492.47,138.08) .. (498.8,138.15) .. controls (505.13,138.21) and (510.21,143.4) .. (510.14,149.73) .. controls (510.08,156.06) and (504.89,161.14) .. (498.56,161.07) .. controls (492.23,161.01) and (487.15,155.82) .. (487.22,149.49) -- cycle ;

\draw (3,177.4) node [anchor=north west][inner sep=0.75pt]    {$0$};
\draw (98,179.4) node [anchor=north west][inner sep=0.75pt]    {$2$};
\draw (155,179.4) node [anchor=north west][inner sep=0.75pt]    {$3$};
\draw (213,179.4) node [anchor=north west][inner sep=0.75pt]    {$4$};
\draw (276,179.4) node [anchor=north west][inner sep=0.75pt]    {$5$};
\draw (141,269) node [anchor=north west][inner sep=0.75pt]   [align=left] {A};
\draw (333,179.4) node [anchor=north west][inner sep=0.75pt]    {$6$};
\draw (389,179.4) node [anchor=north west][inner sep=0.75pt]    {$7$};
\draw (445,179.4) node [anchor=north west][inner sep=0.75pt]    {$8$};
\draw (143,230) node [anchor=north west][inner sep=0.75pt]   [align=left] {c};
\draw (117,188) node [anchor=north west][inner sep=0.75pt]   [align=left] {cb};
\draw (315,263) node [anchor=north west][inner sep=0.75pt]   [align=left] {B};
\draw (277,219) node [anchor=north west][inner sep=0.75pt]   [align=left] {a};
\draw (297,187) node [anchor=north west][inner sep=0.75pt]   [align=left] {ac};
\draw (353,207) node [anchor=north west][inner sep=0.75pt]   [align=left] {aa};
\draw (465,8) node [anchor=north west][inner sep=0.75pt]   [align=left] {A};
\draw (478,20) node [anchor=north west][inner sep=0.75pt]   [align=left] {a};
\draw (479,30) node [anchor=north west][inner sep=0.75pt]   [align=left] {b};
\draw (488,69) node [anchor=north west][inner sep=0.75pt]   [align=left] {ca};
\draw (508,89) node [anchor=north west][inner sep=0.75pt]   [align=left] {caa};
\draw (509,105) node [anchor=north west][inner sep=0.75pt]   [align=left] {cab};
\draw (509,129) node [anchor=north west][inner sep=0.75pt]   [align=left] {cac};
\draw (489,141) node [anchor=north west][inner sep=0.75pt]   [align=left] {cb};
\draw (489,206) node [anchor=north west][inner sep=0.75pt]   [align=left] {cb};
\draw (480,51) node [anchor=north west][inner sep=0.75pt]   [align=left] {c};
\draw (467,232) node [anchor=north west][inner sep=0.75pt]   [align=left] {B};
\draw (482,248) node [anchor=north west][inner sep=0.75pt]   [align=left] {a};
\draw (490,272) node [anchor=north west][inner sep=0.75pt]   [align=left] {aa};
\draw (510,288) node [anchor=north west][inner sep=0.75pt]   [align=left] {aaa};
\draw (510,302) node [anchor=north west][inner sep=0.75pt]   [align=left] {aab};
\draw (510,327) node [anchor=north west][inner sep=0.75pt]   [align=left] {aac};
\draw (491,335) node [anchor=north west][inner sep=0.75pt]   [align=left] {ab};
\draw (492,359) node [anchor=north west][inner sep=0.75pt]   [align=left] {ac};
\draw (45,179.4) node [anchor=north west][inner sep=0.75pt]    {$1$};
\draw (64,188) node [anchor=north west][inner sep=0.75pt]   [align=left] {ca};
\draw (510,155) node [anchor=north west][inner sep=0.75pt]   [align=left] {cba};
\draw (511,175) node [anchor=north west][inner sep=0.75pt]   [align=left] {cbb};
\draw (511,195) node [anchor=north west][inner sep=0.75pt]   [align=left] {cbc};

\end{tikzpicture}

    \caption{A strand diagram for a quiver of type $\tilde{\mathbb{A}}_7$ with straight orientation (left) with its corresponding label (right).}
    \label{fig: Example of Labeling}
\end{figure}
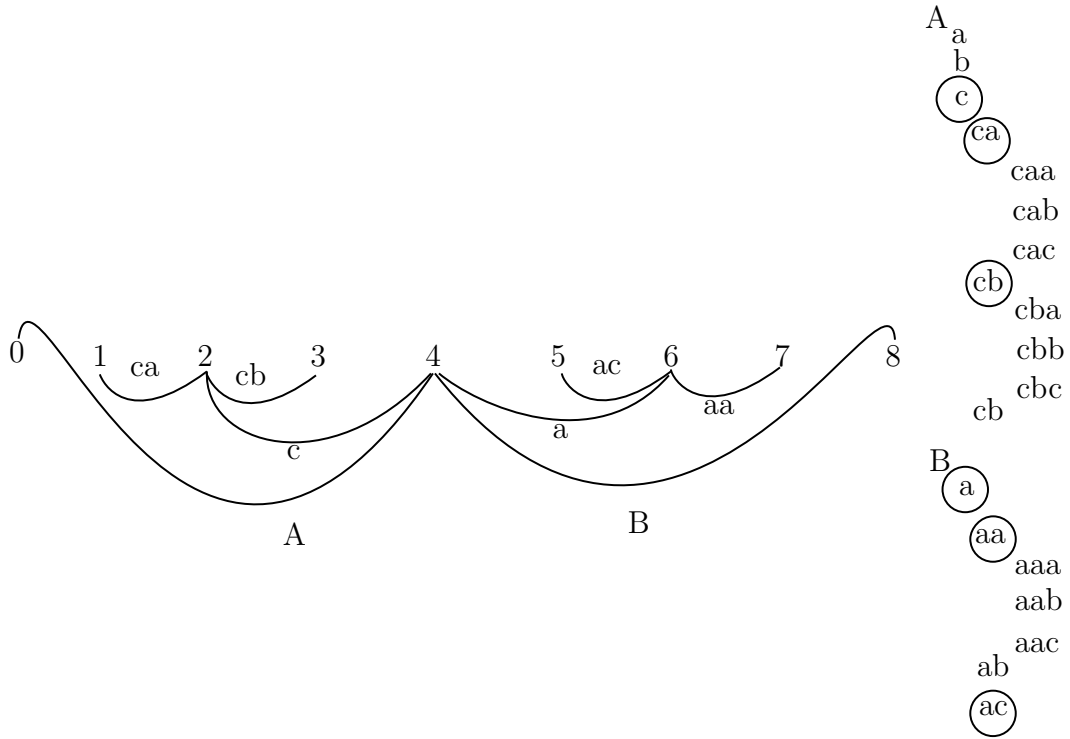

By reading this label from top to bottom, we attain a lattice path by beginning at the origin in $\mathbb{R}^2$, moving one unit horizontally for each un-circled word beginning with a lowercase letter and one unit vertically for each circled word beginning with a lowercase letter. \\

\begin{thm}\label{thm: label bijection}
The label given in Definition \ref{defn: label} provides a bijection from the set of outer equivalence classes of families of exceptional collections of type $\tilde{\mathbb{A}}_n$ with straight orientation to the set of lattice paths $P_{n-1}(4,2)$. 
\end{thm}

\begin{proof}
We first show that the map is well defined by showing each label produces a lattice path that lies above the line $l$. To do this, we note that on the $(i+1)$st step of the path, when $y=i$ there are $3+2i$ possible horizontal movements before hitting the line $l$. In our label, the first circled word must be one of $a, b,$ or $c$ under $A$ or $a$ under $B$. Therefore, the first step of the lattice path lies above the line $l$. We proceed inductively. \\

Suppose we have $i-1$ circled words and $j$ uncircled ones prior to the $(i-1)$st circled one where $j \leq 2(i-2)$. After circling the $(i-1)$st word, there at most two words of the same length left uncircled. After circling the $i$th word, three new words are introduced. It follows that the maximum number of uncircled words after circling the $i$th one is $j + 5 \leq 2(i-2) + 5 = 2(i-1) + 3$. \\

Since $A_{n-1}(4,3) = |P_{n-1}(4,2)|$, it follows that $P_{n-1}(4,2)$ is a finite set. Moreover, the label association in Definition \ref{defn: label} is unique so the map is an injection between two finite sets of the same cardinality by Theorem \ref{thm: counting exceptional families}, hence a bijection. 
\end{proof}

\section{Examples}

In this section, we provide explicit examples for $Q$ a quiver of type $\tilde{\mathbb{A}}_3$ with straight orientation.

\begin{exmp}
The table below depicts all $15 = 3(5)$ small triangulations of $A_{Q^{\varepsilon}}$ and the corresponding cluster underneath. Note that each row represents an outer equivalence class.

\begin{center}
\includegraphics[width = 12 cm, height = 18 cm]{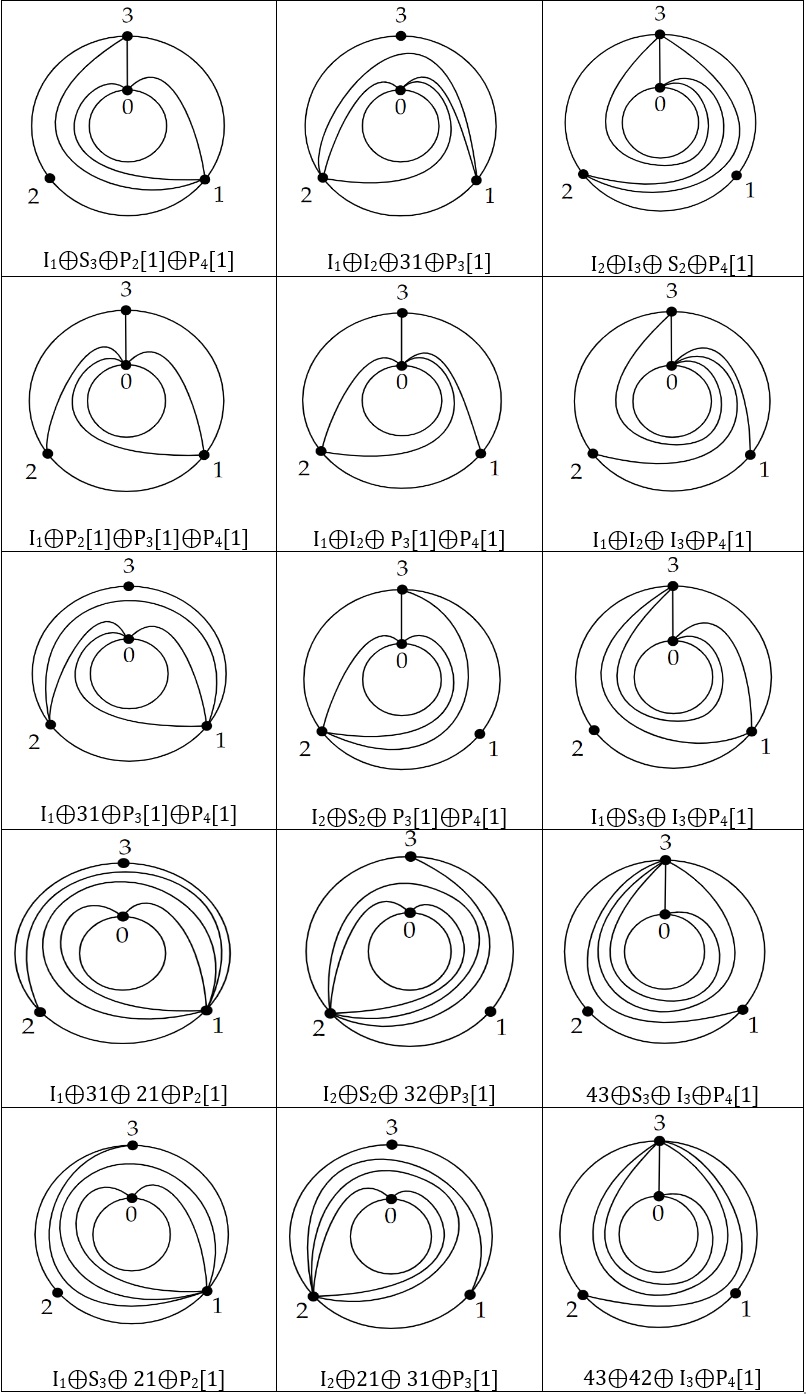}
\end{center}

\end{exmp}
\eject
\begin{exmp}\label{exmp: families of exceptional collections}
The tables below depict the $A_2(4,3) = 18$ unique small fundamental arc diagrams from Lemma \ref{lem: existence of nice diagram} and the corresponding strand diagram underneath in the left column. In the right column, it shows the corresponding lattice path and label. The colors are mainly to more easily see the lifts of the strands; however, the green strand is the maximal projective and the yellow strand is the maximal injective. The other colors do not carry any meaning.
\begin{center}
\includegraphics[width = 16.5 cm, height = 18 cm]{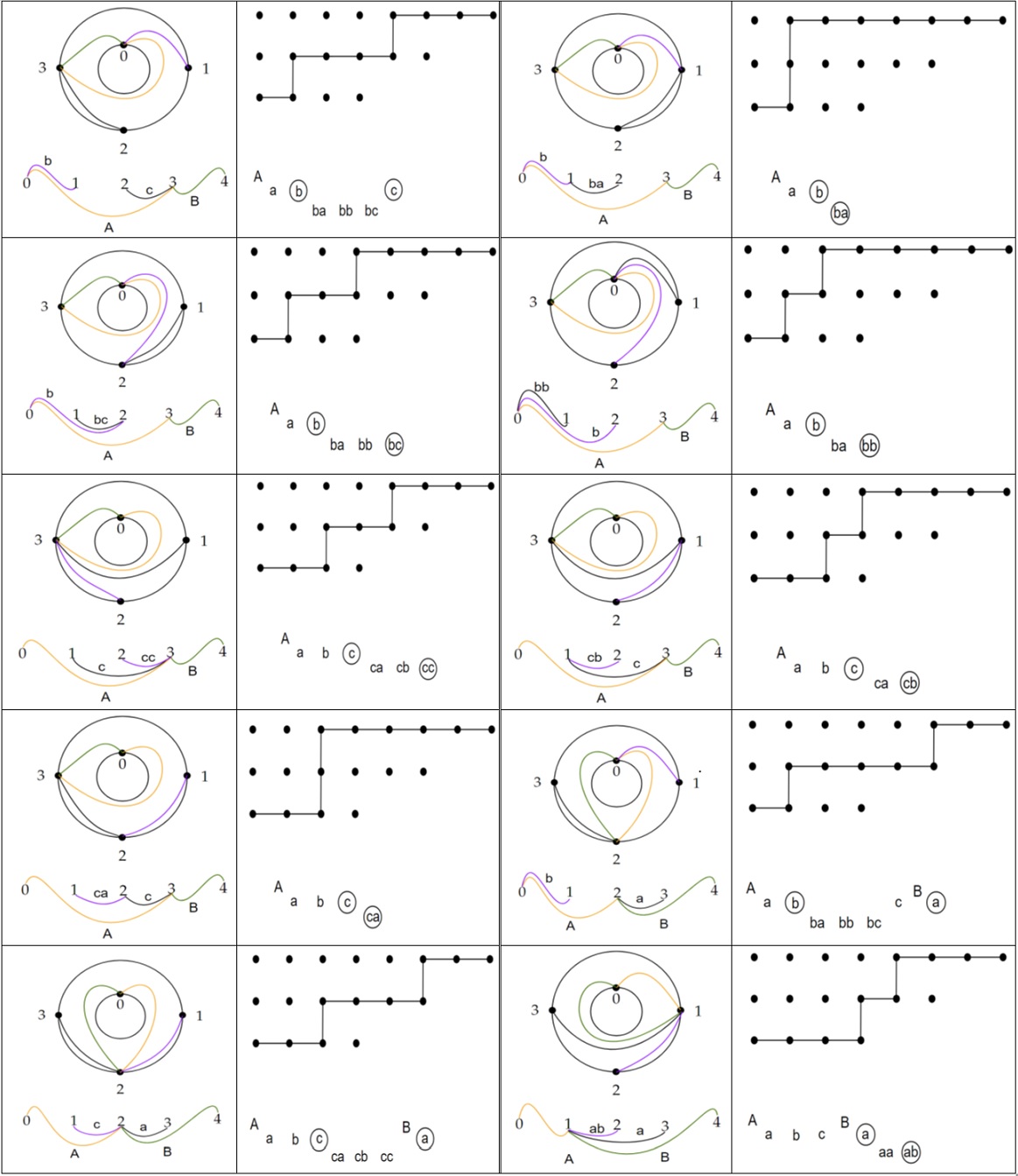}
\end{center}

\begin{center}
\includegraphics[width = 16.5 cm, height = 18 cm]{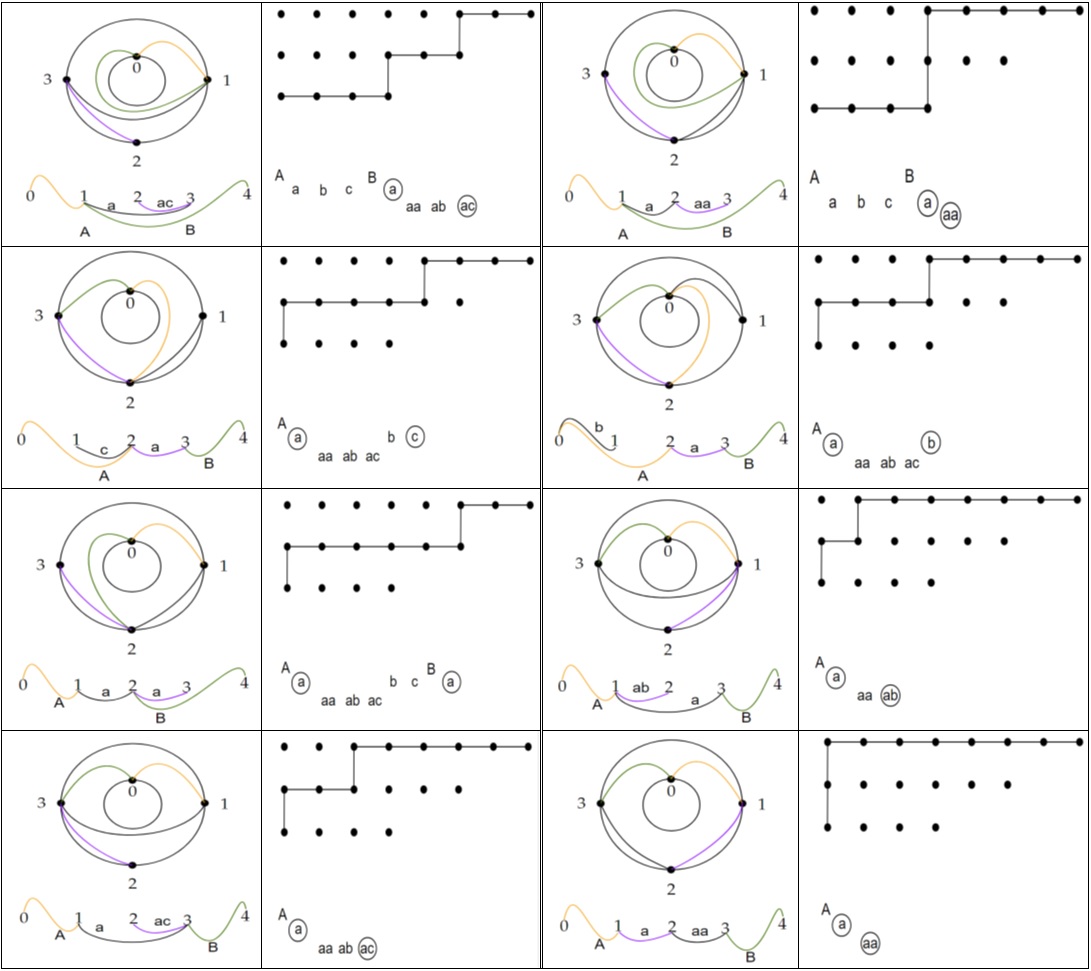}
\end{center}
\end{exmp}

\pagebreak

\section{References}

\begin{enumerate}[ {[}1{]} ]

\item Assem, I., Simson, D., Skowronski, A., \textit{Elements of the representation theory of associative algebras: volume 1: techniques of representation theory.}. Cambridge University Press. 2006. \label{ref: blue book} \\ 

\item Crawley-Boevey, W. \textit{Exceptional sequences of representations of quivers}. [MR1206935 (94c:16017)]. In Representations of algebras (Ottawa, ON, 1992),volume 14 of CMS Conf. Proc., pages 117–124. Amer. Math. Soc., Providence, RI, 1993. \label{ref: Bill Exceptional Sequences} \\

\item Dvoretzky, A., Motzkin, T., \textit{A problem of arrangements}. Duke Mathematical Journal, 14 (2), 305–313, 1947 \label{ref: generalized Catalan numbers count paths 2} \\

\item Fomin, S., Shapiro, M., Thurston, D. \textit{Cluster algebras and triangulated surfaces part I: cluster complexes}. Preprint, arXiv:0608367v3 [math.RA], 2007. \label{ref: clusters are in bijection with triangulations}\\

\item Garver, A., Igusa, K., Matherne, P., J., Ostroff, J. \textit{Combinatorics of exceptional sequences in type A}. The Electronic Journal of Combinatorics. \textbf{26} 2019. \label{ref: combinatorics of exceptional sequences type A}\\ 

\item Grossman, H.D., \textit{Fun with lattice points}. —21, Scripta Math., 16, 1950 \label{ref: generalized Catalan numbers count paths 1} \\

\item Gould H.W., \textit{Some generalizations of Vandermonde’s convolution}. The American Mathematical Monthly, 63 (2), 84–91, 1956 \label{ref: Rothe Convolution} \\

\item Gould, H. W. \textit{Fundamentals of series, eight tables based on seven unpublished
manuscript notebooks}. (1945–1990), edited and compiled by J. Quaintance 2010 \label{ref: Gould Generalized Catalan Numbers} \\

\item Hilton, P., Pedersen, J., \textit{Catalan numbers, their generalization, and their uses}. The Mathematical Intelligencer, VOL 13, NO 2, 1991 \label{ref: Catalan Numbers} \\

\item Igusa, K., Sen, E. \textit{Exceptional sequences and rooted labeled forests}. Preprint, arXiv:2108.11351 [math.RT], 2021. \label{ref: Igusa and Sen Rooted Labeled Forests} \\

\item Igusa, K., Schiffler, R. \textit{Exceptional sequences and clusters}. Jounal of Algebra 323 2010 2183-2202 \label{ref: Clusters form exceptional collections} \\

\item Kahkeshani, R. \textit{A generalization of the Catalan numbers}. Journal of Integer Sequences, Vol. 16 2013. \label{ref: RK}\\

\item Laking, R. \textit{String algebras in representation theory}. [Thesis]. Manchester, UK: The University of Manchester; 2016. \label{ref: String Algebra Info} \\

\item Maresca, R. \textit{Combinatorics of exceptional sequences of type $\tilde{\mathbb{A}}_n$}. Preprint, arXiv:2204.00959 [math.RT], 2022 \label{ref: Maresca} \\  

\item Mohanty, S. G., \textit{Lattice path counting and applications}. Academic Press, New York, 1979 \label{ref: generalized Catalan numbers count paths 3}\\

\item Richardson, S. L. Jr., \textit{Enumeration of the generalized Catalan numbers}. M.Sc. Thesis, Eberly College of Arts and Sciences, West Virginia University, Morgantown, West Virginia 2005 \label{ref: Master's Thesis Generalized Catalan Numbers}  \\ 

\item Ringel, M., C. \textit{The braid group action on the set of exceptional sequences of a hereditary Artin algebra}. Abelian group theory and related topics, Oberwolfach, 1993, volume 171 of Contemp. Math., pages 339–352. Amer. Math. Soc., Providence, RI, 1994. \label{ref: Braid Group Action on Exceptional Sequences}\\

\item Reiten, I. \textit{Cluster categories}. Preprint, arXiv:1012.4949v1 [math.RT], 2010 \label{ref: Reiten Cluster Categories} \\ 

\item Schiffler, R. \textit{Quiver representations}. CMS Books in Mathematics, Springer International Publishing 2014 \label{ref: Schiffler Quiver Reps} \\

\item  Seidel, U. \textit{Exceptional sequences for quivers of Dynkin type}. Communications in Algebra, 29(3),
2001, 1373–1386 \label{ref: Dynkin Exceptional Sequences} \\

\item Warkentin, M. \textit{Fadenmoduln $\ddot{u}$ber $\tilde{A}_n$ und Cluster-Kombinatorik}. Diploma Thesis, University of Bonn. Available from http://nbn-resolving.de/urn:nbn:de:bsz:ch1-qucosa-94793 2008 \label{ref: Master's Thesis Clusters and Triangulations} \\

\end{enumerate}

\end{document}